\documentclass[11pt]{preprint}
\usepackage{difftrees} 
\usepackage[full]{textcomp}
\usepackage[osf]{newtxtext} 
\usepackage[cal=boondoxo]{mathalfa}
\usepackage{colortbl}

\usepackage{tikz-cd} 
\usetikzlibrary{cd} 
\usepackage[all,cmtip]{xy}
\usepackage{comment}

\usepackage{amssymb}
\usepackage{mathtools}
\usepackage{hyperref}
\usepackage{breakurl}
\usepackage{mhenvs}
\usepackage{mhequ} 
\usepackage{mhsymb}
\usepackage{booktabs}
\usepackage{tikz}
\usepackage{tcolorbox}
\usepackage{mathrsfs}
\usepackage[utf8]{inputenc}
\usepackage{longtable}
\usepackage{wrapfig}
\usepackage{rotating} 
\usepackage{subcaption}
\usepackage{mathrsfs}
\usepackage{epsfig}
\usepackage{microtype}
\usepackage{comment}
\usepackage{wasysym}
\usepackage{centernot}
\usepackage{enumitem}
\usepackage{bm}
\usepackage{stackrel}
\usepackage{graphicx}
\usepackage{axodraw}
\usepackage{xspace}
\usepackage{subcaption} 
\usepackage{epsfig} 
\usepackage{axodraw} 
\usepackage{xspace} 
\usepackage[toc,page]{appendix} 
\usepackage[all,cmtip]{xy} 
\usepackage{relsize} 
\usepackage{shuffle} 
\makeatletter
\newcommand{\globalcolor}[1]{%
  \color{#1}\global\let\default@color\current@color
}
\makeatother

\usetikzlibrary{calc}
\usetikzlibrary{decorations}
\usetikzlibrary{positioning}
\usetikzlibrary{shapes}
\usetikzlibrary{external}

\definecolor{blush}{rgb}{0.87, 0.36, 0.51}
	\definecolor{brightcerulean}{rgb}{0.11, 0.67, 0.84}
	\definecolor{greenryb}{rgb}{0.4, 0.69, 0.2}

\newif\ifdark
\darkfalse

\ifdark
\definecolor{darkred}{rgb}{0.9,0.2,0.2}
\definecolor{darkblue}{rgb}{0.7,0.3,1}
\definecolor{darkgreen}{rgb}{0.1,0.9,0.1}
\definecolor{franck}{rgb}{0,0.8,1}
\definecolor{pagebackground}{rgb}{.15,.21,.18}
\definecolor{pageforeground}{rgb}{.84,.84,.85}
\pagecolor{pagebackground}
\AtBeginDocument{\globalcolor{pageforeground}}
\definecolor{symbols}{rgb}{0,0.7,1}
\colorlet{connection}{red!80!black}
\colorlet{boxcolor}{blue!50}

\else

\definecolor{darkred}{rgb}{0.7,0.1,0.1}
\definecolor{darkblue}{rgb}{0.4,0.1,0.8}
\definecolor{darkgreen}{rgb}{0.1,0.7,0.1}
\definecolor{franck}{rgb}{0,0,1}
\definecolor{pagebackground}{rgb}{1,1,1}
\definecolor{pageforeground}{rgb}{0,0,0}
\colorlet{symbols}{blue!90!black}
\colorlet{connection}{red!30!black}
\colorlet{boxcolor}{blue!50!black}

\fi

\def\slash{\leavevmode\unskip\kern0.18em/\penalty\exhyphenpenalty\kern0.18em}
\def\dash{\leavevmode\unskip\kern0.18em--\penalty\exhyphenpenalty\kern0.18em}

\DeclareMathAlphabet{\mathbbm}{U}{bbm}{m}{n}

\DeclareFontFamily{U}{BOONDOX-calo}{\skewchar\font=45 }
\DeclareFontShape{U}{BOONDOX-calo}{m}{n}{
  <-> s*[1.05] BOONDOX-r-calo}{}
\DeclareFontShape{U}{BOONDOX-calo}{b}{n}{
  <-> s*[1.05] BOONDOX-b-calo}{}
\DeclareMathAlphabet{\mcb}{U}{BOONDOX-calo}{m}{n}
\SetMathAlphabet{\mcb}{bold}{U}{BOONDOX-calo}{b}{n}

\setlist{noitemsep,topsep=4pt,leftmargin=1.5em}

\DeclareMathAlphabet{\mathbbm}{U}{bbm}{m}{n}

\DeclareMathAlphabet{\mcb}{U}{BOONDOX-calo}{m}{n}
\SetMathAlphabet{\mcb}{bold}{U}{BOONDOX-calo}{b}{n}
\DeclareFontFamily{U}{mathx}{\hyphenchar\font45}
\DeclareFontShape{U}{mathx}{m}{n}{
      <5> <6> <7> <8> <9> <10>
      <10.95> <12> <14.4> <17.28> <20.74> <24.88>
      mathx10
      }{}
\DeclareSymbolFont{mathx}{U}{mathx}{m}{n}
\DeclareMathSymbol{\bigtimes}{1}{mathx}{"91}

\providecommand{\figures}{false}
{ \ifthenelse{\equal{\figures}{false}} {#1}{\[ {\rm Figure \ missing !} \]} }{}
\def\id{\mathrm{id}}

\def\CH{\mathcal{H}}

\def\CT{\mathcal{T}}

\tikzstyle{tinydots}=[dash pattern=on \pgflinewidth off \pgflinewidth]
\tikzstyle{superdense}=[dash pattern=on 4pt off 1pt]

\newcommand{\mcT}{\mathcal{T}}

\newcommand{\beq}{\begin{equation}}
\newcommand{\eeq}{\end{equation}}

\usepackage{empheq}

\newcommand{\mfe}{\mathfrak{e}}

\newcommand{\mff}{\mathfrak{f}}

\def\${|\!|\!|}

\newcounter{theorem}

\newtheorem{meta}[theorem]{Metatheorem}
\newtheorem{ex}[theorem]{Example}

\newenvironment{DIFnomarkup}{}{} 

\theorembodyfont{\rmfamily}

\newcommand{\rrightarrow}{{\to\hskip -4.9mm\raise 1pt\hbox{$\to$}}}

\newfont{\indic}{bbmss12}

\def\Nabla_#1{\nabla_{\!#1}}

\makeatletter
\pgfdeclareshape{crosscircle}
{
  \inheritsavedanchors[from=circle] 
  \inheritanchorborder[from=circle]
  \inheritanchor[from=circle]{north}
  \inheritanchor[from=circle]{north west}
  \inheritanchor[from=circle]{north east}
  \inheritanchor[from=circle]{center}
  \inheritanchor[from=circle]{west}
  \inheritanchor[from=circle]{east}
  \inheritanchor[from=circle]{mid}
  \inheritanchor[from=circle]{mid west}
  \inheritanchor[from=circle]{mid east}
  \inheritanchor[from=circle]{base}
  \inheritanchor[from=circle]{base west}
  \inheritanchor[from=circle]{base east}
  \inheritanchor[from=circle]{south}
  \inheritanchor[from=circle]{south west}
  \inheritanchor[from=circle]{south east}
  \inheritbackgroundpath[from=circle]
  \foregroundpath{
    \centerpoint%
    \pgf@xc=\pgf@x%
    \pgf@yc=\pgf@y%
    \pgfutil@tempdima=\radius%
    \pgfmathsetlength{\pgf@xb}{\pgfkeysvalueof{/pgf/outer xsep}}%
    \pgfmathsetlength{\pgf@yb}{\pgfkeysvalueof{/pgf/outer ysep}}%
    \ifdim\pgf@xb<\pgf@yb%
      \advance\pgfutil@tempdima by-\pgf@yb%
    \else%
      \advance\pgfutil@tempdima by-\pgf@xb%
    \fi%
    \pgfpathmoveto{\pgfpointadd{\pgfqpoint{\pgf@xc}{\pgf@yc}}{\pgfqpoint{-0.707107\pgfutil@tempdima}{-0.707107\pgfutil@tempdima}}}
    \pgfpathlineto{\pgfpointadd{\pgfqpoint{\pgf@xc}{\pgf@yc}}{\pgfqpoint{0.707107\pgfutil@tempdima}{0.707107\pgfutil@tempdima}}}
    \pgfpathmoveto{\pgfpointadd{\pgfqpoint{\pgf@xc}{\pgf@yc}}{\pgfqpoint{-0.707107\pgfutil@tempdima}{0.707107\pgfutil@tempdima}}}
    \pgfpathlineto{\pgfpointadd{\pgfqpoint{\pgf@xc}{\pgf@yc}}{\pgfqpoint{0.707107\pgfutil@tempdima}{-0.707107\pgfutil@tempdima}}}
  }
}
\makeatother

\def\symbol#1{\textcolor{symbols}{#1}}

\def\decorate#1#2{
        \ifnum#2>0
    		\foreach \count in {1,...,#2}{
	       	let
				\p1 = (sourcenode.center),
                \p2 = (sourcenode.east),
				\n1 = {\x2-\x1},
				\n2 = {1mm},
				\n3 = {(1.3+0.6*(\count-1))*\n1},
				\n4 = {0.7*\n1}
			in 
        		node[rectangle,fill=symbols,rotate=30,inner sep=0pt,minimum width=0.2*\n2,minimum height=\n2] at ($(sourcenode.center) + (\n3,\n4)$) {}
				}
		\fi
        \ifnum#1>0
    		\foreach \count in {1,...,#1}{
	       	let
				\p1 = (sourcenode.center),
                \p2 = (sourcenode.east),
				\n1 = {\x2-\x1},
				\n2 = {1mm},
				\n3 = {(1.3+0.6*(\count-1))*\n1},
				\n4 = {0.7*\n1}
			in 
        		node[rectangle,fill=symbols,rotate=-30,inner sep=0pt,minimum width=0.2*\n2,minimum height=\n2] at ($(sourcenode.center) + (-\n3,\n4)$) {}
				}
		\fi
}

\tikzset{
    dectriangle/.style 2 args={
        triangle,
        alias=sourcenode,
        append after command={\decorate{#1}{#2}}
    },
    dectriangle/.default={0}{0},
}

\tikzset{
	cross/.style={path picture={ 
  		\draw[symbols]
			(path picture bounding box.south east) -- (path picture bounding box.north west) (path picture bounding box.south west) -- (path picture bounding box.north east);
		}},
root/.style={circle,fill=green!50!black,inner sep=0pt, minimum size=1.2mm},
        dot/.style={circle,fill=pageforeground,inner sep=0pt, minimum size=1mm},
        dotred/.style={circle,fill=pageforeground!50!pagebackground,inner sep=0pt, minimum size=2mm},
        var/.style={circle,fill=pageforeground!10!pagebackground,draw=pageforeground,inner sep=0pt, minimum size=3mm},
        var2/.style={circle,fill=darkgreen,draw=pageforeground,inner sep=0pt, minimum size=3mm},
        kernel/.style={semithick,draw=green,shorten >=2pt,shorten <=2pt},
        kernels/.style={snake=zigzag,shorten >=2pt,shorten <=2pt,segment amplitude=1pt,segment length=4pt,line before snake=2pt,line after snake=5pt,},
        rho/.style={densely dashed,semithick,shorten >=2pt,shorten <=2pt},
           testfcn/.style={dotted,semithick,shorten >=2pt,shorten <=2pt},
        renorm/.style={shape=circle,fill=pagebackground,inner sep=1pt},
        labl/.style={shape=rectangle,fill=pagebackground,inner sep=1pt},
        xic/.style={very thin,circle,draw=symbols,fill=symbols,inner sep=0pt,minimum size=1.2mm},
        g/.style={very thin,rectangle,draw=symbols,fill=symbols!10!pagebackground,inner sep=0pt,minimum width=2.5mm,minimum height=1.2mm},
        xi/.style={very thin,circle,draw=symbols,fill=symbols!10!pagebackground,inner sep=0pt,minimum size=1.2mm},
	xies/.style={very thin,rectangle,fill=green!50!black!25,draw=symbols,inner sep=0pt,minimum size=1.1mm},
	xiesf/.style={very thin,rectangle,fill=green!50!black,draw=symbols,inner sep=0pt,minimum size=1.1mm},
        xix/.style={very thin,crosscircle,fill=symbols!10!pagebackground,draw=symbols,inner sep=0pt,minimum size=1.2mm},
        X/.style={very thin,cross,rectangle,fill=pagebackground,draw=symbols,inner sep=0pt,minimum size=1.2mm},
	xib/.style={thin,circle,fill=symbols!10!pagebackground,draw=symbols,inner sep=0pt,minimum size=1.6mm},
	xie/.style={thin,circle,fill=green!50!black,draw=symbols,inner sep=0pt,minimum size=1.6mm},
	xid/.style={thin,circle,fill=symbols,draw=symbols,inner sep=0pt,minimum size=1.6mm},
	xibx/.style={thin,crosscircle,fill=symbols!10!pagebackground,draw=symbols,inner sep=0pt,minimum size=1.6mm},
	kernels2/.style={very thick,draw=connection,segment length=12pt},
	keps/.style={thin,draw=symbols,->},
	kepspr/.style={thick,draw=connection,->},
	krho/.style={thin,draw=symbols,superdense,->},
	krhopr/.style={thick,draw=connection,superdense},
	triangle/.style = { regular polygon, regular polygon sides=3},
	not/.style={thin,circle,draw=connection,fill=connection,inner sep=0pt,minimum size=0.5mm},
	diff/.style = {very thin,draw=symbols,triangle,fill=red!50!black,inner sep=0pt,minimum size=1.6mm},
	diff1/.style = {very thin,dectriangle={1}{0},fill=red!50!black,draw=symbols,inner sep=0pt,minimum size=1.6mm},
	diff2/.style = {very thin,dectriangle={1}{1},fill=red!50!black,draw=symbols,inner sep=0pt,minimum size=1.6mm},
		diffmini/.style = {very thin,rectangle,fill=black,draw=black,inner sep=0pt,minimum size=0.75mm},
	 kernelsmod/.style={very thick,draw=connection,segment length=12pt},
	 rec/.style = {very thin,rectangle,fill=black,draw=black,inner sep=0pt,minimum size=2mm},
	cerc/.style={very thin,circle,draw=black,fill=symbols,inner sep=0pt,minimum size=2mm},
	stars/.style={very thin,star,star points=6,star point ratio=0.5, draw=black,fill=red,inner sep=0pt,minimum size=0.7mm},
	>=stealth,
        }
        \tikzset{
root/.style={circle,fill=black!50,inner sep=0pt, minimum size=3mm},
        circ/.style={circle,fill=white,draw=black,very thin,inner sep=.5pt, minimum size=1.2mm},
        round1/.style={fill=white,outer sep = 0,inner sep=2pt,rounded corners=1mm,draw,text=black,thin,minimum size=1.2mm},
          circ1/.style={circle,fill=red!10,draw=red,very thin,inner sep=.5pt, minimum size=1.2mm},
        rect/.style={fill=white,outer sep = 0,inner sep=2pt,rectangle,draw,text=black,thin,minimum size=1.2mm},
        rect1/.style={fill=white,outer sep = 0,inner sep=2pt,rectangle,draw,text=black,thin,minimum size=1.2mm},
        round2/.style={fill=red!10,outer sep = 0,inner sep=2pt,rounded corners=1mm,draw,text=black,thin,minimum size=1.2mm},
       round3/.style={fill=blue!10,outer sep = 0,inner sep=2pt,rounded corners=1mm,draw,text=black,thin,minimum size=1.2mm}, 
        rect2/.style={fill=black!10,outer sep = 0,inner sep=2pt,rectangle,draw,text=black,thin,minimum size=1.2mm},
        dot/.style={circle,fill=black,inner sep=0pt, minimum size=1.2mm},
        dotred/.style={circle,fill=black!50,inner sep=0pt, minimum size=2mm},
        var/.style={circle,fill=black!10,draw=black,inner sep=0pt, minimum size=3mm},
        kernel/.style={semithick,draw=darkgreen},
         diag/.style={thin,shorten >=4pt,shorten <=4pt},
        kernel1/.style={thick},
        kernels/.style={snake=zigzag,shorten >=2pt,shorten <=2pt,segment amplitude=1pt,segment length=4pt,line before snake=2pt,line after snake=5pt},
		kernels1/.style={snake=zigzag,segment amplitude=0.5pt,segment length=2pt},
		rho1/.style={densely dotted,semithick},
        rho/.style={densely dashed,semithick,shorten >=2pt,shorten <=2pt},
           testfcn/.style={dotted,semithick,shorten >=2pt,shorten <=2pt},
           visible/.style={draw, circle, fill, inner sep=0.25ex},
        renorm/.style={shape=circle,fill=white,inner sep=1pt},
        labl/.style={shape=rectangle,fill=white,inner sep=1pt},
        xic/.style={very thin,circle,fill=symbols,draw=black,inner sep=0pt,minimum size=1.2mm},
        xi/.style={very thin,circle,fill=blue!10,draw=black,inner sep=0pt,minimum size=1.2mm},
	xib/.style={very thin,circle,fill=blue!10,draw=black,inner sep=0pt,minimum size=1.6mm},
	xie/.style={very thin,circle,fill=green!50!black,draw=black,inner sep=0pt,minimum size=1mm},
	xid/.style={very thin,circle,fill=symbols,draw=black,inner sep=0pt,minimum size=1.6mm},
	edgetype/.style={very thin,circle,draw=black,inner sep=0pt,minimum size=5mm},
	nodetype/.style={very thick,circle,draw=black,inner sep=0pt,minimum size=5mm},
	kernels2/.style={very thick,draw=connection,segment length=12pt},
clean/.style={thin,circle,fill=black,inner sep=0pt,minimum size=1mm},	not/.style={thin,circle,fill=symbols,draw=connection,fill=connection,inner sep=0pt,minimum size=0.8mm},
	>=stealth,
        }

\makeatletter
\def\DeclareSymbol#1#2#3{%
	\expandafter\gdef\csname MH@symb@#1\endcsname{\tikzsetnextfilename{symbol#1}%
	\tikz[baseline=#2,scale=0.15,draw=symbols,line join=round]{#3}}%
	\expandafter\gdef\csname MH@symb@#1s\endcsname{\scalebox{0.75}{\tikzsetnextfilename{symbol#1}%
	\tikz[baseline=#2,scale=0.15,draw=symbols,line join=round]{#3}}}%
	\expandafter\gdef\csname MH@symb@#1ss\endcsname{\scalebox{0.65}{\tikzsetnextfilename{symbol#1}%
	\tikz[baseline=#2,scale=0.15,draw=symbols,line join=round]{#3}}}%
	}
\def\<#1>{\ifthenelse{\boolean{mmode}}{\mathchoice{\csname MH@symb@#1\endcsname}{\csname MH@symb@#1\endcsname}{\csname MH@symb@#1s\endcsname}{\csname MH@symb@#1ss\endcsname}}{\csname MH@symb@#1\endcsname}}
\makeatother

\DeclareSymbol{Xi22}{0.5}{\draw (0,0) node[xi] {} -- (-1,1) node[not] {} -- (0,2) node[xi] {};} 
\DeclareSymbol{Xi2}{-2}{\draw (-1,-0.25) node[xi] {} -- (0,1) node[xi] {};} 
\DeclareSymbol{Xi2b}{-2}{\draw (-1,-0.25) node[xic] {} -- (0,1) node[xic] {};} 
\DeclareSymbol{Xi2g}{-2}{\draw (-1,-0.25) node[xies] {} -- (0,1) node[xi] {};} 
\DeclareSymbol{Xi2g2}{-2}{\draw (-1,-0.25) node[xi] {} -- (0,1) node[xies] {};} 
\DeclareSymbol{cXi2}{-2}{\draw (0,-0.25) node[xi] {} -- (-1,1) node[xic] {};}
\DeclareSymbol{Xi3}{0}{\draw (0,0) node[xi] {} -- (-1,1) node[xi] {} -- (0,2) node[xi] {};}
\DeclareSymbol{XiIIXi}{0}{\draw (0,0) node[xi] {} -- (-1,1); \draw[kernels2] (-1,1) node[not] {} -- (0,2) node[xi] {};}

\DeclareSymbol{Xi4}{2}{\draw (0,0) node[xi] {} -- (-1,1) node[xi] {} -- (0,2) node[xi] {} -- (-1,3) node[xi] {};}
\DeclareSymbol{Xi4_1}{2}{\draw (0,0) node[xic] {} -- (-1,1) node[xic] {} -- (0,2) node[xi] {} -- (-1,3) node[xi] {};}
\DeclareSymbol{Xi4_2}{2}{\draw (0,0) node[xic] {} -- (-1,1) node[xi] {} -- (0,2) node[xi] {} -- (-1,3) node[xic] {};}
\DeclareSymbol{Xi2X}{-2}{\draw (0,-0.25) node[xi] {} -- (-1,1) node[xix] {};}
\DeclareSymbol{XXi2}{-2}{\draw (0,-0.25) node[xix] {} -- (-1,1) node[xi] {};}
\DeclareSymbol{IIXi}{0}{\draw (0,-0.25) node[not] {} -- (-1,1) node[xi] {} -- (0,2) node[xi] {};}
\DeclareSymbol{IXi^2}{-1}{\draw (-1,1) node[xi] {} -- (0,0) node[not] {} -- (1,1) node[xi] {};}
\DeclareSymbol{IIXi^2}{-4}{\draw (0,-1.5) node[not] {} -- (0,0);
\draw[kernels2] (-1,1) node[xi] {} -- (0,0) node[not] {} -- (1,1) node[xi] {};}
\DeclareSymbol{XiX}{-2.8}{\node[xibx] {};}
\DeclareSymbol{tauX}{-2.8}{ \node[X] {};}
\DeclareSymbol{Xi}{-2.8}{\node[xib] {};}

\DeclareSymbol{IXiX}{-1}{\draw (0,-0.25) node[not] {} -- (-1,1) node[xix] {};}
\DeclareSymbol{IXi3}{2}{\draw (0,-0.25) node[not] {} -- (-1,1) node[xi] {} -- (0,2) node[xi] {} -- (-1,3) node[xi] {};}
\DeclareSymbol{IXi}{-2}{\draw (0,-0.25) node[not] {} -- (-1,1) node[xi] {};}
\DeclareSymbol{XiI}{-2}{\draw (0,-0.25) node[xi] {} -- (-1,1) node[not] {};}

\DeclareSymbol{Xi4b}{0}{\draw(0,1.5) node[xi] {} -- (0,0); \draw (-1,1) node[xi] {} -- (0,0) node[xi] {} -- (1,1) node[xi] {};}
\DeclareSymbol{Xi4b'}{0}{\draw(0,1.5) node[xi] {} -- (0,-0.2); \draw (-1,1) node[xi] {} -- (0,-0.2) node[not] {} -- (1,1) node[xi] {};}
\DeclareSymbol{Xi4c}{0}{\draw (0,1) -- (0.8,2.2) node[xi] {};\draw (0,-0.25) node[xi] {} -- (0,1) node[xi] {} -- (-0.8,2.2) node[xi] {};}
\DeclareSymbol{Xi4d}{-4.5}{\draw (0,-1.5) node[not] {} -- (0,0); \draw (-1,1) node[xi] {} -- (0,0) node[xi] {} -- (1,1) node[xi] {};}
\DeclareSymbol{Xi4e}{0}{\draw (0,2) node[xi] {} -- (-1,1) node[xi] {} -- (0,0) node[xi] {} -- (1,1) node[xi] {};}
\DeclareSymbol{Xi4e'}{0}{\draw (0,2) node[xi] {} -- (-1,1) node[xi] {} -- (0,-0.2) node[not] {} -- (1,1) node[xi] {};}

\DeclareSymbol{Xitwo}
{0}{\draw[kernels2] (0,0) node[not] {} -- (-1,1) node[not] {}
-- (-2,2) node[not]{} -- (-3,3) node[xi]  {};
\draw[kernels2] (0,0) -- (1,1) node[xi] {};
\draw[kernels2] (-1,1) -- (0,2) node[xi] {};
\draw[kernels2] (-2,2) -- (-1,3) node[xi] {};}

\DeclareSymbol{IXitwo}
{0}{\draw (-.7,1.2) node[xi] {} -- (0,-0.2) -- (.7,1.2) node[xi] {};}
\DeclareSymbol{I1Xitwo}
{0}{\draw[kernels2] (0,0) node[not] {} -- (-1,1) node[xi] {};
\draw[kernels2] (0,0) -- (1,1) node[xi] {};}

\DeclareSymbol{I1Xitwobis}
{0}{\draw[kernels2] (0,0) node[not] {} -- (-1,1) node[xies] {};
\draw[kernels2] (0,0) -- (1,1) node[xies] {};}

\DeclareSymbol{I1Xitwog}
{0}{\draw[kernels2] (0,0) node[not] {} -- (-1,1) node[xies] {};
\draw[kernels2] (0,0) -- (1,1) node[xi] {};}

\DeclareSymbol{cI1Xitwo}
{0}{\draw[kernels2] (0,0) node[not] {} -- (-1,1) node[xic] {};
\draw[kernels2] (0,0) -- (1,1) node[xi] {};}

\DeclareSymbol{I1IXi3}{0}{\draw (0,0) node[xi] {} -- (-1,1) ; 
\draw[kernels2] (-1,1) node[not] {} -- (0,2) node[xi] {};
\draw[kernels2] (-1,1) node[not] {} -- (-2,2) node[xi] {};}

\DeclareSymbol{I1Xi3c}{-1}{\draw[kernels2](0,1.5) node[xi] {} -- (0,0) node[not] {}; \draw (-1,1) node[xi] {} -- (0,0) ; \draw[kernels2] (0,0) -- (1,1) node[xi] {};}

\DeclareSymbol{I1Xi3cbis}{-1}{\draw[kernels2](0,1.5) node[xies] {} -- (0,0) node[not] {}; \draw (-1,1) node[xies] {} -- (0,0) ; \draw[kernels2] (0,0) -- (1,1) node[xies] {};}

\DeclareSymbol{I1IXi3b}{0}{\draw[kernels2] (0,0) node[not] {} -- (-1,1) ; \draw[kernels2] (0,0)   -- (1,1) node[xi] {} ;
\draw (-1,1) node[xi] {} -- (0,2) node[xi] {};
}

\DeclareSymbol{I1IXi3c}{0}{\draw[kernels2] (0,0) node[not] {} -- (-1,1) ; \draw[kernels2] (0,0)   -- (1,1) node[xi] {} ;
\draw[kernels2] (-1,1) node[not] {} -- (0,2) node[xi] {};
\draw[kernels2] (-1,1) node[not] {} -- (-2,2) node[xi] {};}

\DeclareSymbol{I1IXi3cbis}{0}{\draw[kernels2] (0,0) node[not] {} -- (-1,1) ; \draw[kernels2] (0,0)   -- (1,1) node[xies] {} ;
\draw[kernels2] (-1,1) node[not] {} -- (0,2) node[xies] {};
\draw[kernels2] (-1,1) node[not] {} -- (-2,2) node[xies] {};}

\DeclareSymbol{I1Xi}{0}{\draw[kernels2] (0,0) node[not] {} -- (-1,1)  node[xi] {} ;}

\DeclareSymbol{I1Xi4a}{2}{\draw[kernels2] (0,0) node[not] {} -- (-1,1) ; \draw[kernels2] (0,0) node[not] {} -- (1,1) node[xi] {} ;
\draw (-1,1) node[xi] {} -- (0,2) node[xi] {} -- (-1,3) node[xi] {};}

\DeclareSymbol{cI1Xi4a}{2}{\draw[kernels2] (0,0) node[not] {} -- (-1,1) ; \draw[kernels2] (0,0) node[not] {} -- (1,1) node[xic] {} ;
\draw (-1,1) node[xic] {} -- (0,2) node[xi] {} -- (-1,3) node[xi] {};}

\DeclareSymbol{I1Xi4b}{2}{\draw (0,0) node[xi] {} -- (-1,1) node[xi] {} -- (0,2) ; \draw[kernels2] (0,2) node[not] {} -- (-1,3) node[xi] {};\draw[kernels2] (0,2)  -- (1,3) node[xi] {};
}

\DeclareSymbol{cI1Xi4b}{2}{\draw (0,0) node[xic] {} -- (-1,1) node[xic] {} -- (0,2) ; \draw[kernels2] (0,2) node[not] {} -- (-1,3) node[xi] {};\draw[kernels2] (0,2)  -- (1,3) node[xi] {};
}

\DeclareSymbol{I1Xi4c}{2}{\draw (0,0) node[xi] {} -- (-1,1) node[not] {}; \draw[kernels2] (-1,1) -- (0,2) ; 
\draw[kernels2] (-1,1) -- (-2,2) node[xi] {} ;
\draw (0,2) node[xi] {} -- (-1,3) node[xi] {};}

\DeclareSymbol{cI1Xi4c}{2}{\draw (0,0) node[xic] {} -- (-1,1) node[not] {}; \draw[kernels2] (-1,1) -- (0,2) ; 
\draw[kernels2] (-1,1) -- (-2,2) node[xic] {} ;
\draw (0,2) node[xi] {} -- (-1,3) node[xi] {};}

\DeclareSymbol{I1Xi4ab}{2}{\draw[kernels2] (0,0) node[not] {} -- (-1,1) ; \draw[kernels2] (0,0) node[not] {} -- (1,1) node[xi] {};\draw (-1,1) node[xi] {} -- (0,2) ; \draw[kernels2] (0,2) node[not] {} -- (-1,3) node[xi] {};\draw[kernels2] (0,2)  -- (1,3) node[xi] {}; }

\DeclareSymbol{cI1Xi4ab}{2}{\draw[kernels2] (0,0) node[not] {} -- (-1,1) ; \draw[kernels2] (0,0) node[not] {} -- (1,1) node[xic] {};\draw (-1,1) node[xic] {} -- (0,2) ; \draw[kernels2] (0,2) node[not] {} -- (-1,3) node[xi] {};\draw[kernels2] (0,2)  -- (1,3) node[xi] {}; }

\DeclareSymbol{I1Xi4bc}{2}{\draw (0,0) node[xi] {} -- (-1,1) node[not] {}; \draw[kernels2] (-1,1) -- (0,2) ; 
\draw[kernels2] (-1,1) -- (-2,2) node[xi] {} ; \draw[kernels2] (0,2) node[not] {} -- (-1,3) node[xi] {};\draw[kernels2] (0,2)  -- (1,3) node[xi] {};
}

\DeclareSymbol{cI1Xi4bc}{2}{\draw (0,0) node[xic] {} -- (-1,1) node[not] {}; \draw[kernels2] (-1,1) -- (0,2) ; 
\draw[kernels2] (-1,1) -- (-2,2) node[xic] {} ; \draw[kernels2] (0,2) node[not] {} -- (-1,3) node[xi] {};\draw[kernels2] (0,2)  -- (1,3) node[xi] {};
}

\DeclareSymbol{I1Xi4abcc1}{2}{\draw[kernels2] (0,0) node[not] {} -- (-1,1) node[not] {}
-- (-2,2) node[not]{} -- (-3,3) node[xic]  {};
\draw[kernels2] (0,0) -- (1,1) node[xic] {};
\draw[kernels2] (-1,1) -- (0,2) node[xi] {};
\draw[kernels2] (-2,2) -- (-1,3) node[xi] {};
}

\DeclareSymbol{I1Xi4abcc1b}{2}{\draw[kernels2] (0,0) node[not] {} -- (-1,1) node[not] {}
-- (-2,2) node[not]{} -- (-3,3) node[xi]  {};
\draw[kernels2] (0,0) -- (1,1) node[xic] {};
\draw[kernels2] (-1,1) -- (0,2) node[xic] {};
\draw[kernels2] (-2,2) -- (-1,3) node[xi] {};
}

\DeclareSymbol{I1Xi4abcc2}{2}{\draw[kernels2] (0,0) node[not] {} -- (-1,1) node[not] {}
-- (-2,2) node[not]{} -- (-3,3) node[xic]  {};
\draw[kernels2] (0,0) -- (1,1) node[xi] {};
\draw[kernels2] (-1,1) -- (0,2) node[xi] {};
\draw[kernels2] (-2,2) -- (-1,3) node[xic] {};
}

\DeclareSymbol{I1Xi4ac}{2}{\draw[kernels2] (0,0) node[not] {} -- (-1,1) ; \draw[kernels2] (0,0) node[not] {} -- (1,1) node[xi] {}; 
\draw[kernels2] (-1,1) node[not] {} -- (0,2) ; 
\draw[kernels2] (-1,1) -- (-2,2) node[xi] {} ;
\draw (0,2) node[xi] {} -- (-1,3) node[xi] {};}

\DeclareSymbol{cI1Xi4ac}{2}{\draw[kernels2] (0,0) node[not] {} -- (-1,1) ; \draw[kernels2] (0,0) node[not] {} -- (1,1) node[xic] {}; 
\draw[kernels2] (-1,1) node[not] {} -- (0,2) ; 
\draw[kernels2] (-1,1) -- (-2,2) node[xic] {} ;
\draw (0,2) node[xi] {} -- (-1,3) node[xi] {};}

\DeclareSymbol{I1Xi4acc1}{2}{\draw[kernels2] (0,0) node[not] {} -- (-1,1) ; \draw[kernels2] (0,0) node[not] {} -- (1,1) node[xic] {}; 
\draw[kernels2] (-1,1) node[not] {} -- (0,2) ; 
\draw[kernels2] (-1,1) -- (-2,2) node[xi] {} ;
\draw (0,2) node[xic] {} -- (-1,3) node[xi] {};}

\DeclareSymbol{I1Xi4acc2}{2}{\draw[kernels2] (0,0) node[not] {} -- (-1,1) ; \draw[kernels2] (0,0) node[not] {} -- (1,1) node[xic] {}; 
\draw[kernels2] (-1,1) node[not] {} -- (0,2) ; 
\draw[kernels2] (-1,1) -- (-2,2) node[xi] {} ;
\draw (0,2) node[xi] {} -- (-1,3) node[xic] {};}

\DeclareSymbol{2I1Xi4}{2}{\draw[kernels2] (0,0) node[not] {} -- (-1,1) node[not] {};
\draw[kernels2] (0,0) -- (1,1) node[not] {};
\draw[kernels2] (-1,1) -- (-1.5,2.5) node[xi] {};
\draw[kernels2] (-1,1) -- (-0.5,2.5) node[xi] {};
\draw[kernels2] (1,1) -- (0.5,2.5) node[xi] {};
\draw[kernels2] (1,1) -- (1.5,2.5) node[xi] {};
}

\DeclareSymbol{2I1Xi4dis}{2}{\draw[kernels2] (0,0) node[not] {} -- (-1,1) node[not] {};
\draw[kernels2] (0,0) -- (1,1) node[not] {};
\draw[kernels2] (-1,1) -- (-1.5,2.5) node[xies] {};
\draw[kernels2] (-1,1) -- (-0.5,2.5) node[xies] {};
\draw[kernels2] (1,1) -- (0.5,2.5) node[xies] {};
\draw[kernels2] (1,1) -- (1.5,2.5) node[xies] {};
}

\DeclareSymbol{2I1Xi4c1}{2}{\draw[kernels2] (0,0) node[not] {} -- (-1,1) node[not] {};
\draw[kernels2] (0,0) -- (1,1) node[not] {};
\draw[kernels2] (-1,1) -- (-1.5,2.5) node[xic] {};
\draw[kernels2] (-1,1) -- (-0.5,2.5) node[xi] {};
\draw[kernels2] (1,1) -- (0.5,2.5) node[xic] {};
\draw[kernels2] (1,1) -- (1.5,2.5) node[xi] {};
}

\DeclareSymbol{2I1Xi4c2}{2}{\draw[kernels2] (0,0) node[not] {} -- (-1,1) node[not] {};
\draw[kernels2] (0,0) -- (1,1) node[not] {};
\draw[kernels2] (-1,1) -- (-1.5,2.5) node[xic] {};
\draw[kernels2] (-1,1) -- (-0.5,2.5) node[xic] {};
\draw[kernels2] (1,1) -- (0.5,2.5) node[xi] {};
\draw[kernels2] (1,1) -- (1.5,2.5) node[xi] {};
}

\DeclareSymbol{2I1Xi4b}{2}{\draw[kernels2] (0,0) node[not] {} -- (-1,1) ;
\draw[kernels2] (0,0) -- (1,1);
\draw (-1,1) node[xi] {} -- (-1,2.5) node[xi] {};
\draw (1,1)  node[xi] {} -- (1,2.5) node[xi] {};
}

\DeclareSymbol{2I1Xi4bb}{2}{\draw[kernels2] (0,0) node[not] {} -- (-1,1) ;
\draw[kernels2] (0,0) -- (1,1);
\draw (-1,1) node[xi] {} -- (-1,2.5) node[xiesf] {};
\draw (1,1)  node[xi] {} -- (1,2.5) node[xic] {};
}

\DeclareSymbol{2I1Xi4c}{2}{\draw[kernels2] (0,0) node[not] {} -- (-1,1);
\draw[kernels2] (0,0) -- (1,1) node[not] {};
\draw (-1,1)  node[xi] {} -- (-1,2.5) node[xi] {};
\draw[kernels2] (1,1) -- (0.4,2.5) node[xi] {};
\draw[kernels2] (1,1) -- (1.6,2.5) node[xi] {};
}

\DeclareSymbol{2I1Xi4cc1}{2}{\draw[kernels2] (0,0) node[not] {} -- (-1,1);
\draw[kernels2] (0,0) -- (1,1) node[not] {};
\draw (-1,1)  node[xic] {} -- (-1,2.5) node[xi] {};
\draw[kernels2] (1,1) -- (0.4,2.5) node[xic] {};
\draw[kernels2] (1,1) -- (1.6,2.5) node[xi] {};
}

\DeclareSymbol{2I1Xi4cc2}{2}{\draw[kernels2] (0,0) node[not] {} -- (-1,1);
\draw[kernels2] (0,0) -- (1,1) node[not] {};
\draw (-1,1)  node[xic] {} -- (-1,2.5) node[xic] {};
\draw[kernels2] (1,1) -- (0.4,2.5) node[xi] {};
\draw[kernels2] (1,1) -- (1.6,2.5) node[xi] {};
}

\DeclareSymbol{Xi4ba}{0}{\draw(-0.5,1.5) node[xi] {} -- (0,0); \draw (-1.5,1) node[xi] {} -- (0,0) node[not] {}; \draw[kernels2] (0,0) -- (1.5,1) node[xi] {};
\draw[kernels2] (0,0) -- (0.5,1.5) node[xi] {} ;}

\DeclareSymbol{Xi4badis}{0}{\draw(-0.5,1.5) node[xies] {} -- (0,0); \draw (-1.5,1) node[xies] {} -- (0,0) node[not] {}; \draw[kernels2] (0,0) -- (1.5,1) node[xies] {};
\draw[kernels2] (0,0) -- (0.5,1.5) node[xies] {} ;}

\DeclareSymbol{Xi4ba1}{0}{\draw(-0.5,1.5) node[xi] {} -- (0,0); \draw (-1.5,1) node[xi] {} -- (0,0) node[not] {}; \draw[kernels2] (0,0) -- (1.5,1) node[xic] {};
\draw[kernels2] (0,0) -- (0.5,1.5) node[xic] {} ;}

\DeclareSymbol{Xi4ba1b}{0}{\draw(-0.5,1.5) node[xic] {} -- (0,0); \draw (-1.5,1) node[xic] {} -- (0,0) node[not] {}; \draw[kernels2] (0,0) -- (1.5,1) node[xi] {};
\draw[kernels2] (0,0) -- (0.5,1.5) node[xi] {} ;}

\DeclareSymbol{Xi4ba1bdiff}{0}{\draw(-0.5,1.5) node[xic] {} -- (0,0); \draw (-1.5,1) node[xic] {} -- (0,0) node[not] {}; \draw (0,0) -- (1.5,1) node[xi] {};
\draw (0,0) -- (0.5,1.5) node[xi] {};
\draw(0,0) node[diff] {};}

\DeclareSymbol{Xi4ba1bb}{0}{\draw(-0.5,1.5) node[xic] {} -- (0,0); \draw (-1.5,1) node[xiesf] {} -- (0,0) node[not] {}; \draw[kernels2] (0,0) -- (1.5,1) node[xi] {};
\draw[kernels2] (0,0) -- (0.5,1.5) node[xi] {} ;}

\DeclareSymbol{Xi4ba2}{0}{\draw(-0.5,1.5) node[xi] {} -- (0,0); \draw (-1.5,1) node[xic] {} -- (0,0) node[not] {}; \draw[kernels2] (0,0) -- (1.5,1) node[xi] {};
\draw[kernels2] (0,0) -- (0.5,1.5) node[xic] {} ;}

\DeclareSymbol{Xi4ba2b}{0}{\draw(-0.5,1.5) node[xi] {} -- (0,0); \draw (-1.5,1) node[xic] {} -- (0,0) node[not] {}; \draw[kernels2] (0,0) -- (1.5,1) node[xi] {};
\draw[kernels2] (0,0) -- (0.5,1.5) node[xiesf] {} ;}

\DeclareSymbol{Xi4ca}{0}{\draw (0,1) -- (-1,2.2) node[xi] {};\draw (0,-0.25) node[xi] {} -- (0,1) ; \draw[kernels2] (0,1) node[not] {} -- (1,2.2) node[xi] {};
\draw[kernels2] (0,1) {} -- (0,2.7) node[xi] {};
}

\DeclareSymbol{Xi4cb}{0}{\draw (-1,1) -- (-2,2) node[xi] {};\draw[kernels2] (0,0)  -- (-1,1) node[xi] {} ; \draw[kernels2] (0,0) node[not] {} -- (1,1) node[xi] {} ; 
\draw (-1,1) node[xi] {} -- (0,2) node[xi] {};}

\DeclareSymbol{Xi4cbb}{0}{\draw (-1,1) -- (-2,2) node[xiesf] {};\draw[kernels2] (0,0)  -- (-1,1) node[xi] {} ; \draw[kernels2] (0,0) node[not] {} -- (1,1) node[xi] {} ; 
\draw (-1,1) node[xi] {} -- (0,2) node[xic] {};}

\DeclareSymbol{Xi4cbc1}{0}{\draw (-1,1) -- (-2,2) node[xic] {};\draw[kernels2] (0,0)  -- (-1,1) node[xic] {} ; \draw[kernels2] (0,0) node[not] {} -- (1,1) node[xi] {} ; 
\draw (-1,1) node[xic] {} -- (0,2) node[xi] {};}

\DeclareSymbol{Xi4cbc2}{0}{\draw (-1,1) -- (-2,2) node[xi] {};\draw[kernels2] (0,0)  -- (-1,1) node[xi] {} ; \draw[kernels2] (0,0) node[not] {} -- (1,1) node[xic] {} ; 
\draw (-1,1) node[xic] {} -- (0,2) node[xi] {};}

\DeclareSymbol{Xi4cab}{0}{\draw (-1,1) -- (-2,2) node[xi] {};\draw[kernels2] (0,0)  -- (-1,1); \draw[kernels2] (0,0) node[not] {} -- (1,1) node[xi] {} ; 
\draw[kernels2] (-1,1)  {} -- (0,2) node[xi] {};
\draw[kernels2] (-1,1) node[not] {} -- (-1,2.5) node[xi] {};
}

\DeclareSymbol{Xi4cabdis}{0}{\draw (-1,1) -- (-2,2) node[xies] {};\draw[kernels2] (0,0)  -- (-1,1); \draw[kernels2] (0,0) node[not] {} -- (1,1) node[xies] {} ; 
\draw[kernels2] (-1,1)  {} -- (0,2) node[xies] {};
\draw[kernels2] (-1,1) node[not] {} -- (-1,2.5) node[xies] {};
}

\DeclareSymbol{Xi4cabc1}{0}{\draw (-1,1) -- (-2,2) node[xi] {};\draw[kernels2] (0,0)  -- (-1,1); \draw[kernels2] (0,0) node[not] {} -- (1,1) node[xic] {} ; 
\draw[kernels2] (-1,1)  {} -- (0,2) node[xic] {};
\draw[kernels2] (-1,1) node[not] {} -- (-1,2.5) node[xi] {};
}

\DeclareSymbol{Xi4cabc2}{0}{\draw (-1,1) -- (-2,2) node[xic] {};\draw[kernels2] (0,0)  -- (-1,1); \draw[kernels2] (0,0) node[not] {} -- (1,1) node[xic] {} ; 
\draw[kernels2] (-1,1)  {} -- (0,2) node[xi] {};
\draw[kernels2] (-1,1) node[not] {} -- (-1,2.5) node[xi] {};
}

\DeclareSymbol{Xi4ea}{1.5}{\draw (-1,2.5) node[xi] {} -- (-1,1) node[xi] {} -- (0,0); 
 \draw[kernels2] (0,0)  -- (1,1) node[xi] {};
\draw[kernels2] (0,0) node[not] {} -- (0,1.5) node[xi] {}; }

\DeclareSymbol{Xi4eac1}{1.5}{\draw (-1,2.5) node[xic] {} -- (-1,1) node[xi] {} -- (0,0); 
 \draw[kernels2] (0,0)  -- (1,1) node[xic] {};
\draw[kernels2] (0,0) node[not] {} -- (0,1.5) node[xi] {}; }

\DeclareSymbol{Xi4eac1b}{1.5}{\draw (-1,2.5) node[xic] {} -- (-1,1) node[xi] {} -- (0,0); 
 \draw[kernels2] (0,0)  -- (1,1) node[xiesf] {};
\draw[kernels2] (0,0) node[not] {} -- (0,1.5) node[xi] {}; }

\DeclareSymbol{Xi4eac2}{1.5}{\draw (-1,2.5) node[xic] {} -- (-1,1) node[xic] {} -- (0,0); 
 \draw[kernels2] (0,0)  -- (1,1) node[xi] {};
\draw[kernels2] (0,0) node[not] {} -- (0,1.5) node[xi] {}; }

\DeclareSymbol{Xi4eact1}{1.5}{\draw (-1,2.5) node[xic] {} -- (-1,1) node[xi] {} -- (0,0); 
 \draw (0,0)  -- (1,1) node[xic] {};
\draw[rho] (0,0) node[not] {} -- (0,1.5) node[xi] {}; }

\DeclareSymbol{Xi4eact2}{1.5}{\draw[rho] (-1,2.5) node[xic] {} -- (-1,1) node[xi] {} -- (0,0); 
 \draw (0,0)  -- (1,1) node[xic] {};
\draw (0,0) node[not] {} -- (0,1.5) node[xi] {}; }

\DeclareSymbol{Xi4eabis}{1.5}{\draw (-1,2.5) node[xi] {} -- (-1,1) ; \draw[kernels2] (-1,1) node[xi] {} -- (0,0); 
 \draw (0,0)  -- (1,1) node[xi] {};
\draw[kernels2] (0,0) node[not] {} -- (0,1.5) node[xi] {}; }

\DeclareSymbol{Xi4eabisc1}{1.5}{\draw (-1,2.5) node[xic] {} -- (-1,1) ; \draw[kernels2] (-1,1) node[xi] {} -- (0,0); 
 \draw (0,0)  -- (1,1) node[xi] {};
\draw[kernels2] (0,0) node[not] {} -- (0,1.5) node[xic] {}; }

\DeclareSymbol{Xi4eabisc1b}{1.5}{\draw (-1,2.5) node[xic] {} -- (-1,1) ; \draw[kernels2] (-1,1) node[xi] {} -- (0,0); 
 \draw (0,0)  -- (1,1) node[xi] {};
\draw[kernels2] (0,0) node[not] {} -- (0,1.5) node[xiesf] {}; }

\DeclareSymbol{Xi4eabisc1bis}{1.5}{\draw (-1,2.5) node[xi] {} -- (-1,1) ; \draw[kernels2] (-1,1) node[xi] {} -- (0,0); 
 \draw (0,0)  -- (1,1) node[xi] {};
\draw[kernels2] (0,0) node[not] {} -- (0,1.5) node[xi] {};
\draw (-2,1) node[] {\tiny{$i$}};
\draw (-2,2.5) node[] {\tiny{$\ell$}};
\draw (2,1) node[] {\tiny{$k$}};
\draw (0,2.5) node[] {\tiny{$j$}};
 }

\DeclareSymbol{Xi4eabisc1tris}{1.5}{\draw (-1,2.5) node[xi] {} -- (-1,1) ; \draw[kernels2] (-1,1) node[xi] {} -- (0,0); 
 \draw (0,0)  -- (1,1) node[xi] {};
\draw[kernels2] (0,0) node[not] {} -- (0,1.5) node[xi] {};
\draw (-2,1) node[] {\tiny{i}};
\draw (-2,2.5) node[] {\tiny{j}};
\draw (2,1) node[] {\tiny{j}};
\draw (0,2.5) node[] {\tiny{i}};
 }

\DeclareSymbol{Xi4eabisc1quater}{1.5}{\draw (-1,2.5) node[xic] {} -- (-1,1) ; \draw[kernels2] (-1,1) node[xi] {} -- (0,0); 
 \draw (0,0)  -- (1,1) node[xic] {};
\draw[kernels2] (0,0) node[not] {} -- (0,1.5) node[xi] {};
 }

\DeclareSymbol{Xi4eabisc2}{1.5}{\draw (-1,2.5) node[xic] {} -- (-1,1) ; \draw[kernels2] (-1,1) node[xi] {} -- (0,0); 
 \draw (0,0)  -- (1,1) node[xic] {};
\draw[kernels2] (0,0) node[not] {} -- (0,1.5) node[xi] {}; }

\DeclareSymbol{Xi4eabisc2l}{1.5}{\draw (-1,2.5) node[xiesf] {} -- (-1,1) ; \draw[kernels2] (-1,1) node[xi] {} -- (0,0); 
 \draw (0,0)  -- (1,1) node[xic] {};
\draw[kernels2] (0,0) node[not] {} -- (0,1.5) node[xi] {}; }

\DeclareSymbol{Xi4eabisc2r}{1.5}{\draw (-1,2.5) node[xic] {} -- (-1,1) ; \draw[kernels2] (-1,1) node[xi] {} -- (0,0); 
 \draw (0,0)  -- (1,1) node[xiesf] {};
\draw[kernels2] (0,0) node[not] {} -- (0,1.5) node[xi] {}; }

\DeclareSymbol{Xi4eabisc3}{1.5}{\draw (-1,2.5) node[xic] {} -- (-1,1) ; \draw[kernels2] (-1,1) node[xic] {} -- (0,0); 
 \draw (0,0)  -- (1,1) node[xi] {};
\draw[kernels2] (0,0) node[not] {} -- (0,1.5) node[xi] {}; }

\DeclareSymbol{Xi4eb}{0}{
\draw[kernels2] (0,2) node[xi] {} -- (-1,1) ; \draw[kernels2] (-2,2)  node[xi] {} -- (-1,1) ; \draw (-1,1)  node[not] {} -- (0,0); 
 \draw (0,0) node[xi] {}  -- (1,1) node[xi] {};
}

\DeclareSymbol{Xi4eab}{1.5}{\draw[kernels2] (-1,2.5) node[xi] {} -- (-1,1) ; \draw[kernels2] (-2,2)  node[xi] {} -- (-1,1) ; \draw (-1,1)  node[not] {} -- (0,0); 
 \draw[kernels2] (0,0)  -- (1,1) node[xi] {};
\draw[kernels2] (0,0) node[not] {} -- (0,1.5) node[xi] {}; 
}

\DeclareSymbol{Xi4eabdis}{1.5}{\draw[kernels2] (-1,2.5) node[xies] {} -- (-1,1) ; \draw[kernels2] (-2,2)  node[xies] {} -- (-1,1) ; \draw (-1,1)  node[not] {} -- (0,0); 
 \draw[kernels2] (0,0)  -- (1,1) node[xies] {};
\draw[kernels2] (0,0) node[not] {} -- (0,1.5) node[xies] {}; 
}

\DeclareSymbol{Xi4eabc1}{1.5}{\draw[kernels2] (-1,2.5) node[xic] {} -- (-1,1) ; \draw[kernels2] (-2,2)  node[xi] {} -- (-1,1) ; \draw (-1,1)  node[not] {} -- (0,0); 
 \draw[kernels2] (0,0)  -- (1,1) node[xic] {};
\draw[kernels2] (0,0) node[not] {} -- (0,1.5) node[xi] {}; 
}

\DeclareSymbol{Xi4eabc2}{1.5}{\draw[kernels2] (-1,2.5) node[xi] {} -- (-1,1) ; \draw[kernels2] (-2,2)  node[xi] {} -- (-1,1) ; \draw (-1,1)  node[not] {} -- (0,0); 
 \draw[kernels2] (0,0)  -- (1,1) node[xic] {};
\draw[kernels2] (0,0) node[not] {} -- (0,1.5) node[xic] {}; 
}

\DeclareSymbol{Xi4eabbis}{1.5}{\draw[kernels2] (-1,2.5) node[xi] {} -- (-1,1) ; \draw[kernels2] (-2,2)  node[xi] {} -- (-1,1) ; \draw[kernels2] (-1,1)  node[not] {} -- (0,0); 
 \draw (0,0)  -- (1,1) node[xi] {};
\draw[kernels2] (0,0) node[not] {} -- (0,1.5) node[xi] {}; 
}

\DeclareSymbol{Xi4eabbisc1}{1.5}{\draw[kernels2] (-1,2.5) node[xic] {} -- (-1,1) ; \draw[kernels2] (-2,2)  node[xi] {} -- (-1,1) ; \draw[kernels2] (-1,1)  node[not] {} -- (0,0); 
 \draw (0,0)  -- (1,1) node[xic] {};
\draw[kernels2] (0,0) node[not] {} -- (0,1.5) node[xi] {}; 
}

\DeclareSymbol{Xi4eabbisc1perm}{1.5}{\draw[kernels2] (-1,2.5) node[xi] {} -- (-1,1) ; \draw[kernels2] (-2,2)  node[xic] {} -- (-1,1) ; \draw[kernels2] (-1,1)  node[not] {} -- (0,0); 
 \draw (0,0)  -- (1,1) node[xic] {};
\draw[kernels2] (0,0) node[not] {} -- (0,1.5) node[xi] {}; 
}

\DeclareSymbol{Xi4eabbisc2}{1.5}{\draw[kernels2] (-1,2.5) node[xi] {} -- (-1,1) ; \draw[kernels2] (-2,2)  node[xi] {} -- (-1,1) ; \draw[kernels2] (-1,1)  node[not] {} -- (0,0); 
 \draw (0,0)  -- (1,1) node[xic] {};
\draw[kernels2] (0,0) node[not] {} -- (0,1.5) node[xic] {}; 
}

\DeclareSymbol{Xi2cbis}{0}{\draw[kernels2] (0,1) -- (0.8,2.2) node[xi] {};\draw[kernels2] (0,-0.25) node[not] {} -- (0,1); \draw[kernels2] (0,1) node[not] {} -- (-0.8,2.2) node[xi] {};}

\DeclareSymbol{Xi2cbis1}{0}{\draw (0,1) -- (-0.8,2.2) node[xi] {};\draw[kernels2] (0,-0.25) node[not] {} -- (0,1) node[xi] {}; }

\DeclareSymbol{Xi2Xbis}{-2}{\draw[kernels2] (0,-0.25)  -- (-1,1) ; \draw (-1,1) node[xix] {};
\draw[kernels2] (0,-0.25) node[not] {} -- (1,1) node[xi] {};}

\DeclareSymbol{XXi2bis}{-2}{\draw[kernels2] (0,-0.25) -- (-1,1) node[xi] {};
\draw[kernels2] (0,-0.25) node[X] {} -- (1,1) node[xi] {};}

\DeclareSymbol{I1XiIXi}{0}{\draw[kernels2] (0,-0.25) -- (1,1) node[xi] {};
\draw (0,-0.25) node[not] {} -- (-1,1) node[xi] {};}

\DeclareSymbol{I1XiIXib}{0}{\draw  (0,-0.25) node[xi] {} -- (0,1) node[not] {};
\draw[kernels2] (0,1) -- (0,2.25) ; \draw (0,2.25) node[xi]{}; }

\DeclareSymbol{I1XiIXic}{0}{
\draw[kernels2] (0,0) -- (1,1) node[xi] {} ; 
\draw[kernels2] (0,0) node[not] {}  -- (-1,1) node[not] {} -- (0,2) node[xi] {};
}

\DeclareSymbol{thin}{1.4}{\draw[pagebackground] (-0.3,0) -- (0.3,0); \draw  (0,0) -- (0,2);}
\DeclareSymbol{thin2}{1.4}{\draw[pagebackground] (-0.3,0) -- (0.3,0); \draw[tinydots]  (0,0) -- (0,2);}

\DeclareSymbol{thick}{1.4}{\draw[pagebackground] (-0.3,0) -- (0.3,0); \draw[kernels2]  (0,0) -- (0,2);}

\DeclareSymbol{thick2}{1.4}{\draw[pagebackground] (-0.3,0) -- (0.3,0); \draw[kernels2,tinydots]  (0,0) -- (0,2);}

\DeclareSymbol{Xi4ind}{2}{\draw (0,0) node[xi,label={[label distance=-0.2em]right: \scriptsize  $ i $}]  { } -- (-1,1) node[xi,label={[label distance=-0.2em]left: \scriptsize  $ j $}] {} -- (0,2) node[xi,label={[label distance=-0.2em]right: \scriptsize  $ k $}] {} -- (-1,3) node[xi,label={[label distance=-0.2em]left: \scriptsize  $ \ell $}] {};}

\DeclareSymbol{Xi4c1}{2}{\draw (0,0) node[xic] {} -- (-1,1) node[xi] {} -- (0,2) node[xic] {} -- (-1,3) node[xi] {};} 
\DeclareSymbol{IXi2ex}{0}{\draw (0,-0.25) node[xie] {} -- (-1,1) node[xi] {} ; \draw (0,-0.25)-- (1,1) node[xi] {};}
\DeclareSymbol{IXi2ex1}{0}{\draw (0,-0.25) node[xie] {} -- (-1,1) node[xi] {} -- (0,2) node[xi] {};}

\DeclareSymbol{Xi4b1}{0}{\draw(0,1.5) node[xic] {} -- (0,0); \draw (-1,1) node[xic] {} -- (0,0) node[xi] {} -- (1,1) node[xi] {};}

\DeclareSymbol{Xi4ec1}{0}{\draw (0,2) node[xi] {} -- (-1,1) node[xic] {} -- (0,0) node[xic] {} -- (1,1) node[xi] {};}
\DeclareSymbol{Xi4ec2}{0}{\draw (0,2) node[xic] {} -- (-1,1) node[xi] {} -- (0,0) node[xic] {} -- (1,1) node[xi] {};}
\DeclareSymbol{Xi4ec3}{0}{\draw (0,2) node[xic] {} -- (-1,1) node[xic] {} -- (0,0) node[xi] {} -- (1,1) node[xi] {};}

\DeclareSymbol{I1Xi4ac1}{2}{\draw[kernels2] (0,0) node[not] {} -- (-1,1) ; \draw[kernels2] (0,0) node[not] {} -- (1,1) node[xic] {} ;
\draw (-1,1) node[xi] {} -- (0,2) node[xic] {} -- (-1,3) node[xi] {};}

\DeclareSymbol{I1Xi4ac2}{2}{\draw[kernels2] (0,0) node[not] {} -- (-1,1) ; \draw[kernels2] (0,0) node[not] {} -- (1,1) node[xic] {} ;
\draw (-1,1) node[xi] {} -- (0,2) node[xi] {} -- (-1,3) node[xic] {};}

\DeclareSymbol{I1Xi4bp}{2}{\draw (0,0) node[not] {} -- (-1,1) node[xi] {} -- (0,2) ; \draw[kernels2] (0,2) node[not] {} -- (-1,3) node[xi] {};\draw[kernels2] (0,2)  -- (1,3) node[xi] {};
}

\DeclareSymbol{I1Xi4bc1}{2}{\draw (0,0) node[xic] {} -- (-1,1) node[xi] {} -- (0,2) ; \draw[kernels2] (0,2) node[not] {} -- (-1,3) node[xi] {};\draw[kernels2] (0,2)  -- (1,3) node[xic] {};
}

\DeclareSymbol{I1Xi4bc2}{2}{\draw (0,0) node[xic] {} -- (-1,1) node[xi] {} -- (0,2) ; \draw[kernels2] (0,2) node[not] {} -- (-1,3) node[xic] {};\draw[kernels2] (0,2)  -- (1,3) node[xi] {};
}

\DeclareSymbol{I1Xi4cp}{2}{\draw (0,0) node[not] {} -- (-1,1) node[not] {}; \draw[kernels2] (-1,1) -- (0,2) ; 
\draw[kernels2] (-1,1) -- (-2,2) node[xi] {} ;
\draw (0,2) node[xi] {} -- (-1,3) node[xi] {};}

\DeclareSymbol{I1Xi4cc1}{2}{\draw (0,0) node[xic] {} -- (-1,1) node[not] {}; \draw[kernels2] (-1,1) -- (0,2) ; 
\draw[kernels2] (-1,1) -- (-2,2) node[xi] {} ;
\draw (0,2) node[xic] {} -- (-1,3) node[xi] {};}

\DeclareSymbol{I1Xi4cc2}{2}{\draw (0,0) node[xic] {} -- (-1,1) node[not] {}; \draw[kernels2] (-1,1) -- (0,2) ; 
\draw[kernels2] (-1,1) -- (-2,2) node[xi] {} ;
\draw (0,2) node[xi] {} -- (-1,3) node[xic] {};}

\DeclareSymbol{I1Xi4abc1}{2}{\draw[kernels2] (0,0) node[not] {} -- (-1,1) ; \draw[kernels2] (0,0) node[not] {} -- (1,1) node[xic] {};\draw (-1,1) node[xi] {} -- (0,2) ; \draw[kernels2] (0,2) node[not] {} -- (-1,3) node[xic] {};\draw[kernels2] (0,2)  -- (1,3) node[xi] {}; }

\DeclareSymbol{I1Xi4abc2}{2}{\draw[kernels2] (0,0) node[not] {} -- (-1,1) ; \draw[kernels2] (0,0) node[not] {} -- (1,1) node[xic] {};\draw (-1,1) node[xi] {} -- (0,2) ; \draw[kernels2] (0,2) node[not] {} -- (-1,3) node[xi] {};\draw[kernels2] (0,2)  -- (1,3) node[xic] {}; }

\DeclareSymbol{R1}{0}{\draw (-1,1) node[xi] {} -- (0,0) node[not] {};
\draw[kernels2] (0,1.5) node[xic] {} -- (0,0) -- (1,1) node[xic] {};}
\DeclareSymbol{R2}{0}{\draw (-1,1) node[xic] {} -- (0,0) node[not] {};
\draw[kernels2] (0,1.5)  {} -- (0,0) -- (1,1)  {};
\draw (0,1.5) node[xi] {};
\draw (1,1) node[xic] {};
}
\DeclareSymbol{R3}{1}{\draw[kernels2] (-1,1.5)  {} -- (0,0) node[not] {} -- (1,1.5);
\draw (-1,1.5) node[xi] {};
\draw[kernels2] (0,3) {} -- (1,1.5) -- (2,3)  {};
\draw  (0,3) node[xic] {} ;
\draw (2,3) node[xic] {};}
\DeclareSymbol{R4}{1}{\draw[kernels2] (-1,1.5) node[xic] {} -- (0,0) node[not] {} -- (1,1.5);
\draw[kernels2] (0,3) {} -- (1,1.5) -- (2,3) node[xic] {};
\draw (0,3) node[xi] {};}

\DeclareSymbol{I1Xi4bcp}{2}{\draw (0,0) node[not] {} -- (-1,1) node[not] {}; \draw[kernels2] (-1,1) -- (0,2) ; 
\draw[kernels2] (-1,1) -- (-2,2) node[xi] {} ; \draw[kernels2] (0,2) node[not] {} -- (-1,3) node[xi] {};\draw[kernels2] (0,2)  -- (1,3) node[xi] {};
}

\DeclareSymbol{I1Xi4bcc1}{2}{\draw (0,0) node[xic] {} -- (-1,1) node[not] {}; \draw[kernels2] (-1,1) -- (0,2) ; 
\draw[kernels2] (-1,1) -- (-2,2) node[xi] {} ; \draw[kernels2] (0,2) node[not] {} -- (-1,3) node[xi] {};\draw[kernels2] (0,2)  -- (1,3) node[xic] {};
}

\DeclareSymbol{I1Xi4bcc2}{2}{\draw (0,0) node[xic] {} -- (-1,1) node[not] {}; \draw[kernels2] (-1,1) -- (0,2) ; 
\draw[kernels2] (-1,1) -- (-2,2) node[xi] {} ; \draw[kernels2] (0,2) node[not] {} -- (-1,3) node[xic] {};\draw[kernels2] (0,2)  -- (1,3) node[xi] {};
} 

\DeclareSymbol{2I1Xi4bc1}{2}{\draw[kernels2] (0,0) node[not] {} -- (-1,1) ;
\draw[kernels2] (0,0) -- (1,1);
\draw (-1,1) node[xic] {} -- (-1,2.5) node[xi] {};
\draw (1,1)  node[xic] {} -- (1,2.5) node[xi] {};
}

\DeclareSymbol{2I1Xi4bc2}{2}{\draw[kernels2] (0,0) node[not] {} -- (-1,1) ;
\draw[kernels2] (0,0) -- (1,1);
\draw (-1,1) node[xi] {} -- (-1,2.5) node[xic] {};
\draw (1,1)  node[xic] {} -- (1,2.5) node[xi] {};
}

\DeclareSymbol{diff2I1Xi4bc2}{2}{\draw (0,0) node[diff] {} -- (-1,1) ;
\draw (0,0) -- (1,1);
\draw (-1,1) node[xi] {} -- (-1,2.5) node[xic] {};
\draw (1,1)  node[xic] {} -- (1,2.5) node[xi] {};
}

\DeclareSymbol{2I1Xi4bc3}{2}{\draw[kernels2] (0,0) node[not] {} -- (-1,1) ;
\draw[kernels2] (0,0) -- (1,1);
\draw (-1,1) node[xic] {} -- (-1,2.5) node[xic] {};
\draw (1,1)  node[xi] {} -- (1,2.5) node[xi] {};
}

\DeclareSymbol{Xi41}{0}{\draw (0,1) -- (0.8,2.2) node[xic] {};\draw (0,-0.25) node[xi] {} -- (0,1) node[xi] {} -- (-0.8,2.2) node[xic] {};} 

\DeclareSymbol{Xi42}{0}{\draw (0,1) -- (0.8,2.2) node[xi] {};\draw (0,-0.25) node[xic] {} -- (0,1) node[xi] {} -- (-0.8,2.2) node[xic] {};}

\DeclareSymbol{Xi4ca1}{0}{\draw (0,1) -- (-1,2.2) node[xic] {};\draw (0,-0.25) node[xi] {} -- (0,1) ; \draw[kernels2] (0,1) node[not] {} -- (1,2.2) node[xic] {};
\draw[kernels2] (0,1) {} -- (0,2.7) node[xi] {};
}

\DeclareSymbol{Xi4ca2}{0}{\draw (0,1) -- (-1,2.2) node[xi] {};\draw (0,-0.25) node[xi] {} -- (0,1) ; \draw[kernels2] (0,1) node[not] {} -- (1,2.2) node[xic] {};
\draw[kernels2] (0,1) {} -- (0,2.7) node[xic] {};
}

\DeclareSymbol{Xi4cap}{0}{\draw (0,1) -- (-1,2.2) node[xi] {};\draw (0,-0.25) node[not] {} -- (0,1) ; \draw[kernels2] (0,1) node[not] {} -- (1,2.2) node[xi] {};
\draw[kernels2] (0,1) {} -- (0,2.7) node[xi] {};
}

\DeclareSymbol{Xi3a}{0}{
 \draw (-1,1)  node[xi] {} -- (0,0); 
 \draw (0,0) node[xi] {}  -- (1,1) node[xi] {};
}

\DeclareSymbol{Xi4ebc1}{0}{
\draw[kernels2] (0,2) node[xi] {} -- (-1,1) ; \draw[kernels2] (-2,2)  node[xic] {} -- (-1,1) ; \draw (-1,1)  node[not] {} -- (0,0); 
 \draw (0,0) node[xic] {}  -- (1,1) node[xi] {};
}

\DeclareSymbol{Xi4ebc2}{0}{
\draw[kernels2] (0,2) node[xi] {} -- (-1,1) ; \draw[kernels2] (-2,2)  node[xi] {} -- (-1,1) ; \draw (-1,1)  node[not] {} -- (0,0); 
 \draw (0,0) node[xic] {}  -- (1,1) node[xic] {};
}

\DeclareSymbol{Xi2cbispex}{0}{\draw[kernels2] (0,1) -- (0.8,2.2) node[xi] {};\draw (0,-0.25) node[xie] {} -- (0,1); \draw[kernels2] (0,1) node[not] {} -- (-0.8,2.2) node[xi] {};}

\DeclareSymbol{Xi2cbis1p}{0}{\draw (0,1) -- (-0.8,2.2) node[xi] {};\draw (0,-0.25) node[not] {} -- (0,1) node[xi] {}; }

\DeclareSymbol{Xi2Xp}{-2}{\draw (0,-0.25) node[not] {} -- (-1,1) node[xix] {};} 

\DeclareSymbol{I1XiIXib}{0}{\draw  (0,-0.25) node[xi] {} -- (0,1) node[not] {};
\draw[kernels2] (0,1) -- (0,2.25) ; \draw (0,2.25) node[xi]{}; }

\DeclareSymbol{IXi2b}{0}{\draw  (0,-0.25) node[xi] {} -- (0,1) node[not] {};
\draw (0,1) -- (0,2.25) ; \draw (0,2.25) node[xi]{}; }

\DeclareSymbol{IXi2bex}{0}{\draw  (0,-0.25) node[xi] {} -- (0,1) node[xie] {};
\draw (0,1) -- (0,2.25) ; \draw (0,2.25) node[xi]{}; }

 \def\1{\mathbf{\symbol{1}}}

\def\one{\mathbf{1}}

\DeclareSymbol{diff}{0}{
\draw (0,0.5) node[diff] {};
}

\DeclareSymbol{diff1}{0}{
\draw (0,0.5) node[diff1] {};
}

\DeclareSymbol{diff2}{0}{
\draw (0,0.5) node[diff2] {};
}

\DeclareSymbol{geo}{0}{
\draw (0,0) node[diff] {};
\draw (0.3,0) node[diff] {};
}

\DeclareSymbol{generic}{0}{
\draw (0,0.6) node[xi] {};
}

\DeclareSymbol{g}{0}{
\draw (0,0.6) node[g] {};
}

\DeclareSymbol{Ito}{0}{
\draw (0,0.6) node[xies] {};
}

\DeclareSymbol{Itob}{0}{
\draw (0,0.6) node[xiesf] {};
}

\DeclareSymbol{greycirc}{0}{
\draw (0,0.3) node[xi] {};
}

\DeclareSymbol{not}{0}{
\draw (0,0.6) node[not] {};
\draw[tinydots] (0,0.6) circle (0.8);
}

\DeclareSymbol{genericb}{0}{
\draw (0,0.6) node[xic] {};
}

\DeclareSymbol{bluecirc}{0}{
\draw (0,0.3) node[xic] {};
}

\DeclareSymbol{genericxix}{0}{
\draw (0,0.6) node[xix] {};
}

\DeclareSymbol{genericX}{0}{
\draw (0,0.6) node[X] {};
}

\DeclareSymbol{diffIto}{1}{
\draw  (0,2.5) -- (0,0) ;
\draw (0,-0.1) node[diff] {};
\draw (0,2.5) node[xies] {};
}
\DeclareSymbol{Itodiff}{2}{
\draw(0,2.9) -- (0,-0.2);
\draw (0,2.9) node[diff] {};
\draw (0,-0.1) node[xies] {};
}

\DeclareSymbol{diffgeneric}{1}{
\draw  (0,2.5) -- (0,0) ;
\draw (0,-0.1) node[diff] {};
\draw (0,2.5) node[xi] {};
}

\DeclareSymbol{genericdiff}{2}{
\draw(0,2.9) -- (0,-0.2);
\draw (0,2.9) node[diff] {};
\draw (0,-0.1) node[xi] {};
}

\DeclareSymbol{diffdot}{2}{
\draw  (0,3) -- (0,-0.1) ;
\draw (0,3) node[not] {};
\draw (0,-0.1) node[diff] {};
}

\DeclareSymbol{diffdotmini}{0}{
\draw  (0,0) -- (0,1.2) ;
\draw (0,1.2) node[not] {};
\draw (0,0) node[diffmini] {};
}

\DeclareSymbol{dotdiff}{2}{
\draw[kernelsmod]  (0,3) -- (0,-0.1) ;
\draw (0,3) node[diff] {};
\draw (0,-0.1) node[not] {};
}

\DeclareSymbol{dotdiff1}{2}{
\draw[kernelsmod]  (0,3) -- (0,-0.1) ;
\draw (0,3) node[diff1] {};
\draw (0,-0.1) node[not] {};
}

\DeclareSymbol{dotdiff1mini}{0}{
\draw[kernelsmod]  (0,1.2) -- (0,0) ;
\draw (0,1.2) node[diffmini] {};
\draw (0,0) node[not] {};
}

\DeclareSymbol{dotdiff2}{2}{
\draw (0,3) -- (0,-0.1) ;
\draw (0,3) node[diff] {};
\draw (0,-0.1) node[not] {};
}

\DeclareSymbol{dotdiff2mini}{0}{
\draw (0,1.2) -- (0,0) ;
\draw (0,1.2) node[diffmini] {};
\draw (0,0) node[not] {};
}

\DeclareSymbol{dotdiffstraight}{0}{
\draw  (0,3) -- (0,-0.1) ;
\draw (0,3) node[diff] {};
\draw (0,-0.1) node[not] {};
}

\DeclareSymbol{arbre1}{0}{
\draw  (0,0) -- (1.5,1.5) ;
\draw (1.5,1.5) node[not] {};
\draw (0,0) node[not] {};
}

\DeclareSymbol{arbre2}{0}{
\draw  (0,0) -- (1.5,1.5) ;
\draw[kernelsmod] (0,0) -- (-1.5,1.5);
\draw (1.5,1.5) node[not] {};
\draw (0,0) node[not] {};
\draw (-1.5,1.5) node[xi] {};
}

\DeclareSymbol{arbre3}{0}{
\draw  (0,0) -- (1.5,1.5) ;
\draw[kernelsmod] (1.5,1.5) -- (0,3);
\draw (0,0) node[not] {};
\draw (1.5,1.5) node[not] {};
\draw (0,3) node[xi] {};
}

\DeclareSymbol{treeeval}{0}{
\draw (0,0) -- (1,1);
\draw (0,0) node[xi] {};
\draw (1.25,1.25) node[xi] {};
\draw (-0.6,0.6) node[]{\tiny{$i$}};
\draw (0.65,1.85) node[]{\tiny{$j$}};
}

\DeclareSymbol{testeval}{0}{
\draw (0,0) -- (1,1);
\draw (0,0) -- (-1,1);
\draw (0,0) node[xi] {};
\draw (1.25,1.25) node[xi] {};
\draw (-1.25,1.25) node[xi] {};
\draw (-0.6,-0.6) node[]{\tiny{$i$}};
\draw (0.65,1.85) node[]{\tiny{$j$}};
\draw (-1.95,1.85) node[]{\tiny{$k$}};
}

\DeclareSymbol{treeeval2}{0}{
\draw[kernelsmod] (-0.25,-1) -- (1,0.5) ;
\draw[kernelsmod] (1,0.5) -- (-0.25,2);
\draw (1,0.5) node[diff2] {};
\draw (-0.25,-1) node[not] {};
\draw (-0.25,2) node[xi] {};
\draw (-0.6,1.2) node[]{\tiny{1}};
}

\DeclareSymbol{arbreact}{1}{
\draw (0,0) node[not] {};
\draw[kernelsmod] (0,0) -- (1,1);
\draw[kernelsmod] (0,0) -- (-1,1);
\draw (-1,1) node[xic] {};
\draw  (0,2) -- (1,1) ;
\draw (0,2) node[xic] {};
\draw (1,1) node[xi] {};
}

\DeclareSymbol{arbreact1}{0}{
\draw (0,-1.5) -- (0,0);
\draw[kernelsmod] (0,0) -- (1,1);
\draw[kernelsmod] (0,0) -- (-1,1);
\draw  (0,2) -- (1,1) ;
\draw (0,-1.5) node[diff] {};
\draw (0,0) node[not] {};
\draw (-1,1) node[xic] {};
\draw (0,2) node[xic] {};
\draw (1,1) node[xi] {};
}

\DeclareSymbol{arbreact2}{0}{
\draw (0,-0.75) -- (-1,0.5); 
\draw (0,-0.75) -- (1,0.5);
\draw (0,1.5) -- (1,0.5);
\draw (0,1.5) node[xic] {};
\draw (1,0.5) node[xi] {};
\draw (-1,0.5) node[xic] {};
\draw (0,-0.75) node[diff] {};
}

\DeclareSymbol{arbreact3}{0}{
\draw[kernelsmod] (0,-0.75) -- (-1,0.5); 
\draw[kernelsmod] (0,-0.75) -- (1,0.5);
\draw (0,1.75) -- (1,0.5);
\draw (2,1.75) -- (1,0.5);
\draw (0,1.75) node[xic] {};
\draw (1,0.5) node[diff] {};
\draw (-1,0.5) node[xic] {};
\draw (2,1.75) node[xi] {};
\draw (0,-0.75) node[not] {};
}

\DeclareSymbol{pre_im_I1Xitwo}{0}{
\draw[kernels2] (0,-0.3) node[not] {} -- (-0.6,0.7) ;
\draw[kernels2] (0,-0.3) -- (0.6,0.7);
\draw (0,0.9) node[g] {};
}

\DeclareSymbol{pre_im_cI1Xi4ab}{2}{
\draw[kernels2] (0,-1) node[not] {} -- (-0.6,0) ;
\draw[kernels2] (0,-1) -- (0.6,0);
\draw (0,0.2) node[g] {};
\draw (0,0.6) -- (0,1.5);
\draw[kernels2] (0,1.5) node[not] {} -- (-0.6,2.5) ;
\draw[kernels2] (0,1.5) -- (0.6,2.5);
\draw (0,2.7) node[g] {};
}

\DeclareSymbol{pre_im_I1Xi4acc2}{0}{
\draw[kernels2] (-1,-0.5) node[not] {} -- (-1.6,0.5) ;
\draw[kernels2] (-1,-0.5) -- (-0.4,0.5);
\draw[kernels2] (-1,-0.5) -- (0.2,-1.5) node[not] {} ;
\draw (-1,1.1) -- (-1,2);
\draw[kernels2] (0.2,-1.5) -- (0.2,2);
\draw (-1,0.7) node[g] {};
\draw (-0.3,2.2) node[g] {};
}

\DeclareSymbol{pre_im_I1Xi4abcc2}{2}{
\draw[kernels2] (0,-1) node[not] {} -- (-1,0) node[not] {};
\draw[kernels2] (-1,1.2) node[not] {} -- (-1,0);
\draw[kernels2] (-1,1.2) -- (-1.5,2.5);
\draw[kernels2] (-1,1.2) -- (-0.5,2.5);
\draw[kernels2] (-1,0) -- (0.7,2.5);
\draw[kernels2] (0,-1) -- (1.5,2.5);
\draw (-1,2.7) node[g] {};
\draw (1,2.7) node[g] {};
}

\DeclareSymbol{pre_im_2I1Xi4c1}{2}{
\draw[kernels2] (0,-0.5) node[not] {} -- (-1,0.5) node[not] {};
\draw[kernels2] (0,-0.5) -- (1,0.5) node[not] {};
\draw[kernels2] (-1,0.5) node[not] {}-- (-1.7,2);
\draw[kernels2]  (-1,2) -- (1,0.5);
\draw[kernels2] (-1,0.5) -- (1,2);
\draw[kernels2] (1,0.5) -- (1.7,2);
\draw (-1.2,2.2) node[g] {};
\draw (1.2,2.2) node[g] {};
}

\DeclareSymbol{pre_im_Xi4eabisc2}{2}{
\draw[kernels2] (1.2,-0.5) node[not] {} -- (-0.7,0.8) ;
\draw[kernels2] (1.2,-0.5) -- (0.4,0.8);
\draw (0,1.4)  -- (0,2.2);
\draw (1.2,2.2) -- (1.2,-0.6);
\draw (0,1) node[g] {};
\draw (0.6,2.4) node[g] {};
}

\DeclareSymbol{pre_im_Xi4eabisc22}{2}{
\draw (1.2,-0.5) node[not] {} -- (-0.7,0.8) ;
\draw[kernels2] (1.2,-0.5) -- (0.4,0.8);
\draw (0,1.4)  -- (0,2.2);
\draw[kernels2] (1.2,2.2) -- (1.2,-0.6);
\draw (0,1) node[g] {};
\draw (0.6,2.4) node[g] {};
}

\DeclareSymbol{pre_im_Xi4eabisc222}{2}{
\draw[kernels2] (0.4,-0.5) node[not] {} -- (-0.6,1) ;
\draw[kernels2] (1.2,0) -- (0.3,1);
\draw (0,1.1)  -- (0,2.5);
\draw[kernels2] (1.2,2.5) -- (1.2,0) node[not] {} -- (0.4,-0.6);
\draw (0,1.2) node[g] {};
\draw (0.6,2.5) node[g] {};
}

\DeclareSymbol{pre_im_Xi4eabc2}{2}{
\draw (0,-0.5) node[not] {} -- (-1,0.5) node[not] {};
\draw[kernels2] (-1,0.5) -- (-1.5,2);
\draw[kernels2] (0,-0.5)  -- (0.7,2);
\draw[kernels2] (-1,0.5) -- (-0.5,2);
\draw[kernels2] (0,-0.5) -- (1.5,2);
\draw (-1,2.2) node[g] {};
\draw (1,2.2) node[g] {};
}

\DeclareSymbol{pre_im_Xi4eabbisc2}{2}{
\draw[kernels2] (0,-0.5) node[not] {} -- (-1,0.5) node[not] {};
\draw[kernels2] (-1,0.5) -- (-1.5,2);
\draw[kernels2] (0,-0.5)  -- (0.7,2);
\draw[kernels2] (-1,0.5) -- (-0.5,2);
\draw (0,-0.5) -- (1.5,2);
\draw (-1,2.2) node[g] {};
\draw (1,2.2) node[g] {};
}

\DeclareSymbol{pre_im_I1Xi4abcc1}{2}{
\draw[kernels2] (0,-1) node[not] {} -- (-1,0) node[not] {};
\draw[kernels2] (0,1.1) node[not] {} -- (-1,0);
\draw[kernels2] (-1,0) -- (-1.5,2.5);
\draw[kernels2] (0,1.1) node[not] {} -- (-0.5,2.5);
\draw[kernels2] (0,1.1) -- (0.5,2.5);
\draw[kernels2] (0,-1) -- (1.5,2.5);
\draw (-1,2.7) node[g] {};
\draw (1,2.7) node[g] {};
}

\DeclareSymbol{pre_im_Xi4eabc1}{2}{
\draw (0,-0.5) node[not] {} -- (-1,0.5) node[not] {};
\draw[kernels2] (-1,0.5) -- (-1.5,2);
\draw[kernels2] (0,-0.5)  -- (-0.5,2);
\draw[kernels2] (-1,0.5) -- (0.8,2);
\draw[kernels2] (0,-0.5) -- (1.5,2);
\draw (-1,2.2) node[g] {};
\draw (1,2.2) node[g] {};
}

\DeclareSymbol{pre_im_Xi4ba1b}{2}{
\draw[kernels2] (0,0) node[not] {}  -- (1.8,1.5);
\draw[kernels2] (0,0) -- (0.8,1.5);
\draw (0,-0.1) -- (-1.8,1.5);
\draw (0,-0.1) -- (-0.8,1.5);
\draw (-1,1.7) node[g] {};
\draw (1,1.7) node[g] {};
}

\DeclareSymbol{pre_im_Xi4ba2}{2}{
\draw (0,-0.1) node[not] {}  -- (1.8,1.5);
\draw[kernels2] (0,0) -- (0.8,1.5);
\draw (0,-0.1) -- (-1.8,1.5);
\draw[kernels2] (0,0) -- (-0.8,1.5);
\draw (-1,1.7) node[g] {};
\draw (1,1.7) node[g] {};
}

\DeclareSymbol{pre_im_Xi4cabc2}{2}{
\draw[kernels2] (0,-0.5) node[not] {} -- (-1,0.5) node[not] {};
\draw[kernels2] (-1,0.5) -- (-1.5,2);
\draw (-1,0.5)  -- (0.7,2);
\draw[kernels2] (-1,0.5) -- (-0.5,2);
\draw[kernels2] (0,-0.5) -- (1.5,2);
\draw (-1,2.2) node[g] {};
\draw (1,2.2) node[g] {};
}

\DeclareSymbol{pre_im_Xi4cabc1}{2}{
\draw[kernels2] (0,-0.5) node[not] {} -- (-1,0.5) node[not] {};
\draw (-1,0.5) -- (-1.5,2);
\draw[kernels2] (-1,0.5)  -- (0.7,2);
\draw[kernels2] (-1,0.5) -- (-0.5,2);
\draw[kernels2] (0,-0.5) -- (1.5,2);
\draw (-1,2.2) node[g] {};
\draw (1,2.2) node[g] {};
}

\DeclareSymbol{pre_im_Xi4eabbisc1}{2}{
\draw[kernels2] (0,-0.5) node[not] {} -- (-1,0.5) node[not] {};
\draw[kernels2] (-1,0.5) -- (-1.5,2);
\draw[kernels2] (0,-0.5)  -- (-0.5,2);
\draw[kernels2] (-1,0.5) -- (0.8,2);
\draw (0,-0.5) -- (1.5,2);
\draw (-1,2.2) node[g] {};
\draw (1,2.2) node[g] {};
}

\DeclareSymbol{pre_im_1}{0}{
\draw[kernels2] (0,-0.5) node[not] {} -- (-0.6,0.5) ;
\draw[kernels2] (0,-0.5) -- (0.6,0.5);
\draw (0,1.1)  -- (-0.55,2);
\draw (0,1.1)  -- (0.55,2);
\draw (0,0.7) node[g] {};
\draw (0,2.2) node[g] {};
}

\DeclareSymbol{disconnect}{0}{
\draw[kernels2] (0,-0.5) node[not] {} -- (-0.6,0.5) ;
\draw[kernels2] (0,-0.5) -- (0.6,0.5);
\draw (-0.55,1.1)  -- (-0.55,2.3);
\draw (0.55,2.3) -- (0.55,1.5) -- (1.2,1.5) -- (1.2,3.5) -- (0.55,3.5) -- (0.55,2.7);
\draw (0,0.7) node[g] {};
\draw (0,2.5) node[g] {};
}

\DeclareSymbol{pre_im_2}{2}{\draw[kernels2] (0,0) node[not] {} -- (-1,1) node[not] {};
\draw[kernels2] (0,0) -- (1,1) node[not] {};
\draw[kernels2] (-1,1) -- (-1.5,2.5);
\draw[kernels2] (-1,1) -- (-0.5,2.5);
\draw[kernels2] (1,1) -- (0.5,2.5);
\draw[kernels2] (1,1) -- (1.5,2.5);
\draw (-1,2.7) node[g] {};
\draw (1,2.7) node[g] {};
}

\DeclareSymbol{CX_rec}{0}{
\draw [black] (-0.3,1) to (-0.3,-0.3);
\draw [black] (0.3,1) to (0.3,-0.3);
\draw [black] (-0.3,1) to (-0.3,2.3);
\draw [black] (0.3,1) to (0.3,2.3);
\draw (0,1) node[rec] {};
}

\DeclareSymbol{CX_cerc}{0}{
\draw [black] (0,1) to (0,-0.3);
\draw (0,1) node[cerc] {};
}

\DeclareMathAlphabet{\mathpzc}{OT1}{pzc}{m}{it}

\def\eqref#1{(\ref{#1})}

\makeatletter 
\newcommand*{\bigcdot}{}
\DeclareRobustCommand*{\bigcdot}{%
  \mathbin{\mathpalette\bigcdot@{}}%
}
\newcommand*{\bigcdot@scalefactor}{.5}
\newcommand*{\bigcdot@widthfactor}{1.15}
\newcommand*{\bigcdot@}[2]{%
  \sbox0{$#1\vcenter{}$}
  \sbox2{$#1\cdot\m@th$}%
  \hbox to \bigcdot@widthfactor\wd2{%
    \hfil
    \raise\ht0\hbox{%
      \scalebox{\bigcdot@scalefactor}{%
        \lower\ht0\hbox{$#1\bullet\m@th$}%
      }%
    }%
    \hfil
  }%
}
\makeatother

\tcbset
{colframe=boxcolor,colback=symbols!7!pagebackground,coltext=pageforeground,
fonttitle=\bfseries,nobeforeafter,center title,size=fbox,boxsep=1.5pt,
top=0mm,bottom=0mm,boxsep=0mm,tcbox raise base}

\def\two{{\<generic>\kern0.05em\<genericb>}}
\def\twoI{{\<Ito>\kern0.05em\<Itob>}}

\def\mail#1{\burlalt{#1}{mailto:#1}}

\usepackage{thmtools} 

\newcommand{\nlsletter}[4]{
	\begin{tikzpicture}[scale=0.2,baseline=-5]
		\coordinate (root) at (0,-1);
		\coordinate (rightc) at (1,2);
		\coordinate (right) at (3,2);
		\coordinate (leftc) at (-1,2);
		\coordinate (left) at (-3,2);
		\draw[kernels2] (right) -- (root);
		\draw[kernels2] (rightc) -- (root);
		\draw[kernels2,tinydots] (left) -- (root);
		\draw[kernels2,tinydots] (leftc) -- (root);
		\node[var2] (rootnode) at (left) {\tiny$#1$};
		\node[var] (rootnode) at (leftc) {\tiny$#2$};	
		\node[var] (rootnode) at (rightc) {\tiny$#3$};	
		\node[var] (rootnode) at (right) {\tiny$#4$};
	\end{tikzpicture}
}

\newcommand{\nlsconjletter}[4]{
	\begin{tikzpicture}[scale=0.2,baseline=-5]
		\coordinate (root) at (0,-1);
		\coordinate (rightc) at (1,2);
		\coordinate (right) at (3,2);
		\coordinate (leftc) at (-1,2);
		\coordinate (left) at (-3,2);
		\draw[kernels2,tinydots] (right) -- (root);
		\draw[kernels2,tinydots] (rightc) -- (root);
		\draw[kernels2] (left) -- (root);
		\draw[kernels2] (leftc) -- (root);
		\node[var2] (rootnode) at (left) {\tiny$#1$};
		\node[var] (rootnode) at (leftc) {\tiny$#2$};	
		\node[var] (rootnode) at (rightc) {\tiny$#3$};	
		\node[var] (rootnode) at (right) {\tiny$#4$};
	\end{tikzpicture}
}

\begin{document}

\title{Cancellations for dispersive PDEs with random initial data}
\author{Yvain Bruned$^1$, Leonardo Tolomeo$^2$}
\institute{ 
 Université de Lorraine, CNRS, IECL, F-54000 Nancy, France
 \and School of Mathematics, University of Edinburgh
  \\
Email:\ \begin{minipage}[t]{\linewidth}
\mail{yvain.bruned@univ-lorraine.fr}
\\
\mail{l.tolomeo@ed.ac.uk}.
\end{minipage}}

\maketitle

\begin{abstract}
In this work, we provide a combinatorial formalism for dealing with the cancellations that have appeared recently in the context of dispersive PDEs with random initial data. The main idea is to transform iterated integrals encoded by decorated trees into words via an arborification map. This provides a formalism alternative to the one of molecules introduced by Deng and Hani (2023). 
It allows us to compute the cancellations coming from Wave turbulence and the proof of the invariance of the Gibbs measure under the dynamics of the three-dimensional cubic wave equation.
\end{abstract}

\setcounter{tocdepth}{1}
\tableofcontents

\section{Introduction}

In recent years, perturbative expansions involving trees, tensors and their associated Feynman diagrams have been pushed forward in the context of dispersive equations with random initial data in order to obtain (sub)critical results. One of the main pushes in this direction is the work \cite{DNY22}, which introduced the concept of ``random tensors" and provided the analytical framework to compute the tensor norm of Feynman diagrams seen as appropriate (random) multilinear operators. 
The authors immediately applied the theory to showing almost-sure local well-posedness for semilinear Schr\"odinger equations, subcritical under a suitable ``probabilistic scaling".

However, in contrast to Regularity Structures, for which a black box has been designed in  \cite{reg,BHZ,CH16,BCCH} for solving locally a large class of parabolic subcritical equations, such a general formulation is not there in the context of dispersive equations with random initial data. Even equipped with the random tensors theory, one has to complete many estimates in \cite{BDNY24}. The main difficulty is to take into account  the specificities of various equations which at some point is also the case for parabolic equations in long time existence results where no general theory is available.

Nevertheless, this has not stopped the community from showing a number of fundamental results. In \cite{DH23,DH2301,DH2311}, Deng and Hani developed a rigorous justification for wave-kinetic equations . This problem is critical (in a suitable sense), and the main strategy of the proof involves summing infinitely many Feynman diagrams. 
More recently, using the same ideas, a long time derivation of the Boltzmann equation has been provided in \cite{DHM24}. 
In parallel, Bringmann, Deng, Nahmod and Yue in \cite{BDNY24} showed that the $\Phi^4_3$ (Gibbs) measure is invariant under the dynamics of the three-dimensional cubic wave equation. 

Both results above are extremely impressive from an analytical point of view, but something striking about both results is that they rely on some ``miraculous cancellations" \cite[Section 3.3]{DH2301} between Feynman diagrams. 
In \cite{DH23}, in order to perform the combinatorics necessary, the authors introduced the concept of "molecules", which are simplified Feynman diagrams that retain the information necessary in order to perform the analytical estimates for the wave-kinetic theory. They suggest that the ``miraculous cancellations" can be guessed by observing that the Feynman diagrams that cancels out correspond to the same molecule. This concept was also used in \cite{BDNY24} in order to greatly simplify the number of analytical estimates necessary, however, as far as the author are aware, they do not explain the necessary cancellation between sextic objects.


The main aim of this work is to introduce a general unified framework for computing and understanding the cancellations above. Before presenting our framework and its
key ideas, we take a small detour and briefly explain how cancellations are treated in the context of parabolic equations.

For parabolic singular SPDEs, a hidden logarithmic cancellation has been first observed for the KPZ equation in \cite{KPZ}, where rough paths techniques are used for solving this singular equation. The computations there do not use general graphical rules. 
The approach started to become more systematic in \cite{HP15,HQ18}, where graphical rules were introduced. These rely on the fact that the heat kernel $ K $ is non-anticipative, therefore it loops in some oriented Feynman diagrams, which allows improved estimates. Moreover, one uses the following relation in the context of the KPZ equation:
\begin{equs} \label{key_identity_parabolic}
	(\partial_x K * \partial_x K)(z) = \frac{1}{2} ( K(z) + K(-z) )
\end{equs}
where $ *$ is the space-time convolution. See \cite[Lemma 6.11]{HQ18}, where one has this identity up to a small error. Such an identity reflects the fact that the heat kernel is the fundamental solution of the heat equation. Equipped with this formalism, one is able to compute and check cancellations in \cite{Bru} for the generalised KPZ equation. 
This allows us to consider solutions that are ``geometric", meaning that they satisfy the chain rule property. These identities have been pushed further in \cite{Mate19} with general integration by parts identities. They are needed in the context of quasilinear SPDEs to make sure that one gets only local counter-terms in the solution at the level of the renormalised equation (see \cite{BGN24} for a general statement). 
These cancellations have been understood at a more conceptual level in \cite{BGHZ}, where the chain rule symmetry has been characterised as the kernel of a linear map defined on decorated trees. 
The dimension of this kernel and its basis is computed only for space-time white noise in \cite{BGHZ}. The full subcritical regime is treated in a systematic way via operad theory and homological algebra in \cite{BD24}. The specific case of dimension one is considered in \cite{BB24} with multi-indices and elementary techniques.

In this work, we suggest a formalism that is analogous to the first steps of the parabolic counterpart, and we provide an efficient graphical formalism for computing cancellations. Firstly, we have to propose an equivalent of \eqref{key_identity_parabolic} for dispersive equations.
We use two key identities in this paper. The first one for nonlinear Schrödinger equation is given in Proposition~\ref{covariance_2} by:
	\begin{equs} \label{key_identity_NLS}
		e^{i(s-t) k^2} =  	\mathbb{E}( e^{-it k^2}\eta_{k} \overline{e^{-is k^2}\eta_{k}} ),
	\end{equs}
where the $\eta_k$ are i.i.d Gaussian complex random variables.
The second identity is derived in the context of Wave equation given in  Proposition~\ref{covariance} by:
	\begin{equs} \label{key_identity_Wave}
		\frac{\sin((t-t') \langle n \rangle)}{\langle n  \rangle} = - \partial_t	\mathbb{E}( v_n(t) v_{-n'}(t') ) = \partial_{t'}	\mathbb{E}( v_n(t) v_{-n'}(t') ),
	\end{equs}
where $ v_n(t) $ is the solution of the linear wave equation with a random initial data. 
Both identities rewrite the fundamental kernel associated with the equation into an expectation, which directly relates the kernel to the random initial data. 
From the point of view of the associated graphical rules, this is somewhat different from \eqref{key_identity_parabolic}. Indeed, \eqref{key_identity_NLS} and \eqref{key_identity_Wave} could be understood graphically as splitting one edge associated with the fundamental kernel in two edges connected by an expectation, whereas  \eqref{key_identity_parabolic} changes two edges into one via the space-time convolution.
Below we give a brief example illustrating  \eqref{key_identity_NLS}:
	\begin{equs} \label{ex_intro}
		\begin{tikzpicture}[scale=0.2,baseline=-5]
			\coordinate (root) at (0,-1);
			\coordinate (right) at (2,2);
			\coordinate (center) at (0,2);
			\coordinate (left) at (-2,2);
			\coordinate (leftr) at (0,5);
			\coordinate (leftc) at (-2,5);
			\coordinate (leftl) at (-4,5);
			\draw[kernels2] (right) -- (root);
			\draw[symbols,tinydots] (left) -- (root);
			\draw[kernels2] (center) -- (root);
			\draw[kernels2,tinydots] (leftc) -- (left);
			\draw[kernels2,tinydots] (leftr) -- (left);
			\draw[kernels2] (leftl) -- (left);
			\node[not] (rootnode) at (root) {};
			\node[not] (rootnode) at (left) {};
			\node[var] (rootnode) at (right) {\tiny{$ k_{\tiny{2}} $}};
			\node[var] (rootnode) at (center) {\tiny{$ k_{\tiny{1}} $}};
			\node[var] (rootnode) at (leftr) {\tiny{$ k_{\tiny{1}} $}};
			\node[var] (rootnode) at (leftl) {\tiny{$ k_{\tiny{4}} $}};
			\node[var] (trinode) at (leftc) {\tiny{$ k_5 $}};
		\end{tikzpicture} \longrightarrow  \begin{tikzpicture}[scale=0.2,baseline=-5]
			\coordinate (root) at (0,-1);
			\coordinate (rightc) at (1,2);
			\coordinate (right) at (3,2);
			\coordinate (leftc) at (-1,2);
			\coordinate (left) at (-3,2);
			\draw[kernels2,tinydots] (right) -- (root);
			\draw[kernels2] (left) -- (root);
			\draw[kernels2,tinydots] (rightc) -- (root);
			\draw[kernels2] (leftc) -- (root);
			\node[not] (rootnode) at (root) {};
			\node[var2] (rootnode) at (left) {\tiny{$ \ell_{\tiny{1}} $}};
			\node[var] (rootnode) at (leftc) {\tiny{$ k_{\tiny{4}} $}};
			\node[var] (trinode) at (rightc) {\tiny{$ k_{\tiny{5}} $}};
			\node[var] (trinode) at (right) {\tiny{$ k_{\tiny{1}} $}};
		\end{tikzpicture} \, \, \,	\begin{tikzpicture}[scale=0.2,baseline=-5]
			\coordinate (root) at (0,-1);
			\coordinate (right) at (2,2);
			\coordinate (center) at (0,2);
			\coordinate (left) at (-2,2);
			\draw[kernels2] (right) -- (root);
			\draw[kernels2,tinydots] (left) -- (root);
			\draw[kernels2] (center) -- (root);
			\node[not] (rootnode) at (root) {};
			\node[var] (rootnode) at (right) {\tiny{$ k_{\tiny{2}} $}};
			\node[var] (rootnode) at (center) {\tiny{$ k_{\tiny{1}} $}};
			\node[var2] (rootnode) at (left) {\tiny{$ \ell_{\tiny{1}} $}};
		\end{tikzpicture}
		\end{equs}
The decorated trees above correspond to oscillatory integrals where the decorations on the leaves could be associated with some pairings. The blue edge associated with $ e^{i(s-t) k^2} $ is split into two brown edges and we color the new leaves in green to show that they are different from the other leaves.
For a detailed definition of these decorated trees see Section \ref{Sec::2}. We want to stress that this point of view carries similarities to the molecule formulation
in Deng-Hani developed in \cite{DH2301}, where atoms are connected to each other either by leaf-pair bonds
(LP bonds), or by parent-child bonds (PC bonds). This labelling is added to the bonds of the molecule.  The PC bonds correspond precisely to branching nodes and in our formalism we use the green color in \eqref{ex_intro}. Then, one can observe that two trees cancel out if they  have the same molecule picture, but with opposite assignment
of PC or LP on certain bonds. This switching between PC and LP  corresponds in our case to switching the green color in \eqref{switching_nodes}.

The identities \eqref{key_identity_NLS} and \eqref{key_identity_Wave} are applied in some Feynman diagrams that are constructed by taking the expectation of the iterated integrals produced by the Duhamel formulation (integral representation of the equation) of these dispersive equations. The randomness in these iterated integrals is always coming from the initial data.
For describing these objects, we use the decorated trees formalism introduced in \cite{BS} for deriving low regularity schemes for a large class of dispersive equations.  We have already  presented some of these decorated trees in \eqref{ex_intro}. The main idea of this coding is to put decorations on the edges and the nodes of the trees for describing many analytical aspects of the Fourier iterated integrals arising from the Duhamel's formula. It is close in spirit to what has been developed in \cite{BHZ} for singular SPDEs with the aim of covering many equations.
 This formalism has also been used in the context of wave turbulence for the discretisation of the second moment of the Fourier coefficient of the solution in \cite{ABBS24}. Then, with this formalism, one may want to run some Hopf algebra structures for systematising the computation as it has been done in \cite{BS} for the local error analysis. 

It turns out that the natural structures on these iterated integrals are the shuffle algebra and the arborification maps. Indeed, if one has a product of integrals over a simplex, it is possible to rewrite it as a linear combination of iterated integrals over a bigger simplex. For example, one has
\begin{equs} \label{shuffle_identity}
& \ h(t) \times	\int_{0< t_1  < t}  f(t_1)dt_1  \times 	\int_{0< t_2 < t} g(t_2) dt_2 \\  & = h(t) \int_{0< t_1 < t_2 < t} f(t_1) g(t_2) dt_1 dt_2  + h(t) \int_{0< t_2 < t_1 < t} f(t_1) g(t_2) dt_2 dt_1.
\end{equs}
In our context, the functions $f, g$ and $h$ will depend on various frequencies. One can rewrite \eqref{shuffle_identity} into
\begin{equs} \label{identity_two_letters}
(\Pi T)(t) = (\Pi	\begin{tikzpicture}[scale=0.2,baseline=-5]
	\coordinate (root) at (0,-1);
	\coordinate (right) at (-1.5,2);
	\coordinate (left) at (1.5,2);
	\draw[symbols] (right) -- (root);
	\draw[symbols] (left) -- (root);
		\node[var] (rootnode) at (root) {\tiny{$a_3 $}};
	\node[var] (rootnode) at (left) {\tiny{$a_2 $}};
	\node[var] (trinode) at (right) {\tiny{$ a_1 $}};
\end{tikzpicture})(t)   = (\Pi^A a_1 a_2 a_3)(t) + (\Pi^A a_2 a_1 a_3)(t)
\end{equs}
where words are over some alphabet $A$ (here the letters are ordered from the smallest time variable to the biggest, so $a_1$ encodes $f$, $a_2$ encodes $g$ and $a_3$ encodes $h$). They are used for describing iterated integrals over a simplex and decorated trees represent products of these integrals The map $ \Pi $ and $\Pi_A$ interpret the combinatorial objects, decorated trees and words, as integrals. One can write systematically \eqref{identity_two_letters} via an arborification type map $\mathfrak{a}$ a morphism between words and trees:
\begin{equs} \label{identity_two_letters_2}
	(\Pi T)(t) = (\Pi^A \mathfrak{a}(T))(t), \quad \mathfrak{a}(T) = a_1 a_2 a_3 + a_2 a_1 a_3.
\end{equs}
The arborification map is defined by induction by putting the root as the rightmost letter and shuffling its value on the various trees attached to the root. One has
\begin{equs}
	\mathfrak{a}(T) =  \left( \mathfrak{a}(\begin{tikzpicture}[scale=0.2,baseline=-5]
		\coordinate (root) at (0,-0.5);
		\node[var] (rootnode) at (root) {\tiny{$a_1 $}};
	\end{tikzpicture}) \shuffle \mathfrak{a}(\begin{tikzpicture}[scale=0.2,baseline=-5]
	\coordinate (root) at (0,-0.5);
	\node[var] (rootnode) at (root) {\tiny{$a_2 $}};
\end{tikzpicture}) \right)  a_3, \quad \mathfrak{a}(\begin{tikzpicture}[scale=0.2,baseline=-5]
	\coordinate (root) at (0,-0.5);
	\node[var] (rootnode) at (root) {\tiny{$a_i $}};
\end{tikzpicture}) = a_i, \quad a_1 \shuffle a_2 = a_1 a_2 + a_2 a_1.
\end{equs} 
The map could also be defined via the Butcher-Connes-Kreimer coproduct (\cite{Butcher,CK1,CK2}).
The arborification has been used \cite{Br24} for rewriting the Poincaré-Dulac normal form proposed in \cite{GKO13} for dispersive equations. It was introduced by Ecalle for the study of dynamical systems (see \cite{EV04,FM}). It appears in the context of numerical analysis in \cite{Murua2006}). It is also a natural map connecting geometric rough paths (\cite{Lyons98,Gub04}) and branched rough paths \cite{Gub06,HK15}. 

Our main result is to rewrite iterated integrals that appear in \cite{DH2301,BDNY24} via the arborification and the use of \eqref{key_identity_NLS} and \eqref{key_identity_Wave}. Below, we list our main propositions/theorems
\begin{itemize}
	\item Introduction of the decorated trees formalism: Definition~\ref{decorated_trees}.
	\item Key identities: Propositions \ref{covariance_2}, \ref{covariance}.
	\item Definition of the arborifications: \eqref{arborification_NLS}, \eqref{arborification_Wave}.
	\item Connection between $\Pi T$ and $\Pi^A \mathfrak{a}(T)$: Theorems \ref{main_theorem_NLS}, \ref{main_theorem_Wave}.
\end{itemize}

 One can summarise the main results in the following diagram
	\begin{equs}\label{diag_2} \begin{aligned}
			\xymatrix{ 
				\scriptsize{\text{Oscillatory integrals}} \ar[rr]^{\scriptsize{\text{Analysis}}} 	&& \scriptsize{\text{Cancellations}}
				\\ \scriptsize{\text{Decorated trees}}
				\ar[r]^{\mathfrak{a}} \ar[u]^{\Pi} & \scriptsize{\text{Words}}   \ar[ul]^{\Pi^A}  
				\ar[ur]^{\scriptsize{\text{Algebra}}}  
			}
		\end{aligned}
	\end{equs}
	Let us briefly explain the diagram above. One has a commutation property given by $ \Pi T = \Pi^A \mathfrak{a}(T) $ where $T$ is a decorated tree. The arborification map $ \mathfrak{a} $  reflects the analytical transformations given by Propositions \ref{covariance_2} and  \ref{covariance}.
	Then, one wants to compute the cancellation using only algebraic/combinatorial arguments that correspond to the analytical computations. This is the object of the second part of the diagram, where one wants to go from oscillatory integrals or words to the cancellations. The idea is to replace
	analytical operations, which quickly become intractable in very large oscillatory integrals,  with  algebraic manipulations. We show that they can be obtained by working with words which are a simpler combinatorial set in comparison to decorated trees. Let us stress that the second commutative part of the diagram cannot be phrased as a general theorem as it very much depends on the examples and in which order one applies the rules.

Then, from the representation with words, we show cancellations between different terms. One can notice that the arguments diverge at this point. For nonlinear Schrödinger Equation, one has to be in some frequency domain when one can switch some nodes. Also, some words will be small and negligible.  In Proposition \ref{example_computation_1}, we prove one of the simplest cancellations when one switches colors. The validity of this combinatorial approach is justified by \cite[Lemma 7.1, (b)]{DH2301}. 

 For wave equation, the computation relies on integration by parts as \eqref{key_identity_Wave} makes appear a derivative in time. In Proposition\ref{example_computation_2}, we compute the main cancellation of \cite{BDNY24} via the words formalism. One can easily track the different steps of the computation.
  These two examples allow us to state  a Metatheorem below that reflects our computations.
 \begin{meta} Cancellations for dispersive PDEs with random initial data could be understood via words and some well-chosen arborification map.
 	\end{meta}
 
 We have checked this Metatheorem on two dispersive equations for two different contexts. One can hope to see this formalism extremely useful for cancellations.
 For the moment, understanding the cancellations in a more conceptual way seems out of reach without the identification of some symmetries as in the context of parabolic SPDEs or some new arguments. The formalism of words and arborification could also be useful for understanding the recent combinatorics introduced for random tensors and in Wave turbulence.
 
Finally, let us outline the paper by summarising the content of its sections. In Section~\ref{Sec::2}, we focus on the cubic nonlinear Schrödinger Equation in the context of Wave turbulence. The aim of this section is to provide a framework for computing the cancellations observed in \cite{DH2301}. We start by introducing in Definition~\ref{decorated_trees} the formalism of decorated trees coming from \cite{BS} which was partially inspired by singular SPDEs (see \cite{reg,BHZ}). This formalism is used in both sections as it is quite generic and go beyond one single equation. We interpret these decorated trees as oscillatory iterated integrals in \eqref{decorated_trees} via a map $\Pi$. We extend this map in \eqref{general_pi_1} where the expectation is taken among pairs of random initial data. These are the objects of interest in \cite{DH2301} where one wants to compute the second moment of the Fourier coefficient of the solution. At the combinatorial side, we consider decorated trees with a pairing.  Then, we explain the first cancellation observed in  \cite{DH2301} by writing explicitly the iterated integrals and we start to propose a more general framework. It is built upon Proposition~\ref{covariance_2} that allows splitting one edge into two components with an expectation. 
These new expectations are of different nature from the other ones. Therefore, we encode them via a green color on the leaves (see \eqref{example_color}).
The main new formalism is words on an alphabet $A$ whose letters are given by \eqref{letters_1} and could also have green colors for the special pairing. These words are equipped with a shuffle product $ \shuffle $ (see \eqref{shuffle_def}) and an arborification map $\mathfrak{a}$ (see \eqref{arborification_NLS}) that rewrites a decorated tree into a linear combination of words from which it is easy to see the cancellations. This arborification implements in an abstract way an iteration of Proposition~\ref{covariance_2}. One can also define it recursively via a Butcher-Connes-Kreimer type coproduct in \eqref{arbo_new} coproduct that we denote $\Delta_{\text{\tiny{NLS}}}$. The analytical connection between decorated trees and words is made precise in Theorem \ref{main_theorem_NLS}. As a consequence, one can perform all the computations on words for computing the cancellations in \cite{DH2301}. The main arguments are switching colors due to frequencies that are very close (see \eqref{switching_nodes}) and checking that some words give a small contribution (see \eqref{small_word}). We start with the simplest cancellation in Proposition\ref{example_computation_1} and then we proceed with more complex cases.

In Section~\ref{Sec::3}, we use the same formalism of decorated trees and words but for Wave equation. The random initial data is replaced by the solution of the parabolic heat equation given by \eqref{parabolic_solution}. This induces a different splitting of the wave propagator in Proposition \ref{covariance} with a derivative in time. At the level of the letters, it is encoded via a green edge (see \eqref{letter_type}). The map $\Pi$ given in \eqref{general_pi_2} has a slightly different definition.
Then, one also rewrites the arborification in \eqref{arborification_Wave} inductively and also with a new coproduct in \eqref{arbo_new}. The main result is Theorem \ref{main_theorem_Wave}, a variant of the previous theorem adapted to the wave context. Then, we finish the section with Proposition \ref{example_computation_2} which computes the main cancellation of \cite{BDNY24}. These computations involve integration by parts identities for making some cancellation on the words appear. What remains is either small or removed by some renormalisation.

\subsection*{Acknowledgements}

{\small
	Y. B. gratefully acknowledges funding support from the European Research Council (ERC) through the ERC Starting Grant Low Regularity Dynamics via Decorated Trees (LoRDeT), grant agreement No.\ 101075208.
	Views and opinions expressed are however those of the author(s) only and do not necessarily reflect those of the European Union or the European Research Council. Neither the European Union nor the granting authority can be held responsible for them.
}

\section{Wave turbulence}
\label{Sec::2}
In this section, we analyse the cancellations obtained in \cite{DH23,DH2301}.
They consider the following Schrödinger equation:
\begin{equs} \label{equ_schrodinger}
	( \partial_t  - i \Delta ) u = i \mu^2 |u|^2 u, \quad u(0,x) = v(x).
\end{equs}
where $x \in \mathbb{T}^d_L = [0,L]^d$. The random initial data $ v $ is given by
\begin{equs}
	v(x) = \frac{1}{L^d} \sum_{k \in \mathbb{Z}^d_L} v_k e^{2 \pi ikx}, \quad v_{k} = \sqrt{w_k} \eta_{k}
\end{equs}
where $ \mathbb{Z}^d_L = (L^{-1} \mathbb{Z})^d $ and $ w : \mathbb{R}^d  \rightarrow [0, + \infty)$ is a given Schwartz function. The $\eta_k$ are i.i.d centred complex Gaussian random variables satisfying for $k, \ell \in \mathbb{Z}^d_L$
\begin{equs}
	\mathbb{E}(|\eta_k|^2) = 1, \quad 	\mathbb{E}(\eta_k \eta_{\ell}) = 0.
\end{equs}
Equation \eqref{equ_schrodinger} can be rewritten in Duhamel form as
\begin{equs}
	u(t) = e^{it \Delta} v  + i \mu^2\int_0^t e^{i(t-s) \Delta} |u(s)|^2 u(s)ds. 
	\end{equs}
In Fourier space, one has
\begin{equs}
	u_k(t) = e^{-it k^2} v_k  + i \mu^2 \sum_{k = -k_1 + k_2 + k_3} \int_0^t e^{-i(t-s) k^2} \bar{u}_{k_1}(s) u_{k_2}(s) u_{k_3}(s) ds 
\end{equs}
where pointwise product in physical space $ |u|^2 u $ is sent to convolution product in Fourier space. Moreover, $ e^{i t \Delta} $ is sent to $ e^{-itk^2} $. We  want to encode the approximation of $ u_k $ up to order $r$ in time via Duhamel iterations. For this, we need to introduce decorated trees for encoding the various iterated integrals. 

\begin{definition}		\label{decorated_trees}
	We define the set of decorated trees $\mcT$ as elements of the form $T_\mfe^{\mff}$ where 	
	\begin{itemize}	
		\item $ T $ is a non-planar rooted tree (a rooted tree where the order of node's children
			does not matter) with root node $ \varrho_T $, edge set $ E_T $, node set $ N_T $ and leaves set $L_T$. 	
		\item $ \mfe:E_T\rightarrow \{ \mathfrak{t}_1, \mathfrak{t}_2 \} \times \{ 0,1 \} $  are edge decorations. We suppose that all the edges of $ T_\mfe^{\mff} $ connected to leaves are decorated by $ (\mathfrak{t_1},p) $ with $p \in \lbrace 0,1 \rbrace$ and that all the other edges are decorated by $ (\mathfrak{t_2},p) $.
		\item $ \mff:N_T\setminus\{\varrho_T\}\rightarrow \mathbb{Z}^d $ are node decorations  satisfying the relation for every inner node $ u $
		\begin{equs} \label{frequencies_identity}
		(-1)^{\mathfrak{c}(e_u)}	\mff(u)=\sum_{e = (u,v)\in E_T}  	(-1)^{\mathfrak{c}(e)} \mff(v)
		\end{equs}
		where $ \mathfrak{e}(e) = (\mathfrak{t}(e),\mathfrak{c}(e)) $, $ e_u$ is the edge outgoing $u$ of the form $ (w,u) $. Here, by outgoing we mean that $w$ is the parent of $u$, the closest node to $u$ lying on the path to the root of the tree. From this definition, one
		can see that the node decorations  $(\mff(u))_{u \in L_T}$ determine the decorations of the inner nodes.
		We suppose that the root $ \varrho_T $ is not decorated.
	\end{itemize}
\end{definition}

  One difference from the decorated trees used in \cite{DH2301,DH2311}, is that the decorated  trees given in Definition \ref{decorated_trees} are not planar. Planarity is used for encoding the complex conjugate and in our case the edge decoration in $ \{ 0,1 \} $ allows us to encode it. We do not put decorations at the root as we work with planted trees and in the recursive analytical interpretation of this object given by $\Pi$ in \eqref{Pi_def}, a decoration at the root is unnecessary.

We define $ \CH $ as the linear span of $\CT$ and by $\CF$ the linear span of forests formed of decorated trees in $\CT$. These decorated trees and forests are quite generic and can be used for various dispersive PDEs as we use them for two different applications in the present paper. There is no restriction on the number of edges connected to a node. In practice, one always works with a small subclass of these decorated trees, those which could be associated with iterated integrals coming from the Duhamel formulation. 
This was first introduced in \cite[Section 4.1]{BS}.
 The empty tree is denoted by $ \one $.
 One can use a symbolic notation for these trees described below. An  edge decorated by $ o \in \{ \mathfrak{t}_1, \mathfrak{t}_2 \} \times \{ 0,1 \} $ is denoted by $ \mathcal{I}_{o} $. The symbol $ \mathcal{I}_{o} (\lambda_{k}
\cdot) : \CH \rightarrow \CH $ is viewed as the operation that grafts the root of a tree to a new root via an edge decorated by $ o $. The symbol $ \lambda_k $ is to stress that the old root will be decorated by $k$.  If the condition \eqref{frequencies_identity} is not
satisfied on the argument, then $ \mathcal{I}_{o} (\lambda_{k}
\cdot) $  gives zero.
 One important product on decorated trees is the tree product between two decorated trees $T_1$ and $T_2$ that we denote by $ T_1T_2 $. It  consists into merging the roots of the two trees. Below, we give an example
\begin{equs}
T_1 =  \mathcal{I}_{(\mathfrak{t}_1,1)}(\lambda_{k_1}) = 	\begin{tikzpicture}[scale=0.2,baseline=-5]
		\coordinate (root) at (0,2);
		\coordinate (tri) at (0,-1);
		\draw[kernels2,tinydots] (tri) -- (root);
		\node[var] (rootnode) at (root) {\tiny{$ k_1 $}};
		\node[not] (trinode) at (tri) {};
	\end{tikzpicture}, \quad T_2 =  \mathcal{I}_{(\mathfrak{t}_1,0)}(\lambda_{k_2}) = 	\begin{tikzpicture}[scale=0.2,baseline=-5]
	\coordinate (root) at (0,2);
	\coordinate (tri) at (0,-1);
	\draw[kernels2] (tri) -- (root);
	\node[var] (rootnode) at (root) {\tiny{$ k_2 $}};
	\node[not] (trinode) at (tri) {};
\end{tikzpicture}, \quad T_3 =  \mathcal{I}_{(\mathfrak{t}_1,0)}(\lambda_{k_3}) = 	\begin{tikzpicture}[scale=0.2,baseline=-5]
\coordinate (root) at (0,2);
\coordinate (tri) at (0,-1);
\draw[kernels2] (tri) -- (root);
\node[var] (rootnode) at (root) {\tiny{$ k_3 $}};
\node[not] (trinode) at (tri) {};
\end{tikzpicture}
\\ T_1T_2T_3 = \CI_{(\mathfrak{t}_1,1)}(\lambda_{k_1})\CI_{(\mathfrak{t}_1,0)}(\lambda_{k_2})\CI_{(\mathfrak{t}_1,0)}(\lambda_{k_3}) = \begin{tikzpicture}[scale=0.2,baseline=-5]
	\coordinate (root) at (0,-1);
	\coordinate (right) at (2,2);
	\coordinate (center) at (0,2);
	\coordinate (left) at (-2,2);
	\draw[kernels2] (right) -- (root);
	\draw[kernels2,tinydots] (left) -- (root);
	\draw[kernels2] (center) -- (root);
	\node[not] (rootnode) at (root) {};
	\node[var] (rootnode) at (right) {\tiny{$ k_{\tiny{3}} $}};
	\node[var] (rootnode) at (center) {\tiny{$ k_{\tiny{2}} $}};
	\node[var] (rootnode) at (left) {\tiny{$ k_{\tiny{1}} $}};
\end{tikzpicture}.
\end{equs}
We have used brown edges for denoting edges decorated by $ (\mathfrak{t_1},0) $ and brown dotted edges for $ (\mathfrak{t_1},1) $. If we replace $\mathfrak{t}_1$ by $ \mathfrak{t}_2 $ then the edges will be blue. 
One has
\begin{equs} 
	T_4 =  \begin{tikzpicture}[scale=0.2,baseline=-5]
		\coordinate (root) at (0,0);
		\coordinate (tri) at (0,-2);
		\coordinate (t1) at (-2,2);
		\coordinate (t2) at (2,2);
		\coordinate (t3) at (0,3);
		\draw[kernels2,tinydots] (t1) -- (root);
		\draw[kernels2] (t2) -- (root);
		\draw[kernels2] (t3) -- (root);
		\draw[symbols] (root) -- (tri);
		\node[not] (rootnode) at (root) {};t
		\node[not] (trinode) at (tri) {};
		\node[var] (rootnode) at (t1) {\tiny{$ k_{\tiny{1}} $}};
		\node[var] (rootnode) at (t3) {\tiny{$ k_{\tiny{2}} $}};
		\node[var] (trinode) at (t2) {\tiny{$ k_3 $}};
	\end{tikzpicture}
	=	  \CI_{(\mathfrak{t}_2,0)}(\lambda_{k}\CI_{(\mathfrak{t}_1,1)}(\lambda_{k_1})\CI_{(\mathfrak{t}_1,0)}(\lambda_{k_2})\CI_{(\mathfrak{t}_1,0)}(\lambda_{k_3}))
\end{equs}
where $ k = -k_1 + k_2 + k_3 $. One can compute iterated integrals associated to these decorated trees via a map $\Pi$ defined recursively below:
\begin{equs}
	\label{Pi_def}
	\begin{aligned}
	(\Pi \CI_{(\mathfrak{t}_1,0)}(\lambda_k ))(t) & = e^{-itk^2} \eta_k \sqrt{w_k}, \quad (\Pi \CI_{(\mathfrak{t}_1,1)}(\lambda_k ))(t) = \overline{(\Pi \CI_{(\mathfrak{t}_1,0)}(\lambda_k ))(t)},
	\\ (\Pi \CI_{(\mathfrak{t}_2,p)}(\lambda_k T))(t) & =(-1)^{p} i \mu^2\int_{0}^t e^{ (-1)^{p}i (s-t) k^2}
	(\Pi T)(s) ds, \\ (\Pi T_1 T_2)(t) & = (\Pi T_1 )(t) (\Pi  T_2)(t).
	\end{aligned}
\end{equs}
By definition, the map $\Pi$ is a character sending the tree product to the pointwise product.
As an example, one has 
\begin{equs}
	(\Pi T_4)(t) = i \mu^2  \bar{\eta}_{k_1} \sqrt{w_{k_1}} \prod_{j=2}^3
	\eta_{k_j} \sqrt{w_{k_j}} \int_{0}^t e^{ -i(t-s)k^2}
	\left(   e^{is (k_1^2 -k_2^2 - k_3^2)} \right) ds.
\end{equs}
If one iterates the Duhamel formulation in Fourier space, one produces a tree series where iterated integrals of the form $(
\Pi T)(t)$ appear. By truncating this expansion according to the number of blue edges which corresponds to time integrations, one gets the first step of the derivation of the low regularity schemes described in \cite[Sec. 4]{BS}.

We now assume that some of the leaves of a decorated tree come in pairs. This can be described as a subset of the pairs of terminal edges (edges connected to a leaf) in a given tree. We denote this subset as $\mathfrak{p}$ for a given decorated tree $ T_{\mathfrak{e}}^{\mff} $. We use the short hand notation: $ T_{\mathfrak{e},\mathfrak{p}}^{\mff}  $ for encoding this extra data. A pairing imposes that the leaves decorations are the  same. For example, below the leaves paired have the decorations $k_1$ and $\ell_1$:
\begin{equs}
	T_5 = \begin{tikzpicture}[scale=0.2,baseline=-5]
		\coordinate (root) at (0,-1);
		\coordinate (right) at (2,2);
		\coordinate (center) at (0,2);
		\coordinate (left) at (-2,2);
		\coordinate (leftr) at (0,5);
		\coordinate (leftc) at (-2,5);
		\coordinate (leftl) at (-4,5);
		\draw[kernels2] (right) -- (root);
		\draw[symbols,tinydots] (left) -- (root);
		\draw[kernels2] (center) -- (root);
		\draw[kernels2,tinydots] (leftc) -- (left);
		\draw[kernels2,tinydots] (leftr) -- (left);
		\draw[kernels2] (leftl) -- (left);
		\node[not] (rootnode) at (root) {};
		\node[not] (rootnode) at (left) {};
		\node[var] (rootnode) at (right) {\tiny{$ k_{\tiny{2}} $}};
		\node[var] (rootnode) at (center) {\tiny{$ k_{\tiny{1}} $}};
		\node[var] (rootnode) at (leftr) {\tiny{$ k_{\tiny{1}} $}};
		\node[var] (rootnode) at (leftl) {\tiny{$ k_{\tiny{4}} $}};
		\node[var] (trinode) at (leftc) {\tiny{$ k_5 $}};
	\end{tikzpicture}, \quad T_6 = \begin{tikzpicture}[scale=0.2,baseline=-5]
		\coordinate (root) at (0,-1);
		\coordinate (right) at (2,2);
		\coordinate (center) at (0,2);
		\coordinate (left) at (-2,2);
		\coordinate (centerr) at (2,5);
		\coordinate (centerc) at (0,5);
		\coordinate (centerl) at (-2,5);
		\draw[kernels2] (right) -- (root);
		\draw[kernels2,tinydots] (left) -- (root);
		\draw[symbols] (center) -- (root);
		\draw[kernels2] (centerc) -- (center);
		\draw[kernels2] (centerr) -- (center);
		\draw[kernels2,tinydots] (centerl) -- (center);
		\node[not] (rootnode) at (root) {};
		\node[not] (rootnode) at (center) {};
		\node[var] (rootnode) at (right) {\tiny{$ k_{\tiny{2}} $}};
		\node[var] (rootnode) at (left) {\tiny{$ \ell_{\tiny{1}} $}};
		\node[var] (rootnode) at (centerr) {\tiny{$\ell_{\tiny{1}} $}};
		\node[var] (rootnode) at (centerl) {\tiny{$ k_{\tiny{5}} $}};
		\node[var] (trinode) at (centerc) {\tiny{$ k_4 $}};
	\end{tikzpicture}
\end{equs}
where $-\ell_1 = -k_1-k_5 + k_4$. Now, the pairs have to be interpreted as covariances between two random initial data. They are coming from the computation of moments of these integrals. In Wave turbulence, the quantity of interest is the second moment of the Fourier coefficient of the solution: $\mathbb{E}(|u_k|^2)$.
Indeed, the previous iterated integrals can be viewed as multilinear maps in the Gaussian initial data $\eta_k$. We recall the formula for computing expectations of the product of random Gaussian variables which is called  Isserles' theorem.
Let $ I $ be a finite set and $ (X_{i})_{i \in I} $ a collection of centred jointly Gaussian random variables. Then 
\begin{equs} \label{wick}
	\mathbb{E}\left(\prod_{i \in I} X_i \right) = \sum_{\mathfrak{p} \in \mathcal{P}(I)} \prod_{\lbrace i,j \rbrace \in \mathfrak{p}} \mathbb{E}(X_i X_j)
\end{equs}
where $ \mathcal{P}(I) $ are partitions of $ I $ with two elements of $I$ in each block of the partition. The application of this formula to $\mathbb{E}(|u_k|^2)$ produces the pairings on the leaves.
We have chosen to consider a slightly more general framework where not all the leaves are paired but it contains the case of the full pairing of interest in \cite{DH2301}. For interpreting $ T_{\mathfrak{e},\mathfrak{p}}^{\mff}  $, one cannot use a recursive formula but has to proceed with an explicit one:
\begin{equs} \label{general_pi_1}
	\begin{aligned}
	(\Pi T_{\mathfrak{e},\mathfrak{p}}^{\mff})(t)  & =  \int  \prod_{e \in E_T^2} \one_{\lbrace  0<t_{e_-} < t_{e_+} \rbrace} (-1)^{\mathfrak{c}(e)} i \mu^2 \left( \exp((-1)^{\mathfrak{c}(e)}i(t_{e_-}-t_{e_+}) \mff(e_-)^2) \right)  \\ & \prod_{z \in E_T^{\mathfrak{p}}} \xi_{\mff(z_-)}^{\mathfrak{c}(z)}(t_{z_+})  \prod_{\lbrace x,y \rbrace \in \mathfrak{p}} 	\mathbb{E}( \xi_{\mff(x_-)}^{\mathfrak{c}(x)}(t_{x_+}) \xi_{\mff(y_-)}^{\mathfrak{c}(y)}(t_{y_+}) ) \prod_{e \in E_T^2} dt_{e_-}
	\end{aligned}
\end{equs}
where for $ e  \in E_T $, one has  $e = (e_+,e_-)  $, if $e_+ = \varrho_T$ then $t_{e_+} = t$. The set $ E_T^2 $ are the edges $e$ of $T_{\mfe,\mathfrak{p}}^{\mff}$ such that $ \mfe(e) = \mathfrak{t}_2 $ and $ E_T^{\mathfrak{p}} $ corresponds to the terminal edges that are not in $ \mathfrak{p} $.
One has
\begin{equs}
	\xi^0_{k}(t) = e^{-i t k^2} \eta_k \sqrt{w_k}, \quad \xi^1_{k}(t) = \overline{\xi^0_{k}(t)}.
\end{equs}
 Below, we illustrate this definition with two examples:
\begin{equs}
	(\Pi T_5)(t) &=\eta_{k_2}\eta_{k_4} \bar{\eta}_{k_5}  w_{k_1} \sqrt{w_{k_2}}  \sqrt{w_{k_4}}  \sqrt{w_{k_5}}
	\\ & (-i) \mu^2 e^{- it (k_1^2+k_2^2)   }  \int_0^t e^{ i(s-t) \ell_1^2}   e^{-i s(k_1^2+k_5^2-k_4^2)} ds, 
\\
	(\Pi T_6)(t) &=\eta_{k_2}\eta_{k_4} \bar{\eta}_{k_5}  w_{\ell_1} \sqrt{w_{k_2}}  \sqrt{w_{k_4}}  \sqrt{w_{k_5}}
	\\ & i \mu^2 e^{- it (\ell_1^2+k_2^2)   }  \int_0^t e^{ i(s-t) k_1^2}   e^{-i s(\ell_1^2+k_5^2-k_4^2)} ds.
\end{equs}
If we suppose that $ |k_1-\ell_1| \leq L^{-1} $, one can make the following identification up to a small error:
\begin{equs}
	w_{\ell_1} \approx w_{k_1}.
\end{equs}
One can notice from the explicit expression of the iterated integrals described above the following cancellation:
\begin{equs}
	\Pi(T_5 + T_6)  \approx 0.
\end{equs}
This is exactly the first of the three (families of) ``miraculous cancellations" appearing in \cite{DH2301}.
We want to derive a combinatorial formalism that explains this and the other cancellations. We first start with  a simple observation that rewrites any internal edge of the previous tree as a covariance between two edges.
\begin{proposition} \label{covariance_2} One has
	\begin{equs} \label{identity_key}
		e^{i(s-t) k^2} =  	\mathbb{E}(e^{-it k^2}\eta_{k}  \overline{e^{-is k^2}\eta_{k}} ).
	\end{equs}
\end{proposition}
\begin{proof}
	It is an immediate consequence of the definition of the noises $\eta_k$. Indeed, one has
	\begin{equs}
		\mathbb{E}(e^{-it k^2}\eta_{k} \overline{e^{-is k^2}\eta_{k}} ) = e^{i(s-t) k^2} \mathbb{E}(\overline{\eta_k} \eta_k) = e^{i(s-t) k^2}.
	\end{equs}
\end{proof}
In the sequel, we use the following notations:
\begin{equs}
	\hat{\xi}_{k}(t) =	\hat{\xi}_{k}^0(t) = e^{-i t k^2} \eta_k, \quad \hat{\xi}_{k}^1(t) = \overline{\hat{\xi}_{k}^0(t)}.
\end{equs}
The difference between $ \hat{\xi}_{k} $ and $ \xi_{k} $ is the factor $ \sqrt{w_k} $ which is missing.
We want to work with a set of decorated trees that can encode the fact that some of the terminal edges do not contain a factor $ \sqrt{w_k} $ after evaluating with a map $\Pi$. For this, we suppose that given a decorated tree $ T_{\mathfrak{e},\mathfrak{p}}^{\mff} $, one has $\mathfrak{p} = \mathfrak{p}_1 \cup \mathfrak{p}_2 $ where $\mathfrak{p}_1$ will correspond to the pairs that do not contain any $\sqrt{w_k}$. In the symbolic notation, we replace $\CI_{(\mathfrak{t}_1,p)}$ by $\hat{\CI}_{(\mathfrak{t}_1,p)}$. Pictorially, we color the leaves attached to these edges in green. As an example, one has
\begin{equs} \label{example_color} \begin{aligned}
	T_7 & = \hat{\CI}_{(\mathfrak{t}_1,0)}(\lambda_{k_1})\CI_{(\mathfrak{t}_1,0)}(\lambda_{k_2})  \CI_{(\mathfrak{t}_2,1)}( \CI_{(\mathfrak{t}_1,0)}(\lambda_{k_4}) \CI_{(\mathfrak{t}_1,1)}(\lambda_{k_5}) 
	\hat{\CI}_{(\mathfrak{t}_1,1)}(\lambda_{k_1})     )\\  &=  \begin{tikzpicture}[scale=0.2,baseline=-5]
		\coordinate (root) at (0,-1);
		\coordinate (right) at (2,2);
		\coordinate (center) at (0,2);
		\coordinate (left) at (-2,2);
		\coordinate (leftr) at (0,5);
		\coordinate (leftc) at (-2,5);
		\coordinate (leftl) at (-4,5);
		\draw[kernels2] (right) -- (root);
		\draw[symbols,tinydots] (left) -- (root);
		\draw[kernels2] (center) -- (root);
		\draw[kernels2,tinydots] (leftc) -- (left);
		\draw[kernels2,tinydots] (leftr) -- (left);
		\draw[kernels2] (leftl) -- (left);
		\node[not] (rootnode) at (root) {};
		\node[not] (rootnode) at (left) {};
		\node[var] (rootnode) at (right) {\tiny{$ k_{\tiny{2}} $}};
		\node[var2] (rootnode) at (center) {\tiny{$ k_{\tiny{1}} $}};
		\node[var2] (rootnode) at (leftr) {\tiny{$ k_{\tiny{1}} $}};
		\node[var] (rootnode) at (leftl) {\tiny{$ k_{\tiny{4}} $}};
		\node[var] (trinode) at (leftc) {\tiny{$ k_5 $}};
	\end{tikzpicture}
\end{aligned}
\end{equs}
and
\begin{equs}
		(\Pi T_7)(t) &=\eta_{k_2}\eta_{k_4} \bar{\eta}_{k_5}  \sqrt{w_{k_2}}  \sqrt{w_{k_4}}  \sqrt{w_{k_5}}
		\\ &  (-i) \mu^2 e^{- it (k_1^2+k_2^2)   }  \int_0^t e^{ i(s-t) \ell_1^2}   e^{-i s(k_1^2+k_5^2-k_4^2)} ds. 
	\end{equs}
Before applying Proposition \ref{covariance_2} on all the internal edges, we need to introduce an extra structure. We consider words on an alphabet $A$ whose letters are given by:
\begin{equs} \label{letters_1}
	\begin{tikzpicture}[scale=0.2,baseline=-5]
		\coordinate (root) at (0,2);
		\coordinate (tri) at (0,-1);
		\draw[kernels2] (tri) -- (root);
		\node[var] (rootnode) at (root) {\tiny{$ k_1 $}};
		\node[not] (trinode) at (tri) {};
	\end{tikzpicture}, \quad \begin{tikzpicture}[scale=0.2,baseline=-5]
	\coordinate (root) at (0,2);
	\coordinate (tri) at (0,-1);
	\draw[kernels2,tinydots] (tri) -- (root);
	\node[var] (rootnode) at (root) {\tiny{$ k_1 $}};
	\node[not] (trinode) at (tri) {};
\end{tikzpicture}, \quad  \begin{tikzpicture}[scale=0.2,baseline=-5]
\coordinate (root) at (0,-1);
\coordinate (right) at (2,2);
\coordinate (center) at (0,2);
\coordinate (left) at (-2,2);
\draw[kernels2] (right) -- (root);
\draw[kernels2,tinydots] (left) -- (root);
\draw[kernels2] (center) -- (root);
\node[not] (rootnode) at (root) {};
\node[var] (rootnode) at (right) {\tiny{$ \ell_{\tiny{3}} $}};
\node[var] (rootnode) at (center) {\tiny{$ \ell_{\tiny{2}} $}};
\node[var] (rootnode) at (left) {\tiny{$ \ell_{\tiny{1}} $}};
\end{tikzpicture},  \quad \begin{tikzpicture}[scale=0.2,baseline=-5]
\coordinate (root) at (0,-1);
\coordinate (rightc) at (1,2);
\coordinate (right) at (3,2);
\coordinate (leftc) at (-1,2);
\coordinate (left) at (-3,2);
\draw[kernels2,tinydots] (right) -- (root);
\draw[kernels2] (left) -- (root);
\draw[kernels2,tinydots] (rightc) -- (root);
\draw[kernels2] (leftc) -- (root);
\node[not] (rootnode) at (root) {};
\node[var] (rootnode) at (left) {\tiny{$ \ell_{\tiny{1}} $}};
\node[var] (rootnode) at (leftc) {\tiny{$ \ell_{\tiny{2}} $}};
\node[var] (trinode) at (rightc) {\tiny{$ \ell_{\tiny{3}} $}};
\node[var] (trinode) at (right) {\tiny{$ \ell_{\tiny{4}} $}};
\end{tikzpicture}.
\end{equs}
where for the last letter, one must have
\begin{equs}
	\ell_1 + \ell_2 -\ell_3 - \ell_4 = 0.
\end{equs}
They are also in $A$  letters where some leaves have been colored in green. We consider the words on this alphabet  that we denote by $T(A)$. We keep only the words that are compatible with a pairing on the leaves. We denote this space by $\mathcal{W}_A$. Below, we provide an example of such a word:
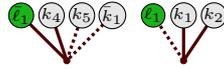
\begin{equs}
\begin{tikzpicture}[scale=0.2,baseline=-5]
		\coordinate (root) at (0,-1);
		\coordinate (rightc) at (1,2);
		\coordinate (right) at (3,2);
		\coordinate (leftc) at (-1,2);
		\coordinate (left) at (-3,2);
		\draw[kernels2,tinydots] (right) -- (root);
		\draw[kernels2] (left) -- (root);
		\draw[kernels2,tinydots] (rightc) -- (root);
		\draw[kernels2] (leftc) -- (root);
		\node[not] (rootnode) at (root) {};
		\node[var2] (rootnode) at (left) {\tiny{$ \ell_{\tiny{1}} $}};
		\node[var] (rootnode) at (leftc) {\tiny{$ k_{\tiny{4}} $}};
		\node[var] (trinode) at (rightc) {\tiny{$ k_{\tiny{5}} $}};
		\node[var] (trinode) at (right) {\tiny{$ k_{\tiny{1}} $}};
	\end{tikzpicture} \, \, \,	\begin{tikzpicture}[scale=0.2,baseline=-5]
		\coordinate (root) at (0,-1);
		\coordinate (right) at (2,2);
		\coordinate (center) at (0,2);
		\coordinate (left) at (-2,2);
		\draw[kernels2] (right) -- (root);
		\draw[kernels2,tinydots] (left) -- (root);
		\draw[kernels2] (center) -- (root);
		\node[not] (rootnode) at (root) {};
		\node[var] (rootnode) at (right) {\tiny{$ k_{\tiny{2}} $}};
		\node[var] (rootnode) at (center) {\tiny{$ k_{\tiny{1}} $}};
		\node[var2] (rootnode) at (left) {\tiny{$ \ell_{\tiny{1}} $}};
	\end{tikzpicture}
\end{equs}
with $ \ell_1 + k_4 -k_5 - k_1 =0 $. We introduce a product $ \shuffle $ on $T(A)$ called shuffle product. It is given  	inductively for two words $au$ and $ bv $ with $a,b \in A$ by:
\begin{equation}\label{shuffle_def}
\begin{aligned}
		&au \shuffle bv = a ( u \shuffle bv ) + b(au \shuffle v), \\
		&a\shuffle \one = \one \shuffle a = a.
\end{aligned}
\end{equation}
Here, $ \one $ denotes the empty word, the neutral for $ \shuffle $.
Now, we define a natural morphism between $ (\CH,\cdot_F) $ and $ (T(A), \shuffle) $ that sends the forest product $ \cdot_F $ to the shuffle product. 
 It is given inductively on the symbolic notation by
\begin{equs} \label{arborification_NLS}
	\begin{aligned}
\ &	\mathfrak{a}( \prod_{j=1}^n \CI_{(\mathfrak{t}_1,p_j)}(\lambda_{\ell_j})  \prod_{r=1}^m \CI_{(\mathfrak{t}_2,q_r)}(\lambda_{k_r} \tau_r))  
	\\ &=  ( \prod_{r=1}^m(-1)^{q_r} i) \left( \mathfrak{a}(\hat{\CI}_{(\mathfrak{t}_1,\bar{q}_1)}(\lambda_{k_1})\tau_1) \shuffle \cdots \shuffle \mathfrak{a}(\hat{\CI}_{(\mathfrak{t}_1,\bar{q}_m)}(\lambda_{k_m})\tau_m) \right)
	\\&\phantom{=\ \ } \prod_{j=1}^n \CI_{(\mathfrak{t}_1,p_j)}(\lambda_{\ell_j})  \prod_{r=1}^m \hat{\CI}_{(\mathfrak{t}_1,q_r)}(\lambda_{k_r} \tau_r)
	\end{aligned}
\end{equs} 
where $ \bar{q}_r = 1 $ if $q_r = 0$ and $0$ otherwise. Let us comment briefly on this recursive formula. The rightmost letter on the third line corresponds to the terminal edges attached to the root of the tree. One can observe that we have added terminal edges for each other edges connected to the root. They also appear as a factor inside the recursion. In the end, we have transformed edges of type $ \mathfrak{t}_2 $ into two edges of type $ \mathfrak{t}_1 $. This is a combinatorial version of a  repeated application of Proposition~\ref{covariance_2}.
Using this transformation, one can have the following words:
\begin{equs}
	\mathfrak{a} (\begin{tikzpicture}[scale=0.2,baseline=-5]
		\coordinate (root) at (0,-1);
		\coordinate (right) at (2,2);
		\coordinate (center) at (0,2);
		\coordinate (left) at (-2,2);
		\coordinate (leftr) at (0,5);
		\coordinate (leftc) at (-2,5);
		\coordinate (leftl) at (-4,5);
		\draw[kernels2] (right) -- (root);
		\draw[symbols,tinydots] (left) -- (root);
		\draw[kernels2] (center) -- (root);
		\draw[kernels2,tinydots] (leftc) -- (left);
		\draw[kernels2,tinydots] (leftr) -- (left);
		\draw[kernels2] (leftl) -- (left);
		\node[not] (rootnode) at (root) {};
		\node[not] (rootnode) at (left) {};
		\node[var] (rootnode) at (right) {\tiny{$ k_{\tiny{2}} $}};
		\node[var] (rootnode) at (center) {\tiny{$ k_{\tiny{1}} $}};
		\node[var] (rootnode) at (leftr) {\tiny{$ k_{\tiny{1}} $}};
		\node[var] (rootnode) at (leftl) {\tiny{$ k_{\tiny{4}} $}};
		\node[var] (trinode) at (leftc) {\tiny{$ k_5 $}};
	\end{tikzpicture})  & = - i  \begin{tikzpicture}[scale=0.2,baseline=-5]
		\coordinate (root) at (0,-1);
		\coordinate (rightc) at (1,2);
		\coordinate (right) at (3,2);
		\coordinate (leftc) at (-1,2);
		\coordinate (left) at (-3,2);
		\draw[kernels2,tinydots] (right) -- (root);
		\draw[kernels2] (left) -- (root);
		\draw[kernels2,tinydots] (rightc) -- (root);
		\draw[kernels2] (leftc) -- (root);
		\node[not] (rootnode) at (root) {};
		\node[var2] (rootnode) at (left) {\tiny{$ \ell_{\tiny{1}} $}};
		\node[var] (rootnode) at (leftc) {\tiny{$ k_{\tiny{4}} $}};
		\node[var] (trinode) at (rightc) {\tiny{$ k_{\tiny{5}} $}};
		\node[var] (trinode) at (right) {\tiny{$ k_{\tiny{1}} $}};
	\end{tikzpicture} \, \, \,	\begin{tikzpicture}[scale=0.2,baseline=-5]
		\coordinate (root) at (0,-1);
		\coordinate (right) at (2,2);
		\coordinate (center) at (0,2);
		\coordinate (left) at (-2,2);
		\draw[kernels2] (right) -- (root);
		\draw[kernels2,tinydots] (left) -- (root);
		\draw[kernels2] (center) -- (root);
		\node[not] (rootnode) at (root) {};
		\node[var] (rootnode) at (right) {\tiny{$ k_{\tiny{2}} $}};
		\node[var] (rootnode) at (center) {\tiny{$ k_{\tiny{1}} $}};
		\node[var2] (rootnode) at (left) {\tiny{$ \ell_{\tiny{1}} $}};
	\end{tikzpicture},
	\quad
	\mathfrak{a}(\begin{tikzpicture}[scale=0.2,baseline=-5]
		\coordinate (root) at (0,-1);
		\coordinate (right) at (2,2);
		\coordinate (center) at (0,2);
		\coordinate (left) at (-2,2);
		\coordinate (centerr) at (2,5);
		\coordinate (centerc) at (0,5);
		\coordinate (centerl) at (-2,5);
		\draw[kernels2] (right) -- (root);
		\draw[kernels2,tinydots] (left) -- (root);
		\draw[symbols] (center) -- (root);
		\draw[kernels2] (centerc) -- (center);
		\draw[kernels2] (centerr) -- (center);
		\draw[kernels2,tinydots] (centerl) -- (center);
		\node[not] (rootnode) at (root) {};
		\node[not] (rootnode) at (center) {};
		\node[var] (rootnode) at (right) {\tiny{$ k_{\tiny{2}} $}};
		\node[var] (rootnode) at (left) {\tiny{$ \ell_{\tiny{1}} $}};
		\node[var] (rootnode) at (centerr) {\tiny{$ \ell_{\tiny{1}} $}};
		\node[var] (rootnode) at (centerl) {\tiny{$ k_{\tiny{5}} $}};
		\node[var] (trinode) at (centerc) {\tiny{$ k_4 $}};
	\end{tikzpicture})  = i  \begin{tikzpicture}[scale=0.2,baseline=-5]
		\coordinate (root) at (0,-1);
		\coordinate (rightc) at (1,2);
		\coordinate (right) at (3,2);
		\coordinate (leftc) at (-1,2);
		\coordinate (left) at (-3,2);
		\draw[kernels2,tinydots] (right) -- (root);
		\draw[kernels2] (left) -- (root);
		\draw[kernels2,tinydots] (rightc) -- (root);
		\draw[kernels2] (leftc) -- (root);
		\node[not] (rootnode) at (root) {};
		\node[var] (rootnode) at (left) {\tiny{$ \ell_{\tiny{1}} $}};
		\node[var] (rootnode) at (leftc) {\tiny{$ k_{\tiny{4}} $}};
		\node[var] (trinode) at (rightc) {\tiny{$ k_{\tiny{5}} $}};
		\node[var2] (trinode) at (right) {\tiny{$ k_{\tiny{1}} $}};
	\end{tikzpicture} \, \, \,	\begin{tikzpicture}[scale=0.2,baseline=-5]
		\coordinate (root) at (0,-1);
		\coordinate (right) at (2,2);
		\coordinate (center) at (0,2);
		\coordinate (left) at (-2,2);
		\draw[kernels2] (right) -- (root);
		\draw[kernels2,tinydots] (left) -- (root);
		\draw[kernels2] (center) -- (root);
		\node[not] (rootnode) at (root) {};
		\node[var] (rootnode) at (right) {\tiny{$ k_{\tiny{2}} $}};
		\node[var2] (rootnode) at (center) {\tiny{$ k_{\tiny{1}} $}};
		\node[var] (rootnode) at (left) {\tiny{$ \ell_{\tiny{1}} $}};
	\end{tikzpicture}.
\end{equs}
The recursive definition of $ \mathfrak{a} $ is enough for understanding the cancellations . We still want to provide another natural definition that involves a Butcher-Connes-Kreimer type coproduct $ \Delta_{\text{\tiny{NLS}}} : \CF \rightarrow \CF \otimes \CF $, a different version of the one introduced in \cite{BS},  defined recursively by
\begin{equs} \label{BCK_new}
	\begin{aligned}
		\Delta_{\text{\tiny{NLS}}} \CI_{(\mathfrak{t}_1,p)}( \lambda_{k}  ) & = \left( \one \, \otimes \CI_{(\mathfrak{t}_1,p)}( \lambda_{k}  )   \right), \\ 
		\Delta_{\text{\tiny{NLS}}} \CI_{(\mathfrak{t}_2,p)}( \lambda_{k}  T ) & = \left( \id \, \otimes \CI_{(\mathfrak{t}_2,p)}( \lambda_{k}  \cdot ) \right) \Delta_{\text{\tiny{NLS}}} T \\ & + (-1)^p i  \   \hat{\CI}_{(\mathfrak{t}_1,\bar{p})}(\lambda_{k} )T   \otimes \hat{\CI}_{(\mathfrak{t}_1,p)}(\lambda_k).
	\end{aligned}
\end{equs}
and it is extended multiplicatively in the following sense:
\begin{equs}
	\Delta_{\text{\tiny{NLS}}} T_i &= \sum_{(T_i)} T_i^{(1)} \otimes T_i^{(2)}, \ \forall i \in \lbrace 1,2 \rbrace,\\ 
		\Delta_{\text{\tiny{NLS}}} (T_1 \cdot T_2) &= \sum_{(T_1),(T_2)} (T_1^{(1)} \cdot_F  T_2^{(1)}) \otimes (T_1^{(2)}
		\cdot T_2^{(2)}).
\end{equs}
We define the arborification map $ \mathfrak{a} $ on a decorated tree $T$ by
\begin{equs} \label{arbo_new_NLS}		\mathfrak{a}(  T )
	= \mathcal{M}_{\tiny{\text{c}}}\left( \mathfrak{a}  \otimes P_{A}   \right) \Delta_{\text{\tiny{NLS}}}  T,
\end{equs}
where $ P_A $ is the projection on the letters of $A$ and $ \mathcal{M}_c $ is given as
\begin{equs}
	\mathcal{M}_c(u \otimes v) = uv.
\end{equs}
Here, $ uv $ denotes the concatenation of the words $u$ and $v$.
The arborification $\mathfrak{a}$ is extended to a forest as a morphism sending the forest product $ \cdot_F $ to the shuffle product $ \shuffle $:
\begin{equs}
	\mathfrak{a}(F_1 \cdot_F F_2) = \mathfrak{a}(F_1) \shuffle \mathfrak{a} (F_2).
\end{equs}
The definition of the arborification in \eqref{arbo_new_NLS} via a coproduct has been introduced the first time in a different context in \cite{BCEF18}.
A similar arborification has been used in \cite{Br24} for describing Poincaré-Dulac Normal forms. We define an evaluation map on words of this alphabet given by
\begin{equs} \label{word_NLS}
	(\Pi^A w_{\mathfrak{p}})(t)  & = \int_{0 < t_1 <.. <t_n} \mu^{2n-2} \prod_{\lbrace x,y \rbrace \in \mathfrak{p}_1} 	\mathbb{E}( \hat{\xi}_{\mff(x_-)}^{\mathfrak{c}(x)}(t_{x_+}) \hat{\xi}_{\mff(y_-)}^{\mathfrak{c}(y)}(t_{y_+}) ) \\ & \prod_{z \in E_w^{\mathfrak{p}}} \xi_{\mff(z_-)}^{\mathfrak{c}(z)}(t_{z_+})  \prod_{\lbrace x,y \rbrace \in \mathfrak{p}_2} 	\mathbb{E}( \xi_{\mff(x_-)}^{\mathfrak{c}(x)}(t_{x_+}) \xi_{\mff(y_-)}^{\mathfrak{c}(y)}(t_{y_+}) )  \prod_{i=1}^{n-1} dt_i
\end{equs}
where we have denoted the pairings $ \mathfrak{p} $ in the subscript of the word $ w_{\mathfrak{p}} $ which has $n$ letters.
A time variable is associated to each letter denoted by $t_i$ with $t_n=t$. Here, $E_w^{\mathfrak{p}}$ denotes the edges in the letters of $w$ which are not paired.
 As an example, one has 
\begin{equs}
	(\Pi^A \mathfrak{a}(\begin{tikzpicture}[scale=0.2,baseline=-5]
		\coordinate (root) at (0,-1);
		\coordinate (right) at (2,2);
		\coordinate (center) at (0,2);
		\coordinate (left) at (-2,2);
		\coordinate (leftr) at (0,5);
		\coordinate (leftc) at (-2,5);
		\coordinate (leftl) at (-4,5);
		\draw[kernels2] (right) -- (root);
		\draw[symbols,tinydots] (left) -- (root);
		\draw[kernels2] (center) -- (root);
		\draw[kernels2,tinydots] (leftc) -- (left);
		\draw[kernels2,tinydots] (leftr) -- (left);
		\draw[kernels2] (leftl) -- (left);
		\node[not] (rootnode) at (root) {};
		\node[not] (rootnode) at (left) {};
		\node[var] (rootnode) at (right) {\tiny{$ k_{\tiny{2}} $}};
		\node[var] (rootnode) at (center) {\tiny{$ k_{\tiny{1}} $}};
		\node[var] (rootnode) at (leftr) {\tiny{$ k_{\tiny{1}} $}};
		\node[var] (rootnode) at (leftl) {\tiny{$ k_{\tiny{4}} $}};
		\node[var] (trinode) at (leftc) {\tiny{$ k_5 $}};
	\end{tikzpicture}))(t) 
 &  = -i  (\Pi^A \begin{tikzpicture}[scale=0.2,baseline=-5]
 	\coordinate (root) at (0,-1);
 	\coordinate (rightc) at (1,2);
 	\coordinate (right) at (3,2);
 	\coordinate (leftc) at (-1,2);
 	\coordinate (left) at (-3,2);
 	\draw[kernels2,tinydots] (right) -- (root);
 	\draw[kernels2] (left) -- (root);
 	\draw[kernels2,tinydots] (rightc) -- (root);
 	\draw[kernels2] (leftc) -- (root);
 	\node[not] (rootnode) at (root) {};
 	\node[var2] (rootnode) at (left) {\tiny{$ \ell_{\tiny{1}} $}};
 	\node[var] (rootnode) at (leftc) {\tiny{$ k_{\tiny{4}} $}};
 	\node[var] (trinode) at (rightc) {\tiny{$ k_{\tiny{5}} $}};
 	\node[var] (trinode) at (right) {\tiny{$ k_{\tiny{1}} $}};
 \end{tikzpicture} \, \, \,	\begin{tikzpicture}[scale=0.2,baseline=-5]
 	\coordinate (root) at (0,-1);
 	\coordinate (right) at (2,2);
 	\coordinate (center) at (0,2);
 	\coordinate (left) at (-2,2);
 	\draw[kernels2] (right) -- (root);
 	\draw[kernels2,tinydots] (left) -- (root);
 	\draw[kernels2] (center) -- (root);
 	\node[not] (rootnode) at (root) {};
 	\node[var] (rootnode) at (right) {\tiny{$ k_{\tiny{2}} $}};
 	\node[var] (rootnode) at (center) {\tiny{$ k_{\tiny{1}} $}};
 	\node[var2] (rootnode) at (left) {\tiny{$ \ell_{\tiny{1}} $}};
 \end{tikzpicture})(t)
\\ & =
 -i \mu^2 \int_{0}^t  \xi_{k_2}(t) \xi_{k_4}(s)  \overline{\xi_{k_5}(s)}\mathbb{E}(\overline{\hat{\xi}_{\ell_1}}(t)\hat{\xi}_{\ell_1}(s) ) \mathbb{E}(\xi_{k_1}(t)\overline{\xi_{k_1}}(s) ) ds
\end{equs}
The next theorem connects the two formalisms via decorated trees and words. 
 It allows us to switch from  decorated trees to words where one can perform the algebraic manipulations required  to compute the cancellations.
\begin{theorem} \label{main_theorem_NLS}
	One has for every decorated tree $T_{\mathfrak{e},\mathfrak{p}}^{\mathfrak{f}}$
	\begin{equs}
		(\Pi T_{\mathfrak{e},\mathfrak{p}}^{\mathfrak{f}})(t) = (\Pi^A \mathfrak{a}(T_{\mathfrak{e},\mathfrak{p}}^{\mathfrak{f}}) )(t).
	\end{equs}
	\end{theorem}
\begin{proof}
	We recall \eqref{general_pi_1}
	\begin{equs} 
		\begin{aligned}
			(\Pi T_{\mathfrak{e},\mathfrak{p}}^{\mff})(t)  & =   \int  \prod_{e \in E_T^2} (-1)^{\mathfrak{c}(e)} i \mu^2 \one_{\lbrace  0<t_{e_-} < t_{e_+} \rbrace} \left( \exp((-1)^{\mathfrak{c}(e)}i(t_{e_-}-t_{e_+}) \mff(e_-)^2) \right)  \\ & \prod_{z \in E_T^{\mathfrak{p}}} \xi_{\mff(z_-)}^{\mathfrak{c}(z)}(t_{z_+})  \prod_{\lbrace x,y \rbrace \in \mathfrak{p}} 	\mathbb{E}( \xi_{\mff(x_-)}^{\mathfrak{c}(x)}(t_{x_+}) \xi_{\mff(y_-)}^{\mathfrak{c}(y)}(t_{y_+}) ) \prod_{e \in E_T^2} dt_{e_-}.
		\end{aligned}
	\end{equs}
By applying Proposition \ref{covariance_2}, one gets
\begin{equs}
		\exp((-1)^{\mathfrak{c}(e)}i(t_{e_-}-t_{e_+}) \mff(e_-)^2) =  \mathbb{E}\left( \hat{\xi}_{\mff(e_-)}^{\mathfrak{c}(e)}(t_{e_+}) \overline{\hat{\xi}_{\mff(e_-)}^{\mathfrak{c}(e)}(t_{e_-})} \right).
\end{equs}
By substituting this relation into the identity for $ (\Pi T_{\mfe,\mathfrak{p}}^{\mff})(t)  $, one gets
\begin{equs}
	(\Pi T_{\mfe,\mathfrak{p}}^{\mff})(t)  & =   \int  \prod_{e \in E_T^2} (-1)^{\mathfrak{c}(e)} i \mu^2 \one_{\lbrace  0<t_{e_-} < t_{e_+} \rbrace} \mathbb{E}\left( \hat{\xi}_{\mff(e_-)}^{\mathfrak{c}(e)}(t_{e_+}) \overline{\hat{\xi}_{\mff(e_-)}^{\mathfrak{c}(e)}(t_{e_-})} \right)  \\ & \prod_{z \in E_T^{\mathfrak{p}}} \xi_{\mff(z_-)}^{\mathfrak{c}(z)}(t_{z_+})  \prod_{\lbrace x,y \rbrace \in \mathfrak{p}} 	\mathbb{E}( \xi_{\mff(x_-)}^{\mathfrak{c}(x)}(t_{x_+}) \xi_{\mff(y_-)}^{\mathfrak{c}(y)}(t_{y_+}) ) \prod_{e \in E_T^2} dt_{e_-}.
\end{equs}
We fix an order on the nodes $ \bar{N}_T $. Let $ n = \mathrm{Card}(\bar{N}_T ) $, then we define $  \mathfrak{S}(\bar{N}_T ) $ as the permutation on these nodes that respects the partial order given by the tree structure of $T$.
Then, one has 
\begin{equs}
	\ &	(\Pi T_{\mfe,\mathfrak{p}}^{\mff})(t)   = \sum_{\sigma \in \mathfrak{S}(\bar{N}_{T})}  \int_{t_{\sigma(u_1)} < ... < t_{\sigma(u_n)}} \prod_{\lbrace x,y \rbrace \in \mathfrak{p}_1} (-1)^{\mathfrak{c}(x)} i \mu^2 \mathbb{E}\left( \hat{\xi}_{\mff(x_-)}^{\mathfrak{c}(x)}(t_{x_+}) \overline{\hat{\xi}_{\mff(y_-)}^{\mathfrak{c}(y)}(t_{y_-})} \right)  \\ & \prod_{z \in E_T^{\mathfrak{p}_2}} \xi_{\mff(z_-)}^{\mathfrak{c}(z)}(t_{z_+})  \prod_{\lbrace x,y \rbrace \in \mathfrak{p}_2} 	\mathbb{E}( \xi_{\mff(x_-)}^{\mathfrak{c}(x)}(t_{x_+}) \xi_{\mff(y_-)}^{\mathfrak{c}(y)}(t_{y_+}) ) \prod_{i=1}^n dt_{\sigma(u_i)} 
\end{equs}
where $ \mathfrak{p}_1 $ corresponds to the new pairing coming from the edges in $E_T^2$. These edges are split into two edges. Then, one has $ \mathfrak{p}_2 = \mathfrak{p} $.
In the previous expression, one can recognise all the analytical expressions of the letters that appear in $ \mathfrak{a}(T_{\mfe,\mathfrak{p}}^{\mff}) $. The order of these letters is encoded in the permutation $\sigma$. Indeed, all the letters of the words appearing in  $ \mathfrak{a}(T_{\mfe,\mathfrak{p}}^{\mff}) $ corresponds to nodes in $ \bar{N}_T $. If we consider a word $ w $ in the decomposition  $ \mathfrak{a}(T_{\mfe,\mathfrak{p}}^{\mff}) $ all the other words can be recovered by permuting the letters of $w$ with an element $ \sigma \in \mathfrak{S}(\bar{N}_{T}) $.
\end{proof}

Making the identification $ w_{k_1} \approx w_{\ell_1} $ boils down to interchange the color of two pairs of leaves. This can be encoded via a linear map on words that we denote by $ \psi_{k_1,\ell_1} $. Below, we show one example of computation

 We now show how this formalism can lead to a pictorial proof for the cancellations in \cite{DH2301}. However, we first need to introduce the concept of ``harmless terms". Informally speaking, we say that a linear combination of words is \emph{harmless} if it can be bounded analytically without the need of any spurious cancellation, and we denote by $\mathscr H$ the set of harmless terms. More precisely, in this setting, we define a linear combination of words to be harmless if it can be bounded via \cite[Lemma 7.1]{DH2301}. Since the content of this lemma is quite technical, we now give an example of a harmless term under this definition, which corresponds to changing the colour of two pairs of leaves. More precisely, we define a linear map on words that we denote by $ \psi_{k_1,\ell_1} $, and swaps the colour of the terms decorated by $k_1$ and $l_1$ respectively. Below, we show one example of this map
\begin{equs} \label{switching_nodes}
	\psi_{k_1,\ell_1}( 
.
\end{equs}
As a consequence of \cite[Lemma 7.1, (b)]{DH2301}, we have the following: if $k_1, \ell_1$ appear on consecutive letters of a word $a$, then 
$$ a - \psi_{k_1,\ell_1}(a) \in \mathscr H.  $$ 
Then we can write the first cancellation in the following way 
\begin{proposition} \label{example_computation_1}
	Let $T_1 = (%
)$. Then $\mathfrak{a}(T_1) - \mathfrak{a}(T_2) \in \mathscr H$. 
\end{proposition}
\begin{proof}
 For convenience of notation, we write $x \approx y$ if $x - y \in \mathscr H$. We have that 
\begin{equs}
\mathfrak{a} (%
). 
\end{equs}
\end{proof}
 We now proceed to show the other two cancellations analogously. The second cancellation in \cite{DH2301} is between  
\begin{equs}
	T_1 = %
\end{equs}
where $\dotsc$ means that these two branches are connected to bigger trees. More precisely, this means that there exist trees $T, R$ such that 
$$ T_1 =  %
$$
In particular, we have that 
\begin{equation} \label{decomposition_concrete}
\mathfrak{a}(T_1) = 	
	\big[%
	),
\end{equation}
and similarly for $T_2$.

The cancellation we are looking for happens under the condition that the two node decorations at the base of each of the trees are the same (in the order shown above), i.e.\
$$ -k_1 + k_2+k_3 + r_2 - r_1 = -h_1+h_2+h_3, $$
and that the trees containing the above subtrees are otherwise identical.
We have that 
\begin{align*}
	\mathfrak{a} ( 
	%
\end{align*}
where the condition 
$$ -k_1 + k_2+k_3 + r_2 - r_1 = -h_1+h_2+h_3  $$
can be rewritten as 
$$ \ell_2 = \ell_4, \quad \ell_3 = \ell_6. $$
The last branches with a green node in the computation above belong in general to a bigger letter. We now make the choice 
$$h_1 = k_1, \quad  h_2 = k_2, \quad  h_3 = \ell_1 = k_3 + r_2 - r_1.$$ This way, we obtain that 
$$ a_1 = b_2, \quad a_2 = b_1, \quad a_3 = b_3. $$
In view of $\eqref{arborification_NLS}$ and $\eqref{decomposition_concrete}$, if $T_1$, $T_2$ are the two trees respectively, we have that there exist $u,v$ in the word algebra $\mathbb C[W_{[A]}]$ such that 
\begin{align*}
	\mathfrak{a}(T_1) &= i \big(a_1a_2 \shuffle a_3 \shuffle u\big)v   \\
	\mathfrak{a}(T_2) &= -i \big(b_1 \shuffle b_2b_3 \shuffle u\big)v= -i \big(a_2 \shuffle a_1a_3 \shuffle u\big)v.
\end{align*}
Therefore, from the definition of the shuffle product \eqref{shuffle_def}, we have
\begin{equs} \label{small_word}
	\mathfrak{a}(T_1) +  \mathfrak{a}(T_2) = i \big(a_3 a_1a_2) \shuffle u\big) v -  i\big(a_2a_1a_3 \shuffle u\big)v.
\end{equs}
In particular, we observe that the terms with the letters $a_1,a_2,a_3$ in this order cancel out. It turns out that from an analytical point of view, these are all the ``problematic" terms, in the sense that  
$$ \Pi^A(\mathfrak{a}(T_1) +  \mathfrak{a}(T_2)) $$
can be estimated directly using \cite[Lemma 7.1]{DH2301}, without exploiting any further cancellation.

The third cancellation can be observed similarly. It corresponds to the two terms 
$$  
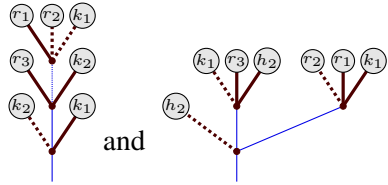
\begin{tikzpicture}[scale=0.2,baseline=-5]
	\coordinate (root) at (0,-3);
	\coordinate (c) at (0,-1);
	\coordinate (cl) at (-2,2);
	\coordinate (cc) at (0,2);
	\coordinate (cr) at (2,2);
	\coordinate (cll) at (-2,5);
	\coordinate (clc) at (0,5);
	\coordinate (clr) at (2,5);
	\coordinate (clcl) at (-2,8);
	\coordinate (clcc) at (0,8);
	\coordinate (clcr) at (2,8);
	\draw[symbols] (root) -- (c);
	\draw[kernels2,tinydots] (c) -- (cl);
	\draw[symbols,tinydots] (cc) -- (clc);
	\draw[symbols] (c) -- (cc);
	\draw[kernels2] (c) -- (cr);
	\draw[kernels2] (cc) -- (cll);
	\draw[kernels2] (cc) -- (clr);
	\draw[kernels2] (clc) -- (clcl);
	\draw[kernels2,tinydots] (clc) -- (clcr);
	\draw[kernels2,tinydots] (clc) -- (clcc);
	\node[not] (rootnode) at (c) {};
	\node[not] (rootnode) at (cc) {};
	\node[not] (rootnode) at (clc) {};
	\node[var] (rootnode) at (cl) {\tiny$k_{\tiny{2}}$};
	\node[var] (rootnode) at (cr) {\tiny$k_{\tiny{1}}$};
	\node[var] (rootnode) at (clr) {\tiny$k_{\tiny{2}}$};
	\node[var] (rootnode) at (clcr) {\tiny$k_{\tiny{1}}$};
	\node[var] (rootnode) at (cll) {\tiny$r_{\tiny{3}}$};
	\node[var] (rootnode) at (clcl) {\tiny$r_{\tiny{1}}$};
	\node[var] (rootnode) at (clcc) {\tiny$r_{\tiny{2}}$};
\end{tikzpicture}
\text{ and } \, \,
\begin{tikzpicture}[scale=0.2,baseline=-5]
	\coordinate (root) at (0,-3);
	\coordinate (rt) at (0,-1);
	\coordinate (l) at (-4,2);
	\coordinate (c) at (0,2);
	\coordinate (r) at (7,2);
	\coordinate (ll) at (-2,5);
	\coordinate (lc) at (0,5);
	\coordinate (lr) at (2,5);
	\coordinate (rl) at (5, 5);
	\coordinate (rc) at (7,5);
	\coordinate (rr) at (9,5);
	\draw[symbols] (root) -- (rt);
	\draw[kernels2,tinydots] (rt) -- (l);
	\draw[symbols] (rt) -- (r); 
	\draw[symbols] (rt) -- (c);
	\draw[kernels2,tinydots] (c) -- (ll);
	\draw[kernels2] (c) -- (lc);
	\draw[kernels2] (c) -- (lr);
	\draw[kernels2,tinydots] (r) -- (rl);
	\draw[kernels2] (r) -- (rc);
	\draw[kernels2] (r) -- (rr);
	\node[not] (rootnode) at (c) {};
	\node[not] (rootnode) at (r) {};
		\node[not] (rootnode) at (rt) {};
	\node[var] (rootnode) at (l) {\tiny$h_{\tiny{2}}$};
	\node[var] (rootnode) at (lr) {\tiny$h_{\tiny{2}}$};
	\node[var] (rootnode) at (lc) {\tiny$r_3$};
	\node[var] (rootnode) at (rl) {\tiny${r}_{\tiny{2}}$};
	\node[var] (rootnode) at (ll) {\tiny$k_1$};
	\node[var] (rootnode) at (rc) {\tiny$r_{\tiny{1}}$};
	\node[var] (rootnode) at (rr) {\tiny$k_{\tiny{1}}$};
\end{tikzpicture},
$$
under the condition 
$ r_1 + r_3 + k_2 - r_2 - k_1 =  r_3 + h_2 - k_1,$
i.e.
$$ h_2 = k_2 + r_1 - r_2. $$
We have that 
\begin{align*}
	\mathfrak{a}(
	\begin{tikzpicture}[scale=0.2,baseline=-5]
	\coordinate (root) at (0,-3);
	\coordinate (c) at (0,-1);
	\coordinate (cl) at (-2,2);
	\coordinate (cc) at (0,2);
	\coordinate (cr) at (2,2);
	\coordinate (cll) at (-2,5);
	\coordinate (clc) at (0,5);
	\coordinate (clr) at (2,5);
	\coordinate (clcl) at (-2,8);
	\coordinate (clcc) at (0,8);
	\coordinate (clcr) at (2,8);
	\draw[symbols] (root) -- (c);
	\draw[kernels2,tinydots] (c) -- (cl);
	\draw[symbols,tinydots] (cc) -- (clc);
	\draw[symbols] (c) -- (cc);
	\draw[kernels2] (c) -- (cr);
	\draw[kernels2] (cc) -- (cll);
	\draw[kernels2] (cc) -- (clr);
	\draw[kernels2] (clc) -- (clcl);
	\draw[kernels2,tinydots] (clc) -- (clcr);
	\draw[kernels2,tinydots] (clc) -- (clcc);
	\node[not] (rootnode) at (c) {};
	\node[not] (rootnode) at (cc) {};
	\node[not] (rootnode) at (clc) {};
	\node[var] (rootnode) at (cl) {\tiny$k_{\tiny{2}}$};
	\node[var] (rootnode) at (cr) {\tiny$k_{\tiny{1}}$};
	\node[var] (rootnode) at (clr) {\tiny$k_{\tiny{2}}$};
	\node[var] (rootnode) at (clcr) {\tiny$k_{\tiny{1}}$};
	\node[var] (rootnode) at (cll) {\tiny$r_{\tiny{3}}$};
	\node[var] (rootnode) at (clcl) {\tiny$r_{\tiny{1}}$};
	\node[var] (rootnode) at (clcc) {\tiny$r_{\tiny{2}}$};
	\end{tikzpicture}
	)
	&= 
i	\nlsconjletter{\ell_{\tiny{1}}}{r_{\tiny{1}}}{r_{\tiny{2}}}{k_{\tiny{1}}} \,\,\,
		\begin{tikzpicture}[scale=0.2,baseline=-5]
		\coordinate (root) at (0,-1);
		\coordinate (rightc) at (1,2);
		\coordinate (right) at (3,2);
		\coordinate (leftc) at (-1,2);
		\coordinate (left) at (-3,2);
		\draw[kernels2] (right) -- (root);
		\draw[kernels2] (rightc) -- (root);
		\draw[kernels2,tinydots] (left) -- (root);
		\draw[kernels2,tinydots] (leftc) -- (root);
		\node[not] (rootnode) at (root) {};
		\node[var] (rootnode) at (rightc) {\tiny{$ {r}_{\tiny{3}} $}};
		\node[var2] (rootnode) at (leftc) {\tiny{$ \ell_{\tiny{2}} $}};
		\node[var2] (rootnode) at (left) {\tiny{$ {\ell}_{\tiny{1}} $}};
		\node[var] (rootnode) at (right) {\tiny{$ k_{\tiny{2}} $}};
	\end{tikzpicture}
	 \,\,\,
		\begin{tikzpicture}[scale=0.2,baseline=-5]
		\coordinate (root) at (0,-1);
		\coordinate (rightc) at (1,2);
		\coordinate (right) at (3,2);
		\coordinate (leftc) at (-1,2);
		\coordinate (left) at (-3,2);
		\draw[kernels2] (right) -- (root);
		\draw[kernels2] (rightc) -- (root);
		\draw[kernels2,tinydots] (left) -- (root);
		\draw[kernels2,tinydots] (leftc) -- (root);
		\node[not] (rootnode) at (root) {};
		\node[var] (rootnode) at (rightc) {\tiny{$ k_{\tiny{1}} $}};
		\node[var] (rootnode) at (leftc) {\tiny{$ k_{\tiny{2}} $}};
		\node[var2] (rootnode) at (left) {\tiny{$ {\ell}_{\tiny{3}} $}};
		\node[var2] (rootnode) at (right) {\tiny{$ \ell_{\tiny{2}} $}};
	\end{tikzpicture} \,\,\, 
	\begin{tikzpicture}[scale=0.2,baseline=-5]
		\coordinate (root) at (0,-1);
		\coordinate (center) at (0,2);
		\draw[kernels2] (root) -- (center);
		\node[var2] (rootnote) at (center) {\tiny{$\ell_{\tiny{3}}$}};
	\end{tikzpicture} &&= i a_1 a_2 a_3 \begin{tikzpicture}[scale=0.2,baseline=-5]
		\coordinate (root) at (0,-1);
		\coordinate (center) at (0,2);
		\draw[kernels2] (root) -- (center);
		\node[var2] (rootnote) at (center) {\tiny{$\ell_{\tiny{3}}$}};
	\end{tikzpicture},
	\\
	\mathfrak{a}(
	\begin{tikzpicture}[scale=0.2,baseline=-5]
			\coordinate (root) at (0,-3);
		\coordinate (rt) at (0,-1);
		\coordinate (l) at (-4,2);
		\coordinate (c) at (0,2);
		\coordinate (r) at (7,2);
		\coordinate (ll) at (-2,5);
		\coordinate (lc) at (0,5);
		\coordinate (lr) at (2,5);
		\coordinate (rl) at (5, 5);
		\coordinate (rc) at (7,5);
		\coordinate (rr) at (9,5);
		\draw[symbols] (root) -- (rt);
		\draw[kernels2,tinydots] (rt) -- (l);
		\draw[symbols] (rt) -- (r); 
		\draw[symbols] (rt) -- (c);
		\draw[kernels2,tinydots] (c) -- (ll);
		\draw[kernels2] (c) -- (lc);
		\draw[kernels2] (c) -- (lr);
		\draw[kernels2,tinydots] (r) -- (rl);
		\draw[kernels2] (r) -- (rc);
		\draw[kernels2] (r) -- (rr);
		\node[not] (rootnode) at (c) {};
		\node[not] (rootnode) at (r) {};
		\node[not] (rootnode) at (rt) {};
		\node[var] (rootnode) at (l) {\tiny$h_{\tiny{2}}$};
		\node[var] (rootnode) at (lr) {\tiny$h_{\tiny{2}}$};
		\node[var] (rootnode) at (lc) {\tiny$r_3$};
		\node[var] (rootnode) at (rl) {\tiny${r}_{\tiny{2}}$};
		\node[var] (rootnode) at (ll) {\tiny$k_1$};
		\node[var] (rootnode) at (rc) {\tiny$r_{\tiny{1}}$};
		\node[var] (rootnode) at (rr) {\tiny$k_{\tiny{1}}$};
	\end{tikzpicture}
	)
	&= -i
	(\nlsletter{\ell_{\tiny{5}}}{k_{\tiny 1}}{r_{\tiny{3}}}{h_{\tiny2}} \shuffle \nlsletter{\ell_{\tiny{6}}}{r_{\tiny{2}}}{r_{\tiny{1}}}{k_{\tiny{1}}}) 
\, \, \, \begin{tikzpicture}[scale=0.2,baseline=-5]
	\coordinate (root) at (0,-1);
	\coordinate (rightc) at (1,2);
	\coordinate (right) at (3,2);
	\coordinate (leftc) at (-1,2);
	\coordinate (left) at (-3,2);
	\draw[kernels2] (right) -- (root);
	\draw[kernels2] (rightc) -- (root);
	\draw[kernels2,tinydots] (left) -- (root);
	\draw[kernels2,tinydots] (leftc) -- (root);
	\node[not] (rootnode) at (root) {};
	\node[var2] (rootnode) at (rightc) {\tiny{$ \ell_{\tiny{5}} $}};
	\node[var] (rootnode) at (leftc) {\tiny{$ h_{\tiny{2}} $}};
	\node[var2] (rootnode) at (left) {\tiny{$ {\ell}_{\tiny{3}} $}};
	\node[var2] (rootnode) at (right) {\tiny{$ \ell_{\tiny{6}} $}};
\end{tikzpicture}  \,\,\, 
	\begin{tikzpicture}[scale=0.2,baseline=-5]
		\coordinate (root) at (0,-1);
		\coordinate (center) at (0,2);
		\draw[kernels2] (root) -- (center);
		\node[var2] (rootnote) at (center) {\tiny{$\ell_{\tiny{3}}$}};
	\end{tikzpicture} &&= -i (b_1 \shuffle b_2) b_3  \begin{tikzpicture}[scale=0.2,baseline=-5]
		\coordinate (root) at (0,-1);
		\coordinate (center) at (0,2);
		\draw[kernels2] (root) -- (center);
		\node[var2] (rootnote) at (center) {\tiny{$\ell_{\tiny{3}}$}};
	\end{tikzpicture}.
\end{align*}
 The condition $h_2 = k_2 + r_2 -r_1$ corresponds to 
$$ \ell_5 = \ell_2, \quad \ell_6 = \ell_1. $$
Therefore, if $\psi_{k_1,\ell_1}$ is the map that swaps the leaf decorations $k_1$ and $\ell_1$ as above, we obtain that 
\begin{equation*}
	\psi_{k_1,\ell_1}(b_1) = a_2, \quad \psi_{k_1,\ell_1}(b_2) = a_1, \quad \psi_{k_1,\ell_1}(b_3) = a_3.
\end{equation*}
In particular, we have that 
$$  \mathfrak{a}(
\begin{tikzpicture}[scale=0.2,baseline=-5]
\coordinate (root) at (0,-3);
\coordinate (c) at (0,-1);
\coordinate (cl) at (-2,2);
\coordinate (cc) at (0,2);
\coordinate (cr) at (2,2);
\coordinate (cll) at (-2,5);
\coordinate (clc) at (0,5);
\coordinate (clr) at (2,5);
\coordinate (clcl) at (-2,8);
\coordinate (clcc) at (0,8);
\coordinate (clcr) at (2,8);
\draw[symbols] (root) -- (c);
\draw[kernels2,tinydots] (c) -- (cl);
\draw[symbols,tinydots] (cc) -- (clc);
\draw[symbols] (c) -- (cc);
\draw[kernels2] (c) -- (cr);
\draw[kernels2] (cc) -- (cll);
\draw[kernels2] (cc) -- (clr);
\draw[kernels2] (clc) -- (clcl);
\draw[kernels2,tinydots] (clc) -- (clcr);
\draw[kernels2,tinydots] (clc) -- (clcc);
\node[not] (rootnode) at (c) {};
\node[not] (rootnode) at (cc) {};
\node[not] (rootnode) at (clc) {};
\node[var] (rootnode) at (cl) {\tiny$k_{\tiny{2}}$};
\node[var] (rootnode) at (cr) {\tiny$k_{\tiny{1}}$};
\node[var] (rootnode) at (clr) {\tiny$k_{\tiny{2}}$};
\node[var] (rootnode) at (clcr) {\tiny$k_{\tiny{1}}$};
\node[var] (rootnode) at (cll) {\tiny$r_{\tiny{3}}$};
\node[var] (rootnode) at (clcl) {\tiny$r_{\tiny{1}}$};
\node[var] (rootnode) at (clcc) {\tiny$r_{\tiny{2}}$};
\end{tikzpicture}
)
+ 
\psi_{k_1,\ell_1}\big( \mathfrak{a}(
\begin{tikzpicture}[scale=0.2,baseline=-5]
		\coordinate (root) at (0,-3);
	\coordinate (rt) at (0,-1);
	\coordinate (l) at (-4,2);
	\coordinate (c) at (0,2);
	\coordinate (r) at (7,2);
	\coordinate (ll) at (-2,5);
	\coordinate (lc) at (0,5);
	\coordinate (lr) at (2,5);
	\coordinate (rl) at (5, 5);
	\coordinate (rc) at (7,5);
	\coordinate (rr) at (9,5);
	\draw[symbols] (root) -- (rt);
	\draw[kernels2,tinydots] (rt) -- (l);
	\draw[symbols] (rt) -- (r); 
	\draw[symbols] (rt) -- (c);
	\draw[kernels2,tinydots] (c) -- (ll);
	\draw[kernels2] (c) -- (lc);
	\draw[kernels2] (c) -- (lr);
	\draw[kernels2,tinydots] (r) -- (rl);
	\draw[kernels2] (r) -- (rc);
	\draw[kernels2] (r) -- (rr);
	\node[not] (rootnode) at (c) {};
	\node[not] (rootnode) at (r) {};
	\node[not] (rootnode) at (rt) {};
	\node[var] (rootnode) at (l) {\tiny$h_{\tiny{2}}$};
	\node[var] (rootnode) at (lr) {\tiny$h_{\tiny{2}}$};
	\node[var] (rootnode) at (lc) {\tiny$r_3$};
	\node[var] (rootnode) at (rl) {\tiny${r}_{\tiny{2}}$};
	\node[var] (rootnode) at (ll) {\tiny$k_1$};
	\node[var] (rootnode) at (rc) {\tiny$r_{\tiny{1}}$};
	\node[var] (rootnode) at (rr) {\tiny$k_{\tiny{1}}$};
\end{tikzpicture}
)
\big) = - i a_2 a_1 a_3 \begin{tikzpicture}[scale=0.2,baseline=-5]
	\coordinate (root) at (0,-1);
	\coordinate (center) at (0,2);
	\draw[kernels2] (root) -- (center);
	\node[var2] (rootnote) at (center) {\tiny{$\ell_{\tiny{3}}$}};
\end{tikzpicture},
$$
which is again a word that can be estimated using \cite[Lemma 7.1 (a)]{DH2301}.

\section{Three-dimensional cubic wave equation}

\label{Sec::3}

In this section, we analyse the cancellations obtained in \cite{BDNY24}. We consider the following equation
\begin{equs}
(\partial_t^2 + 1 - \Delta) u =  - u^3, \quad (u,\langle \nabla  \rangle^{-1}u)|_{t=0}  = (\phi^{\cos},\phi^{\sin}),
\end{equs}
where $(t, x)  \in \mathbb{R} \times \mathbb{T}^3 $.
Let $ \varphi \in \lbrace \cos, \sin \rbrace $, we define $ (W_{n}^{\varphi})_{n \in \mathbb{Z}^3} $ as independent copies of the Gaussian process with mean zero and covariance given by
\begin{equs}
	\mathbb{E}( W _n^{\varphi}(s) W _m^{\varphi}(t)) = \delta_{m+n} s \wedge t.
\end{equs}
Then, $ d W_t^{\varphi} $ is a space-time white noise.
 We define the Gaussian process $ W^{\varphi}_s $ by
 \begin{equs}
 	W^{\varphi}_s(x) := \sum_{n \in \mathbb{Z}^3} W^{\varphi}_n(s) e^{inx} \end{equs}
We consider the stochastic heat equation with space-time white noise given by
\begin{equs} \label{parabolic_solution}
	(\partial_t + 1 - \Delta) u^{\varphi} = \sqrt{2} d W^{\varphi}_t
\end{equs}
Then, one has
\begin{equs}
	u^{\varphi}(t,x) & = \sqrt{2} \sum_{n \in \mathbb{Z}^3} u^{\varphi}_n(t)  e^{in  x},  \quad
u^{\varphi}_n(t)	 = \int^t_{- \infty} e^{-(t-s) \langle n \rangle^2}  d W^{\varphi}_n(s).
\end{equs}
From \cite[Lemma 6.5]{BDNY24} via a simple computation, one gets
\begin{equs}
	\mathbb{E}( u^{\varphi}_n(t) u^{\varphi'}_{n'}(t') ) = \delta_{\varphi = \varphi'} \delta_{n+n'}
	\frac{1}{\langle n \rangle^2 } e^{-|t-t'| \langle n \rangle^2}.
\end{equs}
Now, we consider the solution of the linear wave equation:
\begin{equs}
	(\partial_t^2 + 1 - \Delta) v =  0, \quad (v,\langle \nabla  \rangle^{-1}v)|_{t=0}  = (u^{\cos},u^{\sin}).
\end{equs}
The solution can be described explicitly by
\begin{equs}
	v(t,x)=  \cos(t \langle \nabla \rangle) u^{\cos}(t,x) + \sin(t \langle \nabla \rangle) u^{\sin}(t,x).
\end{equs}
Going into Fourier space, one has
\begin{equs}
	v_n(t)=  \cos(t \langle n \rangle) u^{\cos}_n(t) + \sin(t \langle n \rangle) u^{\sin}_n(t).
	\end{equs}
Then, we can compute the covariance of $ v_n $ and get as in \cite[Lemma 6.10]{BDNY24}
\begin{equs} \label{covariance_v}
	\mathbb{E}( v_n(t) v_{n'}(t') ) = 
	\delta_{n+n'} \frac{\cos((t-t') \langle n \rangle)}{\langle n  \rangle^2}.
\end{equs}
The previous covariance gives a nice identity
 described in the next proposition:
\begin{proposition} \label{covariance} One has
	\begin{equs} \label{identity_key}
		\frac{\sin((t-t') \langle n \rangle)}{\langle n  \rangle} = - \partial_t	\mathbb{E}( v_n(t) v_{-n'}(t') ) = \partial_{t'}	\mathbb{E}( v_n(t) v_{-n'}(t') ).
	\end{equs}
	\end{proposition}
\begin{proof}
	It is an immediate consequence of \eqref{covariance_v}. Indeed, one has
\begin{align*}
		 \partial_t \mathbb{E}( v_n(t) v_{-n'}(t') ) &= 
		 \partial_t \left(  \frac{\cos((t-t') \langle n \rangle)}{\langle n  \rangle^2} \right) \\
		 &= - 	\frac{\sin((t-t') \langle n \rangle)}{\langle n  \rangle} \\
		 &= - \partial_{t'} \left(  \frac{\cos((t-t') \langle n \rangle)}{\langle n  \rangle^2} \right)\\
		 &= - \partial_{t'} \mathbb{E}( v_n(t) v_{-n'}(t') ). 
	\end{align*}
\end{proof}
The identity \eqref{identity_key} is the key to understanding the cancellation happening in dispersive PDEs. It makes appear a derivative in time that could be moved via integration by parts. We need to introduce graphical notation to describe the stochastic iterated integral from whom we want to compute the expectation and check some cancellations. We use the same decorated trees formalism as introduced in the previous section. The main difference is that all the edges will not have conjugate and therefore the second edge decoration is always $0$. 

The analytical interpretation is  however different, due to the presence of the wave propagator as opposed to the Schr\"odinger propagator. We introduce a new map, denoted by $\Pi_N$, defined for a fix $N \in \mathbb{N}$ by
\begin{equs}
		(\Pi_N  \begin{tikzpicture}[scale=0.2,baseline=-5]
		\coordinate (root) at (0,1);
		\coordinate (tri) at (0,-1);
		\draw[kernels2] (tri) -- (root);
		\node[var] (rootnode) at (root) {\tiny{$ k $}};
		\node[not] (trinode) at (tri) {};
	\end{tikzpicture}) (t) & = \one_{\leq N}(k) \,  v_k(t), \quad (\Pi_N T_1 T_2)(t) = (\Pi_N T_1)(t) (\Pi_N  T_2)(t) \\
		(\Pi_N \mathcal{I}_{(\mathfrak{t}_2,0)}(\lambda_k T))(t) & = \one_{\leq N}(k) \int^{t}_0 \left( \frac{\sin((t-s) \langle k \rangle)}{\langle k  \rangle} \right) (\Pi_N T)(s) ds.
\end{equs}
As an example, one gets:
\begin{equs}
	(\Pi_N \begin{tikzpicture}[scale=0.2,baseline=-5]
		\coordinate (root) at (0,-1);
		\coordinate (t1) at (-2,1);
		\coordinate (t2) at (2,1);
		\coordinate (t3) at (0,2);
		\draw[kernels2] (t1) -- (root);
		\draw[kernels2] (t2) -- (root);
		\draw[kernels2] (t3) -- (root);
		\node[not] (rootnode) at (root) {};t
		\node[var] (rootnode) at (t1) {\tiny{$ k_{\tiny{1}} $}};
		\node[var] (rootnode) at (t3) {\tiny{$ k_{\tiny{2}} $}};
		\node[var] (trinode) at (t2) {\tiny{$ k_3 $}};
	\end{tikzpicture}  )(t) & = \prod_{i=1}^3 \one_{\leq N}(k_i) \,  v_{k_i}(t),
	\\ ( \Pi_N \begin{tikzpicture}[scale=0.2,baseline=-5]
		\coordinate (root) at (0,0);
		\coordinate (tri) at (0,-2);
		\coordinate (t1) at (-2,2);
		\coordinate (t2) at (2,2);
		\coordinate (t3) at (0,3);
		\draw[kernels2] (t1) -- (root);
		\draw[kernels2] (t2) -- (root);
		\draw[kernels2] (t3) -- (root);
		\draw[symbols] (root) -- (tri);
		\node[not] (rootnode) at (root) {};t
		\node[not] (trinode) at (tri) {};
		\node[var] (rootnode) at (t1) {\tiny{$ k_{\tiny{1}} $}};
		\node[var] (rootnode) at (t3) {\tiny{$ k_{\tiny{2}} $}};
		\node[var] (trinode) at (t2) {\tiny{$ k_3 $}};
	\end{tikzpicture}) (t) & = \one_{\leq N}(k_{123}) \int^{t}_0 \left( \frac{\sin((t-s) \langle k_{123} \rangle)}{\langle k_{123}  \rangle} \right) \prod_{i=1}^3 \one_{\leq N}(k_i) \,  v_{k_i}(s)   ds,
\end{equs}
where $ k_{123} = k_1 + k_2 + k_3 $. We still use the same pairing structure and consider decorated trees of the form $ T^{\mff}_{\mathfrak{e},\mathfrak{p}} $. We assume that $\mathfrak{p}$ is a partition of the leaves of $T$.
One can extend the definition of $ \Pi_N $ to this case but with a non-recursive definition:
\begin{equs} \label{general_pi_2}
	\begin{aligned}
	(\Pi_N T_{\mfe,\mathfrak{p}}^{\mff})(t)  & =  \prod_{u \in \bar{N}_{T} } \one_{\leq N}(\mff(u)) \int  \prod_{e \in E_T^2} \one_{\lbrace  0<t_{e_-} < t_{e_+} \rbrace} \left( \frac{\sin((t_{e_+}-t_{e_-}) \langle \mff(e_-) \rangle)}{\langle \mff(e_-)  \rangle} \right)  \\ & \prod_{\lbrace x,y \rbrace \in \mathfrak{p}} 	\mathbb{E}( v_{\mff(x_-)}(t_{x_+}) v_{\mff(y_-)}(t_{y_+}) ) \prod_{e \in E_T^2} dt_{e_-}
	\end{aligned}
\end{equs}
where $ \bar{N}_T $ corresponds to the inner nodes of $T$. Below, we illustrate this definition with two examples:
 \begin{equs}
 	(\Pi_N \begin{tikzpicture}[scale=0.2,baseline=-5]
 		\coordinate (root) at (0,-1);
 		\coordinate (right) at (3,2);
 		\coordinate (left) at (-3,2);
 		\coordinate (rightr) at (5,5);
 		\coordinate (rightc) at (3,5);
 		\coordinate (rightl) at (1,5);
 		\coordinate (leftl) at (-5,5);
 		\coordinate (leftc) at (-3,5);
 		\coordinate (leftr) at (-1,5);
 		\draw[symbols] (root) -- (right);
 		\draw[symbols] (root) -- (left);
 		\draw[kernels2] (right) -- (rightc);
 		\draw[kernels2] (right) -- (rightr);
 		\draw[kernels2] (right) -- (rightl);
 		\draw[kernels2] (left) -- (leftc);
 		\draw[kernels2] (left) -- (leftr);
 		\draw[kernels2] (left) -- (leftl);
 		\node[not] (rootnode) at (root) {};
 		\node[not] (rootnode) at (right) {};
 		\node[not] (rootnode) at (left) {};
 		\node[var] (rootnode) at (leftl) {\tiny{$ k_{\tiny{1}} $}};
 		\node[var] (rootnode) at (leftc) {\tiny{$ k_{\tiny{2}} $}};
 		\node[var] (rootnode) at (leftr) {\tiny{$ k_{\tiny{3}} $}};
 		\node[var] (rootnode) at (rightl) {\tiny{$ \bar{k}_{\tiny{1}} $}};
 		\node[var] (rootnode) at (rightc) {\tiny{$ \bar{k}_{\tiny{2}} $}};
 		\node[var] (rootnode) at (rightr) {\tiny{$ \bar{k}_{\tiny{3}} $}};
 	\end{tikzpicture})(t)  & = \one_{\leq N}(k_{123})  \prod_{i =1 }^3 \one_{\leq N}(k_i) \int_0^t \int_0^t  \ \left( \frac{\sin((t-s) \langle k_{123} \rangle)}{\langle k_{123}  \rangle} \right)  \\ & \left( \frac{\sin((t-\bar{s}) \langle \bar{k}_{123} \rangle)}{\langle \bar{k}_{123}  \rangle} \right) \prod_{i=1}^3 	\mathbb{E}( v_{k_i}(s) v_{\bar{k}_i}(\bar{s}) ) ds d\bar{s} 
 \end{equs}
where $ \bar{k}_i = - k_i $ and $ \bar{k}_{123} = - k_{123} $. As a second example, one has
\begin{equs}
\, &	(\Pi_N  \begin{tikzpicture}[scale=0.2,baseline=-5]
		\coordinate (root) at (0,-1);
		\coordinate (right) at (2,1);
		\coordinate (left) at (-2,1);
		\coordinate (leftr) at (0,4);
		\coordinate (leftl) at (-4,4);
		\coordinate (leftc) at (-2,4);
		\coordinate (leftll) at (-6,7);
		\coordinate (leftlr) at (-2,7);
		\coordinate (leftlc) at (-4,7);
		\draw[kernels2] (right) -- (root);
		\draw[symbols] (left) -- (root);
		\draw[kernels2] (left) -- (leftc);
		\draw[kernels2] (left) -- (leftr);
		\draw[symbols] (left) -- (leftl);
		\draw[kernels2] (leftl) -- (leftlc);
		\draw[kernels2] (leftl) -- (leftll);
		\draw[kernels2] (leftl) -- (leftlr);
		\node[not] (rootnode) at (root) {};
		\node[not] (rootnode) at (left) {};
		\node[not] (rootnode) at (leftl) {};
		\node[var] (rootnode) at (right) {\tiny{$ \bar{k}_{\tiny{3}} $}};
		\node[var] (rootnode) at (leftr) {\tiny{$ \bar{k}_{\tiny{2}} $}};
		\node[var] (rootnode) at (leftc) {\tiny{$ \bar{k}_{\tiny{1}} $}};
		\node[var] (rootnode) at (leftlr) {\tiny{$ k_{\tiny{3}} $}};
		\node[var] (trinode) at (leftlc) {\tiny{$ k_2 $}};
		\node[var] (trinode) at (leftll) {\tiny{$ k_1 $}};
	\end{tikzpicture}  )(t)   = \one_{\leq N}(k_{123})  \prod_{i =1 }^3 \one_{\leq N}(k_i) \int_0^t \int_0^s  \ \left( \frac{\sin((t-s) \langle k_{3} \rangle)}{\langle k_{3}  \rangle} \right)  \\ & \left( \frac{\sin((s-r) \langle \bar{k}_{123} \rangle)}{\langle \bar{k}_{123}  \rangle} \right) 	\mathbb{E}( v_{k_3}(r) v_{\bar{k}_3}(t) ) \prod_{i=1}^2	\mathbb{E}( v_{k_i}(r) v_{\bar{k}_i}(\bar{s}) ) dr ds. 
\end{equs}
We consider $A$ the alphabet whose letters are given by decorated trees of the form:
\begin{equs} \label{letter_type}
	\begin{tikzpicture}[scale=0.2,baseline=-5]
		\coordinate (root) at (0,-1);
		\coordinate (rightc) at (1,2);
		\coordinate (right) at (3,2);
		\coordinate (leftc) at (-1,2);
		\coordinate (left) at (-3,2);
		\draw[kernels2] (right) -- (root);
		\draw[kernel,darkgreen] (left) -- (root);
		\draw[kernels2] (rightc) -- (root);
		\draw[kernels2] (leftc) -- (root);
		\node[not] (rootnode) at (root) {};
		\node[var] (rootnode) at (left) {\tiny{$ \ell_{\tiny{1}} $}};
		\node[var] (rootnode) at (leftc) {\tiny{$ \ell_{\tiny{2}} $}};
		\node[var] (trinode) at (rightc) {\tiny{$ \ell_{\tiny{3}} $}};
		\node[var] (trinode) at (right) {\tiny{$ \ell_{\tiny{4}} $}};
	\end{tikzpicture}, \quad \begin{tikzpicture}[scale=0.2,baseline=-5]
\coordinate (root) at (0,-1);
\coordinate (right) at (-1.5,2);
\coordinate (left) at (1.5,2);
\draw[kernels2] (right) -- (root);
\draw[kernels2] (left) -- (root);
\node[not] (rootnode) at (root) {};
\node[var] (rootnode) at (left) {\tiny{$ \ell_{\tiny{2}} $}};
\node[var] (trinode) at (right) {\tiny{$ \ell_{\tiny{1}} $}};
\end{tikzpicture}  
\end{equs}
where for the first letter, we assume that $ \sum_{i=1}^4 \ell_i =0 $. The presence of green edges is due to the necessity of keeping track of the time derivative in \eqref{identity_key}.
 We also consider letters where the brown edges can be replaced by green edges. Below, we give one example of such letter
\begin{equs}
	\begin{tikzpicture}[scale=0.2,baseline=-5]
		\coordinate (root) at (0,-1);
		\coordinate (rightc) at (1,2);
		\coordinate (right) at (3,2);
		\coordinate (leftc) at (-1,2);
		\coordinate (left) at (-3,2);
		\draw[kernels2] (right) -- (root);
		\draw[kernel,darkgreen] (left) -- (root);
		\draw[kernels2] (rightc) -- (root);
		\draw[kernels2] (leftc) -- (root);
		\node[not] (rootnode) at (root) {};
		\node[var] (rootnode) at (left) {\tiny{$ \ell_{\tiny{1}} $}};
		\node[var] (rootnode) at (leftc) {\tiny{$ \ell_{\tiny{2}} $}};
		\node[var] (trinode) at (rightc) {\tiny{$ \ell_{\tiny{3}} $}};
		\node[var] (trinode) at (right) {\tiny{$ \ell_{\tiny{4}} $}};
	\end{tikzpicture}.
\end{equs}
 In symbolic notation, one has
\begin{equs}
 \begin{tikzpicture}[scale=0.2,baseline=-5]
 	\coordinate (root) at (0,-1);
 	\coordinate (rightc) at (1,2);
 	\coordinate (right) at (3,2);
 	\coordinate (leftc) at (-1,2);
 	\coordinate (left) at (-3,2);
 	\draw[kernels2] (right) -- (root);
 	\draw[kernel,darkgreen] (left) -- (root);
 	\draw[kernels2] (rightc) -- (root);
 	\draw[kernels2] (leftc) -- (root);
 	\node[not] (rootnode) at (root) {};
 	\node[var] (rootnode) at (left) {\tiny{$ \ell_{\tiny{1}} $}};
 	\node[var] (rootnode) at (leftc) {\tiny{$ \ell_{\tiny{2}} $}};
 	\node[var] (trinode) at (rightc) {\tiny{$ \ell_{\tiny{3}} $}};
 	\node[var] (trinode) at (right) {\tiny{$ \ell_{\tiny{4}} $}};
 \end{tikzpicture} & = \hat{\CI}_{(\mathfrak{t}_1,0)}(\lambda_{\ell_1}) \CI_{(\mathfrak{t}_1,0)}(\lambda_{\ell_2}) \CI_{(\mathfrak{t}_1,0)}(\lambda_{\ell_3}) \CI_{(\mathfrak{t}_1,0)}(\lambda_{\ell_4}) , \\ \begin{tikzpicture}[scale=0.2,baseline=-5]
 	\coordinate (root) at (0,-1);
 	\coordinate (right) at (-1.5,2);
 	\coordinate (left) at (1.5,2);
 	\draw[kernels2] (right) -- (root);
 	\draw[kernels2] (left) -- (root);
 	\node[not] (rootnode) at (root) {};
 	\node[var] (rootnode) at (left) {\tiny{$ \ell_{\tiny{2}} $}};
 	\node[var] (trinode) at (right) {\tiny{$ \ell_{\tiny{1}} $}};
 \end{tikzpicture} & = \CI_{(\mathfrak{t}_1,0)}(\lambda_{\ell_1}) \CI_{(\mathfrak{t}_1,0)}(\lambda_{\ell_2}).
\end{equs} 
where we assume that $ \ell_1= - (\ell_2 + \ell_3 + \ell_4) $. Observe that now the notation $\hat{\mathcal{I}}_{(\mathfrak{t}_1,0)}$ is used for encoding green edges.
Let denote by $ W_A $ words on the alphabet $A$ such that the leaves corresponds to a pairing. The first letter of these words is of the second type from \eqref{letter_type}. The other letters are of the first type. Below, we present an example of such a word:
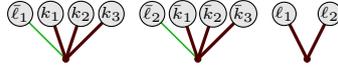
\begin{equs}
		\begin{tikzpicture}[scale=0.2,baseline=-5]
			\coordinate (root) at (0,-1);
			\coordinate (rightc) at (1,2);
			\coordinate (right) at (3,2);
			\coordinate (leftc) at (-1,2);
			\coordinate (left) at (-3,2);
			\draw[kernels2] (right) -- (root);
			\draw[kernel,darkgreen] (left) -- (root);
			\draw[kernels2] (rightc) -- (root);
			\draw[kernels2] (leftc) -- (root);
			\node[not] (rootnode) at (root) {};
			\node[var] (rootnode) at (left) {\tiny{$ \bar{\ell}_{\tiny{1}} $}};
			\node[var] (rootnode) at (leftc) {\tiny{$ k_{\tiny{1}} $}};
			\node[var] (trinode) at (rightc) {\tiny{$ k_{\tiny{2}} $}};
			\node[var] (trinode) at (right) {\tiny{$ k_{\tiny{3}} $}};
		\end{tikzpicture} 
	\, \, \,
	\begin{tikzpicture}[scale=0.2,baseline=-5]
		\coordinate (root) at (0,-1);
		\coordinate (rightc) at (1,2);
		\coordinate (right) at (3,2);
		\coordinate (leftc) at (-1,2);
		\coordinate (left) at (-3,2);
		\draw[kernels2] (right) -- (root);
		\draw[kernel,darkgreen] (left) -- (root);
		\draw[kernels2] (rightc) -- (root);
		\draw[kernels2] (leftc) -- (root);
		\node[not] (rootnode) at (root) {};
		\node[var] (rootnode) at (left) {\tiny{$ \bar{\ell}_{\tiny{2}} $}};
		\node[var] (rootnode) at (leftc) {\tiny{$ \bar{k}_{\tiny{1}} $}};
		\node[var] (trinode) at (rightc) {\tiny{$ \bar{k}_{\tiny{2}} $}};
		\node[var] (trinode) at (right) {\tiny{$ \bar{k}_{\tiny{3}} $}};
	\end{tikzpicture}  \, \, \, 	\begin{tikzpicture}[scale=0.2,baseline=-5]
	\coordinate (root) at (0,-1);
	\coordinate (right) at (-1.5,2);
	\coordinate (left) at (1.5,2);
	\draw[kernels2] (right) -- (root);
	\draw[kernels2] (left) -- (root);
	\node[not] (rootnode) at (root) {};
	\node[var] (rootnode) at (left) {\tiny{$ \ell_{\tiny{2}} $}};
	\node[var] (trinode) at (right) {\tiny{$ \ell_{\tiny{1}} $}};
\end{tikzpicture}
\end{equs}
where  $ \ell_1 = k_1 + k_2 + k_3 $. We equip these words with the shuffle product $\shuffle$ defined in the previous section in \eqref{shuffle_def}.

Now, we define a natural morphism between $ (\CH,\cdot_F) $ and $ (T(A), \shuffle) $ that sends the forest product $ \cdot_F $ to the shuffle product. 
It is given inductively on the symbolic notation by
\begin{equs} \label{arborification_Wave}
	\begin{aligned}
		\ &	\mathfrak{a}( \prod_{j=1}^n \CI_{(\mathfrak{t}_1,0)}(\lambda_{\ell_j})  \prod_{r=1}^m \CI_{(\mathfrak{t}_2,0)}(\lambda_{k_r} \tau_r))  = 
		\\ &   \left( \mathfrak{a}(\hat{\CI}_{(\mathfrak{t}_1,0)}(\lambda_{\bar{k}_1})\tau_1) \shuffle \cdots \shuffle \mathfrak{a}(\hat{\CI}_{(\mathfrak{t}_1,0)}(\lambda_{\bar{k}_m})\tau_m) \right)
		\\& \prod_{j=1}^n \CI_{(\mathfrak{t}_1,0)}(\lambda_{\ell_j})  \prod_{r=1}^m \CI_{(\mathfrak{t}_1,0)}(\lambda_{k_r} \tau_r).
	\end{aligned}
\end{equs}
This formula is slightly different from \eqref{arborification_NLS}. Indeed, the rightmost letter does not contain edges of type $ \hat{\mathcal{I}}$.
  In the end, we have transformed edges of type $ \mathfrak{t}_2 $ into two edges of type $ \mathfrak{t}_1 $.
  We also give an alternative definition via 
  a Butcher-Connes-Kreimer type coproduct $ \Delta_{\text{\tiny{Wave}}} : \CF \rightarrow \CF \otimes \CF $, a different version of the one introduced in \cite{BS},  defined recursively by
  \begin{equs} \label{BCK_new}
  	\begin{aligned}
  		\Delta_{\text{\tiny{Wave}}} \CI_{(\mathfrak{t}_1,p)}( \lambda_{k}  ) & = \left( \one \, \otimes \CI_{(\mathfrak{t}_1,p)}( \lambda_{k}  )   \right), \\ 
  		\Delta_{\text{\tiny{Wave}}} \CI_{(\mathfrak{t}_2,p)}( \lambda_{k}  T ) & = \left( \id \, \otimes \CI_{(\mathfrak{t}_2,p)}( \lambda_{k}  \cdot ) \right) \Delta_{\text{\tiny{Wave}}} T \\ &  \   + \hat{\CI}_{(\mathfrak{t}_1,\bar{p})}(\lambda_{\bar{k}} )T   \otimes \CI_{(\mathfrak{t}_1,p)}(\lambda_k)
  	\end{aligned}
  \end{equs}
  and it is extended multiplicatively as the $ \Delta_{\text{\tiny{Wave}}}$. Then, the arborification map $ \mathfrak{a} $ on a decorated tree $T$ is given by
  \begin{equs} \label{arbo_new}		\mathfrak{a}(  T )
  	= \mathcal{M}_{\tiny{\text{c}}}\left( \mathfrak{a}  \otimes P_{A}   \right) \Delta_{\text{\tiny{Wave}}}  T.
  \end{equs}
As an example, one has
\begin{equs}
\mathfrak{a} (T_1) = 	\mathfrak{a} ( \begin{tikzpicture}[scale=0.2,baseline=-5]
		\coordinate (root) at (0,-1);
		\coordinate (right) at (2,1);
		\coordinate (left) at (-2,1);
		\coordinate (leftr) at (0,4);
		\coordinate (leftl) at (-4,4);
		\coordinate (leftc) at (-2,4);
		\coordinate (leftll) at (-6,7);
		\coordinate (leftlr) at (-2,7);
		\coordinate (leftlc) at (-4,7);
		\draw[kernels2] (right) -- (root);
		\draw[symbols] (left) -- (root);
		\draw[kernels2] (left) -- (leftc);
		\draw[kernels2] (left) -- (leftr);
		\draw[symbols] (left) -- (leftl);
		\draw[kernels2] (leftl) -- (leftlc);
		\draw[kernels2] (leftl) -- (leftll);
		\draw[kernels2] (leftl) -- (leftlr);
		\node[not] (rootnode) at (root) {};
		\node[not] (rootnode) at (left) {};
		\node[not] (rootnode) at (leftl) {};
		\node[var] (rootnode) at (right) {\tiny{$ \bar{k}_{\tiny{3}} $}};
		\node[var] (rootnode) at (leftr) {\tiny{$ \bar{k}_{\tiny{2}} $}};
		\node[var] (rootnode) at (leftc) {\tiny{$ \bar{k}_{\tiny{1}} $}};
		\node[var] (rootnode) at (leftlr) {\tiny{$ k_{\tiny{3}} $}};
		\node[var] (trinode) at (leftlc) {\tiny{$ k_2 $}};
		\node[var] (trinode) at (leftll) {\tiny{$ k_1 $}};
	\end{tikzpicture}) & = \mathfrak{a}( \begin{tikzpicture}[scale=0.2,baseline=-5]
	\coordinate (root) at (0,-1);
	\coordinate (right) at (3,2);
		\coordinate (rightc) at (1,2);
	\coordinate (left) at (-3,2);
		\coordinate (leftc) at (-1,2);
	\coordinate (leftc1) at (-3,5);
	\coordinate (leftc2) at (-1,5);
	\coordinate (leftc3) at (1,5);
	\draw[kernels2] (right) -- (root);
	\draw[kernels2] (rightc) -- (root);
	\draw[symbols] (leftc) -- (root);
	\draw[kernel,darkgreen] (left) -- (root);
	\draw[kernels2] (leftc3) -- (leftc);
	\draw[kernels2] (leftc2) -- (leftc);
	\draw[kernels2] (leftc1) -- (leftc);
	\node[not] (rootnode) at (root) {};
	\node[not] (rootnode) at (leftc) {};
	\node[var] (rootnode) at (right) {\tiny{$ \bar{k}_{\tiny{2}} $}};
	\node[var] (rootnode) at (rightc) {\tiny{$ \bar{k}_{\tiny{1}} $}};
	\node[var] (rootnode) at (left) {\tiny{$ \bar{\ell}_{\tiny{1}} $}};
	\node[var] (trinode) at (leftc1) {\tiny{$ k_1 $}};
	\node[var] (trinode) at (leftc2) {\tiny{$ k_2 $}};
	\node[var] (trinode) at (leftc3) {\tiny{$ k_3 $}};
\end{tikzpicture} ) \begin{tikzpicture}[scale=0.2,baseline=-5]
	\coordinate (root) at (0,-1);
	\coordinate (right) at (-1.5,2);
	\coordinate (left) at (1.5,2);
	\draw[kernels2] (right) -- (root);
	\draw[kernels2] (left) -- (root);
	\node[not] (rootnode) at (root) {};
	\node[var] (rootnode) at (left) {\tiny{$ \bar{k}_{\tiny{3}} $}};
	\node[var] (trinode) at (right) {\tiny{$ \ell_{\tiny{1}} $}};
\end{tikzpicture} 
\\ & = 	\begin{tikzpicture}[scale=0.2,baseline=-5]
	\coordinate (root) at (0,-1);
	\coordinate (rightc) at (1,2);
	\coordinate (right) at (3,2);
	\coordinate (leftc) at (-1,2);
	\coordinate (left) at (-3,2);
	\draw[kernels2] (right) -- (root);
	\draw[kernel,darkgreen] (left) -- (root);
	\draw[kernels2] (rightc) -- (root);
	\draw[kernels2] (leftc) -- (root);
	\node[not] (rootnode) at (root) {};
	\node[var] (rootnode) at (left) {\tiny{$ \bar{\ell}_{\tiny{2}} $}};
	\node[var] (rootnode) at (leftc) {\tiny{$ k_{\tiny{1}} $}};
	\node[var] (trinode) at (rightc) {\tiny{$ k_{\tiny{2}} $}};
	\node[var] (trinode) at (right) {\tiny{$ k_{\tiny{3}} $}};
\end{tikzpicture} 
\, \, \,
\begin{tikzpicture}[scale=0.2,baseline=-5]
	\coordinate (root) at (0,-1);
	\coordinate (rightc) at (1,2);
	\coordinate (right) at (3,2);
	\coordinate (leftc) at (-1,2);
	\coordinate (left) at (-3,2);
	\draw[kernels2] (right) -- (root);
	\draw[kernel,darkgreen] (left) -- (root);
	\draw[kernels2] (rightc) -- (root);
	\draw[kernels2] (leftc) -- (root);
	\node[not] (rootnode) at (root) {};
	\node[var] (rootnode) at (left) {\tiny{$ \bar{\ell}_{\tiny{1}} $}};
	\node[var] (rootnode) at (leftc) {\tiny{$ \ell_{\tiny{2}} $}};
	\node[var] (trinode) at (rightc) {\tiny{$ \bar{k}_{\tiny{1}} $}};
	\node[var] (trinode) at (right) {\tiny{$ \bar{k}_{\tiny{2}} $}};
\end{tikzpicture}  \, \, \, 	\begin{tikzpicture}[scale=0.2,baseline=-5]
	\coordinate (root) at (0,-1);
	\coordinate (right) at (-1.5,2);
	\coordinate (left) at (1.5,2);
	\draw[kernels2] (right) -- (root);
	\draw[kernels2] (left) -- (root);
	\node[not] (rootnode) at (root) {};
	\node[var] (rootnode) at (left) {\tiny{$ \bar{k}_{\tiny{3}} $}};
	\node[var] (trinode) at (right) {\tiny{$ \ell_{\tiny{1}} $}};
\end{tikzpicture}.
\end{equs}
We define an evaluation map on words of this alphabet given by
\begin{equs} \label{words_Wave}
\ &	(\Pi_N^A w_{\mathfrak{p}})(t)   =  \prod_{i=1 }^{n} \one_{\leq N}(w_i) \int_{0 < t_1 <.. <t_n = t}  \\ & \prod_{\lbrace x,y \rbrace \in \mathfrak{p}_1} \partial_{t_{x_+} \wedge t_{y_{-}}}	\mathbb{E}( v_{\mff(x_-)}(t_{x_+}) v_{\mff(y_{-})}(t_{y_{-}}) ) \prod_{\lbrace x,y \rbrace \in \mathfrak{p}_2} 	\mathbb{E}( v_{\mff(x_-)}(t_{x_+}) v_{\mff(y_-)}(t_{y_+}) ) \prod_{i=1}^{n-1} dt_i.
\end{equs}
 with
\begin{equs}
	\one_{\leq N}(\begin{tikzpicture}[scale=0.2,baseline=-5]
		\coordinate (root) at (0,-1);
		\coordinate (rightc) at (1,2);
		\coordinate (right) at (3,2);
		\coordinate (leftc) at (-1,2);
		\coordinate (left) at (-3,2);
		\draw[kernels2] (right) -- (root);
		\draw[kernels2] (left) -- (root);
		\draw[kernels2] (rightc) -- (root);
		\draw[kernels2] (leftc) -- (root);
		\node[not] (rootnode) at (root) {};
		\node[var] (rootnode) at (left) {\tiny{$ \ell_{\tiny{1}} $}};
		\node[var] (rootnode) at (leftc) {\tiny{$ \ell_{\tiny{2}} $}};
		\node[var] (trinode) at (rightc) {\tiny{$ \ell_{\tiny{3}} $}};
		\node[var] (trinode) at (right) {\tiny{$ \ell_{\tiny{4}} $}};
	\end{tikzpicture}) = \prod_{i=1}^4 \one_{\leq N}(\ell_i).
\end{equs}
and we get the same value if some of the edges are green.
One has
\begin{equs} \label{word_T_1}
	\begin{aligned}
		(\Pi_N^A \mathfrak{a}(T_1))(t)  & =  \one_{\leq N}(k_{123})\prod_{i=1 }^{3} \one_{\leq N}(k_i) \int_{0 < r < s<t} \partial_s	\mathbb{E}( v_{\bar{\ell}_1}(s) v_{\ell_1}(t) )  \\ &  \partial_r	\mathbb{E}( v_{\bar{\ell}_2}(r) v_{\ell_2}(s) ) 	\mathbb{E}( v_{k_3}(r) v_{\bar{k}_3}(t) )\prod_{i=1}^2 \mathbb{E}( v_{k_i}(r) v_{\bar{k}_i}(s) ) dr ds.
		\end{aligned}
\end{equs}

\begin{theorem} \label{main_theorem_Wave} For any decorated tree $T_{\mfe,\mathfrak{p}}^{\mff}$, one has
	\begin{equs}
		(\Pi_N T_{\mfe,\mathfrak{p}}^{\mff})(t) = \left( \Pi_N^A \mathfrak{a}(T_{\mfe,\mathfrak{p}}^{\mff}) \right) (t).
	\end{equs}
\end{theorem}
\begin{proof}
	One has from \eqref{general_pi_2}
	\begin{equs} 
		\begin{aligned}
			(\Pi_N T_{\mfe,\mathfrak{p}}^{\mff})(t)  & =  \prod_{u \in \bar{N}_{T} } \one_{\leq N}(\mff(u)) \int  \prod_{e \in E_T^2} \one_{\lbrace  0<t_{e_-} < t_{e_+} \rbrace} \left( \frac{\sin((t_{e_+}-t_{e_-}) \langle \mff(e_-) \rangle)}{\langle \mff(e_-)  \rangle} \right)  \\ & \prod_{\lbrace x,y \rbrace \in \mathfrak{p}} 	\mathbb{E}( v_{\mff(x_-)}(t_{x_+}) v_{\mff(y_-)}(t_{y_+}) ) \prod_{e \in E_T^2} dt_{e_-}.
		\end{aligned}
	\end{equs}
	Then, from Proposition  \ref{covariance}, one has
	\begin{equs} 
		\frac{\sin((t_{e_+}-t_{e_-}) \langle \mff(e_-) \rangle)}{\langle \mff(e_-)  \rangle} = \partial_{t_{e_-}}	\mathbb{E}( v_{\mff(e_-)}(t_{e_+}) v_{-\mff(e_-)}(t_{e_-}) ).
	\end{equs}
	By substituting this relation into the identity for $ (\Pi_N T_{\mfe,\mathfrak{p}}^{\mff})(t)  $, one gets
	\begin{equs}
		(\Pi_N T_{\mfe,\mathfrak{p}}^{\mff})(t)  & =  \prod_{u \in \bar{N}_{T} } \one_{\leq N}(\mff(u)) \int  \prod_{e \in E_T^2} \one_{\lbrace  0<t_{e_-} < t_{e_+} \rbrace}  (\partial_{t_{e_-}}	\mathbb{E}( v_{\mff(e_-)}(t_{e_+}) v_{-\mff(e_-)}(t_{e_-}) ))  \\ & \prod_{\lbrace x,y \rbrace \in \mathfrak{p}} 	\mathbb{E}( v_{\mff(x_-)}(t_{x_+}) v_{\mff(y_-)}(t_{y_+}) ) \prod_{e \in E_T^2} dt_{e_-}.
	\end{equs}
	We fix an order on the nodes $ \bar{N}_T $. Let $ n = \mathrm{Card}(\bar{N}_T ) $, then we define $  \mathfrak{S}(\bar{N}_T ) $ as the permutation on these nodes that respects the partial order given by the tree structure  of $T$.
	Then, one has 
	\begin{equs}
	\ &	(\Pi_N T_{\mfe,\mathfrak{p}}^{\mff})(t)   = \sum_{\sigma \in \mathfrak{S}(\bar{N}_{T})} \prod_{u \in \bar{N}_{T} } \one_{\leq N}(\mff(u)) \int_{t_{\sigma(u_1)} < ... < t_{\sigma(u_n)}} \\ & \prod_{\lbrace x,y \rbrace \in \mathfrak{p}_1} \partial_{t_{x_+}\wedge t_{y_{-}}}	\mathbb{E}( v_{\mff(x_-)}(t_{x_+}) v_{\mff(y_{-})}(t_{y_{-}}) )  \prod_{\lbrace x,y \rbrace \in \mathfrak{p}_2} 	\mathbb{E}( v_{\mff(x_-)}(t_{x_+}) v_{\mff(y_-)}(t_{y_+}) ) \prod_{i=1}^n dt_{\sigma(u_i)}.
	\end{equs}
	where $ \mathfrak{p}_1 $ corresponds to the new pairing coming from the edges in $E_T^2$ and $ \mathfrak{p}_2 = \mathfrak{p} $. We conclude as the same as in the proof of Theorem~\ref{main_theorem_NLS}.
\end{proof}

In the next proposition, we compute the main cancellation obtained in \cite{BDNY24}. We use strongly the words formalism for performing the computations.
\begin{proposition}\label{example_computation_2} One has
	\begin{align*}
		\mathfrak{C}_N(t) & : =\sum_{k_1,k_2,k_3 \in \mathbb{Z}^3} 6 (\Pi_N^A \mathfrak{a} (T_1))(t) - \sum_{k_3 \in \mathbb{Z}^3}
		2\Gamma_N(k_3) 
		\Pi_N^A(
		\begin{tikzpicture}[scale=0.2,baseline=-5]
			\coordinate (root) at (0,-1);
			\coordinate (right) at (-1.5,2);
			\coordinate (left) at (1.5,2);
			\draw[kernel, darkgreen] (right) -- (root);
			\draw[kernels2] (left) -- (root);
			\node[not] (rootnode) at (root) {};
			\node[var] (rootnode) at (left) {\tiny{$ k_{\tiny{3}} $}};
			\node[var] (trinode) at (right) {\tiny{$ \bar{\ell}_{\tiny{1}} $}};
		\end{tikzpicture}
		\, \, \,
		\begin{tikzpicture}[scale=0.2,baseline=-5]
			\coordinate (root) at (0,-1);
			\coordinate (right) at (-1.5,2);
			\coordinate (left) at (1.5,2);
			\draw[kernels2] (right) -- (root);
			\draw[kernels2] (left) -- (root);
			\node[not] (rootnode) at (root) {};
			\node[var] (rootnode) at (left) {\tiny{$ \bar{k}_{\tiny{3}} $}};
			\node[var] (trinode) at (right) {\tiny{$ \ell_{\tiny{1}} $}};
		\end{tikzpicture}
		)(t)
	\\ & 	+ \sum_{k_1,k_2,k_3 \in \mathbb{Z}^3} (\Pi_N^A \mathfrak{a}(T_2) )(t) \\
		&= - 2 \sum_{  k_1+k_2+k_3 = \ell_2} \one_{\leq N}(\ell_2)\prod_{i=1}^3 \one_{\leq N}(k_i)\\
		&\phantom{= -}\times \int_{0< s<t} \frac{\cos(s \langle \ell_2 \rangle)}{\langle \ell_2 \rangle^2} \frac{\cos(s \langle k_1 \rangle)}{\langle k_1 \rangle^2} \frac{\cos(s \langle k_2 \rangle)}{\langle k_2 \rangle^2} \frac{\sin((t-s) \langle k_3 \rangle)}{\langle k_3 \rangle} \frac{\cos(t \langle k_3 \rangle)}{\langle k_3 \rangle^2} ds
	\end{align*}
	where
	 \begin{equs}
	 	\Gamma_N(k_3) = \sum_{k_1, k_2 \in \mathbb{Z}^3}   \frac{\one_{\leq N}(k_1)}{\langle k_{1} \rangle^2} \frac{\one_{\leq N}(k_2)}{\langle k_{2} \rangle^2}
	 	\frac{\one_{\leq N}(k_{123})}{\langle k_{123} \rangle^2}.
	 \end{equs}
	\end{proposition}

\begin{proof}
We continue with the decorated tree $T_1$ using \eqref{word_T_1}. Via an integration by parts in the variable $r$, one gets:
\begin{equs}
	\sum_{k_1,k_2,k_3 \in \mathbb{Z}^3}(\Pi_N^A \mathfrak{a} (T_1)) = 
	-\sum_{k_1,k_2,k_3 \in \mathbb{Z}^3} \Pi^A_N(
)
\end{equs}
where the green node corresponds to setting the time of the terms in the letter to $r=0$. We have also enlarged the alphabet $A$ by considering letters with more than four edges. The last two terms of the computation 
correspond exactly to the boundary terms in the integration by parts.
Due to the symmetry between  $k_1,k_2,\bar{\ell}_2$, we obtain 
\begin{equs}
	\sum_{k_1,k_2,k_3 \in \mathbb{Z}^3}(\Pi_N^A \mathfrak{a} (T_1)) = 
	&-\frac13\sum_{k_1,k_2,k_3 \in \mathbb{Z}^3} \Pi^A_N(
	%
)\\
	&+\frac16 \mathfrak{C}_N,
\end{equs}
For the decorated tree $T_2$, one gets
\begin{equs}
	\mathfrak{a}(T_2) & = 	\mathfrak{a} (%
.
\end{equs}
One can observe by symmetry on the decorations $k_i$ and $\bar{k}_i$ that one has
\begin{equs}
	\sum_{k_1,k_2,k_3 \in \mathbb{Z}^3} (\Pi_N^A \mathfrak{a}(T_2) ) = 2 (\Pi_N^A %
).
\end{equs}
Therefore, up to swapping $k_3$ with $\bar{\ell}_2$, we see that 
\begin{align*}
	\mathfrak{C}_N(t) &=\sum_{k_1,k_2,k_3 \in \mathbb{Z}^3} 6 (\Pi_N^A \mathfrak{a} (T_1))(t) - 
\sum_{k_3 \in \mathbb{Z}^3}	2\Gamma_N(k_3) 
	\Pi_N^A(
	%
)(t)\\
	&= - 2 \sum_{ k_1+k_2+k_3 = \ell_2} \one_{\leq N}(\ell_2)\prod_{i=1}^3 \one_{\leq N}(k_i)\\
	&\phantom{= -}\times \int_{0< s<t} \frac{\cos(s \langle \ell_2 \rangle)}{\langle \ell_2 \rangle^2} \frac{\cos(s \langle k_1 \rangle)}{\langle k_1 \rangle^2} \frac{\cos(s \langle k_2 \rangle)}{\langle k_2 \rangle^2} \frac{\sin((t-s) \langle k_3 \rangle)}{\langle k_3 \rangle} \frac{\cos(t \langle k_3 \rangle)}{\langle k_3 \rangle^2} ds.
\end{align*}
Here, removing the term 
$2\Gamma_N(k_3) 
	\Pi^A_N(
	%
)$ 
corresponds to the renormalisation of the equation, 
\end{proof}

 The notations of the previous proposition, $ \Gamma_N, \mathfrak{C}_N $  are chosen to be consistent with \cite[Section 6]{BDNY24}. Moreover,  the term $\mathfrak{C}_N$ can be shown to be uniformly bounded in $N$, see \cite[Lemma 6.23]{BDNY24}.


\begin{thebibliography}{Cha10}
\expandafter\ifx\csname url\endcsname\relax
  \def\url#1{\texttt{#1}}\fi
\expandafter\ifx\csname urlprefix\endcsname\relax\def\urlprefix{URL }\fi
\expandafter\ifx\csname href\endcsname\relax
  \def\href#1#2{#2}\fi
\expandafter\ifx\csname burlalt\endcsname\relax
  \def\burlalt#1#2{\href{#2}{\texttt{#1}}}\fi

\bibitem{ABBS24}
Y.~Alama Bronsard, Y.~Bruned, K.~Schratz.
\newblock { \em Approximations of Dispersive PDEs in the Presence of Low-Regularity Randomness.}
\newblock Found. Comput. Math. \textbf{24}, (2024), 1819--1869.
\newblock
\burlalt{doi:10.1007/s10208-023-09625-8}{http://dx.doi.org/10.1007/s10208-023-09625-8}.

\bibitem{BB24}
C.~Bellingeri, Y.~Bruned.
\newblock {\em Symmetries for the gKPZ equation via multi-indices.
} To appear in Proc. Lond. Math. Soc.
\newblock \burlalt{arXiv:2410.00834 }{http://arxiv.org/abs/2410.00834}.

\bibitem{BCCH}
{\rm Y. Bruned, A. Chandra, I. Chevyrev,
	M. Hairer}.
\newblock {\em Renormalising SPDEs in regularity structures}.
\newblock J. Eur. Math. Soc. (JEMS), \textbf{23}, no.~3, (2021), 869-947. 
\burlalt{doi:10.4171/JEMS/1025}{http://dx.doi.org/10.4171/JEMS/1025}.

 \bibitem{BCEF18}
{\rm Y.~Bruned, C.~Curry,  K.~Ebrahimi-Fard}.
\newblock \emph{Quasi-shuffle algebras and renormalisation of rough differential
	equations}.
\newblock B. Lond. Math. Soc. \textbf{52}, no.~1, (2020), 43--63.
\burlalt{doi:10.1112/blms.12305}{https://dx.doi.org/10.1112/blms.12305}.


\bibitem{BDNY24}
{\rm B.~Bringmann Y.~Deng, A.~Nahmod, H.~Yue }.
\newblock \emph{Invariant Gibbs measures for the three dimensional cubic nonlinear wave equation}.
\newblock  Invent. Math. \textbf{236}, (2024),  1133–1411.
\newblock
\burlalt{doi:10.1007/s00222-024-01254-4}{https://dx.doi.org/10.1007/s00222-024-01254-4}.

\bibitem{BD24}
Y.~Bruned, V.~Dotsenko.
\newblock {\em Chain rule symmetry for singular SPDEs
}
\newblock \burlalt{arXiv:2403.17066 }{http://arxiv.org/abs/2403.17066}.





\bibitem{BGHZ}
{ \rm Y.~Bruned, F.~Gabriel, M.~Hairer, 
	L.~Zambotti},
\newblock   {\em Geometric stochastic heat equations}.
\\
\newblock J. Amer. Math. Soc. (JAMS) \textbf{35}, no.~1, (2022), 1-80. 
\burlalt{doi:10.1090/jams/977}{http://dx.doi.org/10.1090/jams/977}.

\bibitem{BGN24}
Y.~Bruned, M.~Gerencsér, U.~Nadeem
\newblock {\em Quasi-generalised KPZ equation
}
\newblock \burlalt{arXiv:2401.13620 }{http://arxiv.org/abs/2401.13620}.

\bibitem{BHZ}
{\rm Y. Bruned, M. Hairer, L. Zambotti}.
\newblock {\em Algebraic renormalisation of regularity structures.}
\newblock Invent. Math. \textbf{215}, no.~3, (2019), 1039--1156.
\newblock
\burlalt{doi:10.1007/s00222-018-0841-x}{https://dx.doi.org/10.1007/s00222-018-0841-x}.







	
\bibitem{Bru}
Y.~Bruned.
\newblock \emph{{Singular KPZ Type Equations}}.
\newblock 205 pages, {PhD thesis}, {Universit{\'e} Pierre et Marie Curie - Paris VI},
2015.
\newblock \url{https://tel.archives-ouvertes.fr/tel-01306427}.


\bibitem{Br24}
Y.~Bruned. \newblock { \em Derivation of normal forms for dispersive PDEs via arborification}. \burlalt{arXiv:2409.03642} {https://arxiv.org/abs/2409.03642}.

\bibitem{BS}
Y.~{Bruned}, K.~{Schratz}. \newblock { \em Resonance based schemes for dispersive equations via decorated 
	trees}. Forum of Mathematics, Pi, 10, (2022), E2. 
\newblock \burlalt{doi:10.1017/fmp.2021.13}{https://doi.org/10.1017/fmp.2021.13}.

\bibitem{Butcher}
{ \rm J. C. Butcher}.
\newblock {\em An algebraic theory of integration methods.}
\newblock Math. Comp. \textbf{26}, (1972), 79--106.
\newblock \burlalt{doi:10.2307/2004720}{http://dx.doi.org/10.2307/2004720}.



\bibitem{CH16}
A.~Chandra, M.~Hairer.
\newblock {\textsl{An analytic {BPHZ} theorem for regularity structures.}}
\newblock \burlalt{arXiv:1612.08138}{http://arxiv.org/abs/1612.08138}.























\bibitem{CK1}
A.~Connes, D.~Kreimer.
\newblock { \em Hopf algebras, renormalization and noncommutative geometry.}
\newblock Comm. Math. Phys. \textbf{199}, no.~1, (1998), 203--242.
\newblock
\burlalt{doi:10.1007/s002200050499}{http://dx.doi.org/10.1007/s002200050499}.

\bibitem{CK2}
A.~Connes, D.~Kreimer.
\newblock { \em Renormalization in quantum field theory and the {R}iemann-{H}ilbert
	problem {I}: the {H}opf algebra structure of graphs and the main theorem.}
\newblock  Comm. Math. Phys. \textbf{210}, (2000), 249--73.
\newblock
\burlalt{doi:10.1007/s002200050779}{http://dx.doi.org/10.1007/s002200050779}.











\bibitem{DH23}
{\rm Y.~Deng,  Z.~Hani}.
\newblock \emph{Full derivation of the wave kinetic equation}.
\newblock  Invent. math. \textbf{233}, (2023),  543-724.
\newblock
\burlalt{doi:10.1007/s00222-023-01189-2}{https://dx.doi.org/10.1007/s00222-023-01189-2}.

\bibitem{DH2301}
Y.~Deng, Z.~Hani.
\newblock {\textsl{Derivation of the wave kinetic equation: Full range of scaling laws.}} To appear in Mem. Am. Math. Soc.
\newblock \burlalt{arXiv:2301.07063}{http://arxiv.org/abs/2301.07063}.

\bibitem{DH2311}
Y.~Deng, Z.~Hani.
\newblock {\textsl{Long time justification of wave turbulence theory.}}  
\newblock \burlalt{arXiv:2408.07818}{http://arxiv.org/abs/2408.07818}.

\bibitem{DHM24}
Y.~Deng, Z.~Hani, X.~Ma.
\newblock {\textsl{Long time derivation of Boltzmann equation from hard sphere dynamics.}} To appear in Ann. Math.
\newblock \burlalt{arXiv:2408.07818}{http://arxiv.org/abs/2408.07818}.

 \bibitem{DNY22}
{\rm Y.~Deng,  A.R.~Nahmod, H.~Yue}.
\newblock \emph{Random tensors, propagation of randomness, and nonlinear dispersive equations}.
\newblock  Invent. math. \textbf{228}, (2022),  539–686.
\newblock
\burlalt{doi:10.1007/s00222-021-01084-8}{https://dx.doi.org/10.1007/s00222-021-01084-8}.

\bibitem{EV04}
{\rm J.~Ecalle, B.~Vallet}.
\newblock \emph{The arborification-coarborification transform: analytic, combinatorial, and
	algebraic aspects}.
\newblock   Ann. Fac. Sci. Toulouse (6) \textbf{13}, no. 3, (2004),  575–657.

  \bibitem{FM} {\rm F. Fauvet and F. Menous}, {\em Ecalle’s arborification coarborification transforms and Connes Kreimer Hopf algebra}, Annales Sc. de l’Ecole Normale Sup.
\textbf{50},
no.~1, (2017), 39-83.
\newblock 
\burlalt{doi:10.24033/asens.2315}{https://smf.emath.fr/publications/les-transformations-darborification-coarborification-decalle-et-lalgebre-de-hopf-de}.

	\bibitem{Mate19}
{\rm M.~{Gerencs{\'e}r}}.
\newblock {\em Nondivergence form quasilinear heat equations driven by space-time
	white noise}.
\newblock Annales de l'Institut Henri Poincar\'{e}, Analyse non lin\'{e}aire \textbf{37}, no.~3. (2020), 663--682
\newblock \burlalt{doi:10.1016/j.anihpc.2020.01.003}{https://doi.org/10.1016/j.anihpc.2020.01.003}.

 \bibitem{GKO13}
{\rm Z.~Guo, S.~Kwon, T.~Oh}.
\newblock \emph{Poincaré-Dulac normal form
	reduction for unconditional well-posedness of the periodic cubic NLS}.
\newblock Comm. Math. Phys. \textbf{322}, no. 1, (2013), 19–48.
\newblock
\burlalt{doi:10.1007/s00220-013-1755-5}{https://dx.doi.org/10.1007/s00220-013-1755-5}.

 \bibitem{Gub04}
{\rm M.~Gubinelli}.
\newblock \emph{Controlling rough paths}.
\newblock J. Funct. Anal. \textbf{216}, no.~1, (2004), 86
-- 140.
\newblock
\burlalt{doi:10.1016/j.jfa.2004.01.002}{https://dx.doi.org/10.1016/j.jfa.2004.01.002}.

\bibitem{Gub06}
{\rm M.~Gubinelli}.
\newblock \emph{Ramification of rough paths}.
\newblock J. Differ. Equ. \textbf{248}, no.~4, (2010), 693 -- 721.
\burlalt{doi:10.1016/j.jde.2009.11.015}{https://dx.doi.org/10.1016/j.jde.2009.11.015}.





\bibitem{KPZ}
{\rm M. Hairer}.
\newblock {\em Solving the KPZ equation.}
\newblock Ann. Math. \textbf{178}, no.~2, (2013), 559--664.
\newblock
\burlalt{doi:10.4007/annals.2013.178.2.4}{https://dx.doi.org/10.4007/annals.2013.178.2.4}. 

 \bibitem{reg}
{\rm M. Hairer}.
\newblock {\em A theory of regularity structures.}
\newblock Invent. Math. \textbf{198}, no.~2, (2014), 269--504.
\newblock
\burlalt{doi:10.1007/s00222-014-0505-4}{https://dx.doi.org/10.1007/s00222-014-0505-4}. 

\bibitem{HK15}
{\rm M.~Hairer, D.~Kelly}.
\newblock \emph{Geometric versus non-geometric rough paths}.
\newblock Ann. Inst. H. Poincaré Probab. Statist. \textbf{51}, no.~1,
(2015), 207--251.
\newblock
\burlalt{doi:10.1214/13-AIHP564}{https://dx.doi.org/10.1214/13-AIHP564}.

\bibitem{HP15}
{\rm M.~Hairer, E.~Pardoux}.
\newblock \emph{A Wong-Zakai theorem for stochastic PDEs}.
\newblock J. Math. Soc. Japan, \textbf{67}, no.~4,
(2015), 1551--1604.
\newblock
\burlalt{doi:10.2969/jmsj/06741551}{https://dx.doi.org/10.2969/jmsj/06741551}.


\bibitem{HQ18}
M.~{Hairer}, J.~{Quastel}. \newblock { \em A class of growth models rescaling to KPZ}. Forum of Mathematics, Pi, 6, (2018), E3. 
\newblock \burlalt{doi:10.1017/fmp.2018.2}{https://doi.org/10.1017/fmp.2018.2}.



\bibitem{Lyons98}
{\rm T.~J. Lyons}.
\newblock \emph{Differential equations driven by rough signals}.
\newblock Rev. Mat. Iberoamericana \textbf{14}, no.~2, (1998), 215--310.
\newblock \burlalt{doi:10.4171/RMI/240}{https://dx.doi.org/10.4171/RMI/240}.

\bibitem{Murua2006}
{\rm A.~Murua}.
\newblock \emph{The Hopf Algebra of Rooted Trees, Free Lie Algebras, and Lie Series}.
\newblock Found. Comput. Math. \textbf{17}, no.~6,
(2006), 387–426.
\newblock
\burlalt{doi:10.1007/s10208-003-0111-0}{https://dx.doi.org/10.1007/s10208-003-0111-0}.








 

\end{thebibliography}
\end{document}